\documentclass[12pt]{amsart}

\usepackage{amssymb, a4wide, mathdots,url, amsmath}
\usepackage{verbatim}
\usepackage{enumerate}
\usepackage[usenames]{color}
\usepackage[all]{xy}
\usepackage[colorlinks,pagebackref=true]{hyperref}
\hypersetup{colorlinks,linkcolor={blue},citecolor={blue},urlcolor={red}} 
\usepackage{amstext} 
\usepackage{array} 

\newcolumntype{L}{>{$}l<{$}} 

\providecommand{\germ}{\mathfrak}

\newcommand{\Gal}{{\rm Gal}}

\newcommand{\Sp}{\mathrm{Sp}}

\def\BC{\mathbf{BC}}

\newcommand{\Tr}{\mathrm{Tr}}

\renewcommand{\d}{\delta}

\newcommand{\sn}{\mathrm{sn}}
\newcommand{\wsn}{\mathrm{wsn}}

\newcommand{\odd}{\mathrm{o}}

\def\Res{{\rm Res}}
\def\GL{{\rm GL}}

\def\mult{\mathfrak m}

\newcommand{\C}{\mathbb{C}}
\newcommand{\Z}{\mathbb{Z}}

\newcommand{\R}{\mathbb{R}}
\newcommand{\N}{\mathbb{N}}

\newcommand{\A}{\mathbb{A}}
\newcommand{\D}{\mathcal{D}}
\newcommand{\V}{\mathcal{V}}
\newcommand{\M}{\mathcal{M}}
\newcommand{\UU}{\mathcal{U}}
\newcommand{\bG}{\mathbb{G}}
\newcommand{\G}{\mathcal{G}}

\newcommand{\Y}{{\bf Y}}
\newcommand{\T}{{\bf T}}

\renewcommand{\j}{{\jmath}}

\newcommand{\Ad}{\operatorname{Ad}}

\newcommand{\Hasse}{\operatorname{Hasse}}

\renewcommand{\Re}{\operatorname{Re}}
\newcommand{\Ind}{{\rm Ind}}

\renewcommand{\subset}{\subseteq}

\newcommand{\bs}{\backslash}

\newcommand{\diag}{\operatorname{diag}}

\newcommand{\SO}{\operatorname{SO}}
\renewcommand{\O}{\operatorname{O}}
\newcommand{\SL}{\operatorname{SL}}

\newcommand{\U}{{\operatorname{U}}}
\newcommand{\PGL}{\operatorname{PGL}}
\newcommand{\PU}{\operatorname{PU}}

\newcommand{\SU}{\operatorname{SU}}

\newcommand{\Spin}{\operatorname{Spin}}
\newcommand{\GSpin}{\operatorname{GSpin}}

\renewcommand{\mult}{\mathbb{G}_m}

\newcommand{\abs}[1]{\left|{#1}\right|}
\newcommand{\aaa}{\mathfrak{a}}

\newcommand{\op}{\mathrm{op}}

\newcommand{\triv}{{\bf 1}}

\newcommand{\inj}{\iota}

\newcommand{\sprod}[2]{\left\langle#1,#2\right\rangle}

\newcommand{\set}{{\mathfrak{c}}}
\newcommand{\Out}{\mathrm{Out}}

\newcommand{\weyl}{{\mathfrak W}}

\newcommand{\coh}{\mathrm{H}^1}

\newcommand{\Hom}{\operatorname{Hom}}

\renewcommand{\mod}{{\operatorname{mod}}}

\newcommand{\ad}{\operatorname{ad}}
\renewcommand{\sc}{\operatorname{sc}}

\newcommand{\sgn}{\operatorname{\varepsilon}}

\newcommand{\sm}[4]{{\bigl(\begin{smallmatrix}{#1}&{#2}\\{#3}&{#4}
\end{smallmatrix}\bigr)}}

\def\Aut{\operatorname{Aut}}

\newtheorem{theorem}{Theorem}
\newtheorem*{theorem*}{Theorem}

\newtheorem{lemma}{Lemma}
\newtheorem{conjecture}{Conjecture}

\newtheorem{proposition}{Proposition}
\newtheorem{corollary}{Corollary}

\theoremstyle{definition}

\theoremstyle{remark}
\newtheorem{remark}{Remark}
\theoremstyle{observation}
\newtheorem{observation}{Observation}

\title{Intertwining periods and distinction for p-adic Galois symmetric pairs}

\author{Nadir Matringe}
\email{nadirmatringe@outlook.fr}
\address{Universit\'{e} de Paris, Paris, France}
\author{Omer Offen}
\email{offen@brandeis.edu}
\address{Brandeis University, Waltham, MA}

\date{\today}

\begin{document}
\maketitle 

\begin{abstract}
We consider distinction of representations in the context of $p$-adic Galois symmetric spaces. We provide new sufficient conditions for distinction of parabolically induced representations in terms of similar conditions on the inducing data and deduce a characterization for distinction of representations parabolically induced from cuspidal. We explicate the results further for classical groups and give several applications, in particular, concerning the preservation of distinction via Langlands functoriality. We relate our results with a conjecture of Dipendra Prasad.
\end{abstract}
\tableofcontents

\section{Introduction}

Let $E/F$ be a quadratic extension of $p$-adic fields with Galois involution $\theta$, $\mathbb{H}$ a connected reductive group defined over $F$, and $\mathbb{G}=\Res_{E/F}(\mathbb{H})$. Set $G=\mathbb{G}(F)$ and $H=\mathbb{H}(F)$. In this paper we provide a step in reducing the study of $H$-distinction of irreducible representations of $G$ to that of cuspidal representations. 

Recall that a smooth representation $\pi$ of $G$ is called $H$-distinguished if it admits a nonzero $H$-invariant linear form. Also, any irreducible representation $\pi$ of $G$ is the quotient of an induced representation 
of the form $I_P^G(\sigma)$ for $\sigma$ a cuspidal representation of a Levi part of the parabolic subgroup $P$ of $G$. In particular if $\pi$ is $H$-distinguished then so is $I_P^G(\sigma)$. One of our main results is a necessary and sufficient condition on $\sigma$ for $I_P^G(\sigma)$ to be distinguished, it is the content of Corollary \ref{cor cusp dist}.

Of course one is far from understanding when $\pi$ is distinguished from this result, and to go further one needs to understand  when a nonzero $H$-invariant linear form $L$ on $I_P^G(\sigma)$ descends to $\pi$. 

One of the motivations of the present work is the paper \cite{DP} that makes precise predictions on the Langlands parameter of an irreducible representation of $G$ that is $H$-distinguished, and on the multiplicity of $H$-invariant linear forms on $\pi$ in terms of such a parameter. In fact if we want to 
state results in terms of Langlands parameters, one has to go via more steps. More precisely, $\pi$ is the unique irreducible quotient of a standard module $I_P^G(\tau)$ thanks to the Langlands quotient theorem. Here $\tau$ is an essentially tempered representation and therefore a summand of a representation induced from an essentially discrete series representation. Finally, one studies distinction of discrete series representations via distinction of their cuspidal support. At each step one would like a necessary and sufficient condition on the inducing datum of the induced representation considered to be $H$-distinguished, and an understanding of when an $H$-invariant linear form on the induced representation descends to a given irreducible quotient. 

Necessary conditions for distinction of an induced representation of $G$ in terms of distinction of the inducing datum are usually obtained by means of the Berstein-Zelevinsky geometric lemma, and this technique is nowadays very standard. In fact in the case of representations induced from a cuspidal representation the necessary condition obtained from the geometric lemma becomes sufficient, as we show in Corollary \ref{cor cusp dist}. 

To show that the necessary condition given by the geometric lemma is sufficient is not as straightforward and we make use of local intertwining periods which in fact allow us to obtain sufficient conditions for induced representations to be distinguished in a more general context, see Proposition \ref{prop lift}. This can be seen as the most basic application of intertwining periods to the theory. More generally to understand when a nonzero $H$-invariant linear form on an induced representation descends to an irreducible quotient is a notoriously difficult problem in the field. One way to tackle it is to study local 
intertwining periods on induced representations and their functional equation (see Remark \ref{rem descent of linear form}). We expect our result to have further applications concerning distinction of generic discrete series of classical groups.

Intertwining periods were introduced in the global setting in \cite{MR1324453} and further studied in \cite{MR1625060} and 
\cite{MR2010737}. Informally, suppose for a moment that $F$ is a number field with ring of adeles  
$\A$. The global intertwining periods are meromorphic families of $\mathbb{H}(\A)$-invariant linear forms on spaces obtained by parabolic induction from automorphic representations of Levi subgroups of $\mathbb{G}(\A)$. The linear forms are defined by (the meromorphic continuation of) certain integrals. They arise naturally in the study of period integrals of Eisenstein series and consequently appear in the spectral expansion of the relative trace formula. They have also been considered for non-Galois symmetric pairs (see e.g. \cite{MR2254544}  where they were used to study the $\Sp_{2n}(\A)$-period integral on the residual spectrum 
of $\GL_{2n}(\A)$). 

Here, building on \cite{MR3541705}, we develop in Section \ref{sec int per} their local theory for Galois pairs, in analogy with the global work \cite{MR2010737}. Returning to the local setting where $F$ is $p$-adic, our main result on local intertwining periods is their convergence in a cone and meromorphic continuation. More precisely, let $P$ be a standard parabolic subgroup of $G$ (with respect to a fixed minimal parabolic subgroup and a fixed $\theta$-stable maximal $F$-split torus in it) and $M$ the standard Levi subgroup of $P$. Let $x\in G$ be such that $\theta(x)=x^{-1}$ and assume further that the involution $\theta_x=\Ad(x)\circ \theta$ of $G$ stabilizes $M$. Denote by $S_x$ the  $\theta_x$-fixed points in a subset $S$ of $G$.

For a finite length (smooth complex valued) representation $\sigma$ of $M$ with bounded matrix coefficients, $\ell\in \Hom_{M_x}(\sigma,\triv_{M_x})$ and $\chi$ an unramified character of $M$ anti-invariant under $\theta_x$, we show that for $\chi$ in a certain sufficiently positive cone, the integral 
\[\int_{P_x\backslash G_x} \ell(\varphi_\chi(h))dh\] converges absolutely for all holomorphic sections 
$\varphi_\chi\in I_P^G(\chi\otimes \sigma)$ and admits a meromorphic continuation with respect to $\chi$. When $\ell$ is nonzero, by taking a leading term of such an intertwining period we prove in Proposition \ref{prop lift} the existence of a nonzero $G_x$-invariant linear form on 
$I_P^G(\sigma)$.  The existence of such $\ell$ is our new sufficient condition for distinction. 
It immediately leads to our characterization of distinction when $\sigma$ is cuspidal (Corollary \ref{cor cusp dist}). 

The Levi subgroup $M$ acts naturally by twisted conjugation on the set of all $x$ as above and the condition on a cuspidal $\sigma$ in Corollary \ref{cor cusp dist} is expressed in terms of the $M$-orbits. The formulation is not always  convenient to use in view of the motivation, obtaining results in terms of Langlands parametrization. Hence in Section \ref{section Galois distinction and induction}, we make our criterion completely explicit when $\mathbb{H}$ is a classical group. This is the content of Theorem \ref{thm dist ind class}. In fact, Theorem \ref{thm Hdist ind class} is a more precise version of  Theorem \ref{thm dist ind class} where some stabilizers are exlpicated. Its proof requires a careful analysis of orbits and stabilizers that is carried out in Section \ref{s main pf}. In particular, for this task we obtain in \S\ref{ss orbit classification} the classification of the $G$-orbits in the symmetric space $X=\{x\in G: \theta(x)=x^{-1}\}$ with respect to twisted conjugation $g\cdot x=gx\theta(g)^{-1}$.

In Sections \ref{sec selfdual} and \ref{s llc} we present some applications of Theorem \ref{thm dist ind class}. We emphasize that these are applications of the necessary condition in Theorem \ref{thm dist ind class}, in particular, they do not rely on the theory of local intertwining periods developed in Section \ref{sec int per}.

Let $K/k$ be a quadratic extension of number fields with Galois involution still denoted by $\theta$. We prove in Theorem \ref{thm global} that an irreducible generic cuspidal automorphic representation of the split odd special orthogonal group $\SO_{2n+1}(\A_K)$ which is $\SO_{2n+1}(\A_k)$-distinguished (or more generally, distinguished by a certain $K/k$-inner form) is automatically $\theta$-conjugate self-dual. The rigidity theorem of \cite{MR1983781} is a key ingredient in our proof. From this and a globalization result of Beuzart-Plessis \cite{BP} we deduce a local analogue of this result in Theorem \ref{thm local sodd} that slightly generalizes, for p-adic $\SO_{2n+1}$, a Theorem in \cite{BP} for general p-adic quasi-split groups. We also deduce in Proposition \ref{prop standinv} similar results for some non-generic representations.

In Section \ref{s llc} we study a relation between Galois distinction for representations of quasi-split classical groups and their transfer to a general linear group. 
As explained by Proposition \ref{prop distT} this relation is a consequence of Dipendra Prasad's Conjecture \ref{conjecture Prasad}. We prove this relation in Theorem \ref{thm dist and trans} for a certain class of generic irreducible representations of the classical group, that contains all irreducible generic principal series. Such results (and more) have been established in \cite{MR3991419} and \cite{MR4223066} for some small rank classical groups.

Finally in Appendix \ref{app} we provide explicit formulas for Prasad's quadratic character for classical groups. Most of these formulas are well-known but we give proofs for convenience of the reader. We also establish further properties of this character.

\section*{Acknowledgement} We warmly thank Raphaël Beuzart-Plessis for his help in the writing of Section \ref{subsec prasad char} of Appendix \ref{app} and for sharing his preprint \cite{BP}, and Ahmed Moussaoui for useful conversations and explanations. We also thank Colette Moeglin for useful correspondence and the referee for a very careful reading of this work and many useful comments.
The first named author thanks the CNRS for giving him a ``délégation" in 2020.
The second named author was partially supported by Simons collaboration award number $709678$.
\section{Notation and preliminaries}

Let $F$ be a non-archimedean local field of characteristic zero and $\abs{\cdot}$ the normalized absolute value on $F$. For an affine algebraic variety $\mathbb{X}$ defined over $F$ set $X=\mathbb{X}(F)$. Let $\mathbb{G}$ be a connected reductive group defined over $F$. We denote by $X^*(G)$ the group of $F$-rational characters from $\mathbb{G}$ to the multiplicative group $\mathbb{G}_m$ and let $\aaa_G^*=X^*(G)\otimes_\Z \R$. Let $\aaa_G=\Hom(\aaa_G^*,\R)$ be the dual vector space and let $H_G:G\rightarrow \aaa_G$ be the map defined by
\[
e^{\sprod{\chi}{H_G(g)}}=\abs{\chi(g)},\ \ \ \chi\in X^*(G),\ g\in G.
\]
For $\aaa$ a real vector space we denote by $\aaa_\C=\aaa\otimes_\R\C$ its compexification. 

Throughout the paper, we mostly avoid mention of the underlying algebraic groups defined over $F$ and consider only their groups of $F$-points. 
When we say that $P$ is a parabolic subgroup of $G$ we mean that $P=\mathbb{P}(F)$ where $\mathbb{P}$ is a parabolic subgroup of $\mathbb{G}$ defined over $F$. We use a similar convention for other algebraic properties such as, (split) tori, unipotent subgroups etc.

Let $A_G$ be the split component of $G$, that is, the maximal split torus in the center of $G$. We recall that restriction to $A_G$ identifies $\aaa_G^*$ with $\aaa_{A_G}^*$.

For groups $A\subset B$ denote by $Z_B(A)$ the centralizer of $A$ in $B$ and by $N_B(A)$ the normalizer of $A$ in $B$. 
For a locally compact group $Q$ denote by $\delta_Q$ its modulus function.

\subsection{Parabolic subgroups and Iwasawa decomposition}

Let $P_0$ be a minimal parabolic subgroup of $G$ and $T$ a maximal $F$-split torus of $G$ contained in $P_0$. 
The group $P_0$ admits a Levi decomposition $P_0=M_0\ltimes U_0$ where $M_0=Z_G(T)$ and $U_0$ is the unipotent radical of $P_0$. Let $\aaa_0=\aaa_T=\aaa_{M_0}$ and $\aaa_0^*=\aaa_T^*=\aaa_{M_0}^*$.

A parabolic subgroup $P$ of $G$ is called semi-standard if it contains $T$, and standard if it contains $P_0$. A semi-standard parabolic subgroup $P$ of $G$ with  unipotent radical $U$ admits a unique Levi decomposition $P=M\ltimes U$ with the Levi subgroup $M$ containing $T$. Such a Levi subgroup is called semi-standard. If in addition $P$ is standard, $M$ is called a standard Levi subgroup. Whenever we write $P=M\ltimes U$ is a (semi-)standard parabolic subgroup we implicitly assume the above (semi-)standard Levi decomposition. 

Fix a maximal compact open subgroup $K$ of $G$ in good position with respect to $P_0$ (see \cite[I.1.4]{MR1361168}). For a standard parabolic subgroup $P=M\ltimes U$ of $G$ this allows one to extend 
 $H_M$ to a function on $G=UMK$ by 
 \[
 H_M(umk)=H_M(m),\ \ \ u\in U,\,m\in M,\,k\in K.
 \]

\subsection{Roots and Weyl groups}

\subsubsection{Roots}\label{s roots}

Denote by $R(T,G)$ the root system of $G$ with respect to $T$, by $R(T,P_0)$ the set of positive roots of $R(T,G)$ with respect to $P_0$, and by $\Delta_0$ the corresponding set of simple roots. Note that $R(T,G)$ lies in $\aaa_0^*$. For every $\alpha\in R(T,G)$ we denote by $\alpha^\vee\in \aaa_0$ the corresponding coroot. 

There is a unique element $\rho_0\in \aaa_0^*$ (half the sum of positive roots with multiplicities) such that 
\[
\delta_{P_0}(p)=e^{\sprod{2\rho_0}{H_{M_0}(p)}},\ \ \ p\in P_0.
\]

For a standard parabolic subgroup $P=M\ltimes U$ of $G$ let $R(A_M,G)$ be the set of non-trivial restrictions to $A_M$ of elements of $R(T,G)$, $R(A_M,P)$ the subset of non-trivial restrictions to $A_M$ of elements of $R(T,P_0)$ and $\Delta_P$ the subset of non-zero restrictions to $A_M$ of elements in $\Delta_0$. 
For $\alpha\in R(A_M,G)$ we write $\alpha>0$ if $\alpha\in R(A_M,P)$ and $\alpha<0$ otherwise.

Let $Q=L\ltimes V$ be a parabolic subgroup of $G$ containing $P$. 
The restriction map from $A_M$ to $A_L$ defines a projection $\aaa_M^*\rightarrow \aaa_L^*$ that gives rise to a direct sum decomposition $\aaa_M^*=\aaa_L^*  \oplus (\aaa_M^L)^*$ (the second component is the kernel of this projection) and a compatible decomposition $\aaa_M=\aaa_L  \oplus \aaa_M^L$ for the dual spaces.  For $\lambda\in \aaa_M^*$ (respectively $\nu\in\aaa_M$) we write $\lambda=\lambda_L+\lambda_M^L$ (respectively, $\nu=\nu_L+\nu_M^L$) for the corresponding decomposition.
Set
\[
\Delta_P^Q=\{\alpha\in \Delta_P,\ \alpha_L=0 \}.
\] 
(Note that $\Delta_P^Q$ identifies with $\Delta_{P\cap L}$ defined with respect to $(L,L\cap P_0)$ replacing $(G,P_0)$.)

The coroot $\alpha^\vee\in\aaa_M$ associated to $\alpha\in R(A_M,G)$ is defined as follows: let $\alpha_0\in R(T,G)$ be such that $\alpha=(\alpha_0)_M$ then one sets $\alpha^\vee=(\alpha_0^\vee)_M$ (it is independent of the choice of $\alpha_0$).

Let $\rho_P=(\rho_0)_M$. We have the relation
\[
\delta_P(p)=e^{\sprod{2\rho_P}{H_M(p)}},\ \ \ p\in P.
\]

\subsubsection{Weyl groups and Bruhat decomposition}

 We denote by  $W_G:=N_G(T)/M_0$ the Weyl group of $G$.  
 Note that for a semi-standard parabolic subgroup $P=M\ltimes U$ of $G$ the Levi part $M$ contains $M_0$. Therefore expressions such as $wP$ and $Pw$ are well defined for $w\in W_G$. Furthermore, $W_M$ is a subgroup of $W_G$. 
 Let $P=M\ltimes U$ and $Q=L\ltimes V$ be semi-standard parabolic subgroups of $G$.
 The Bruhat decomposition is the bijection $W_MwW_L\mapsto PwQ$ from $W_M\backslash W_G/ W_L$ to $P\backslash G / Q$. 

\subsubsection{Elementary symmetries}\label{s elementary sym}

For standard Levi subgroups $M\subseteq L$ of $G$ let $W_L(M)$ be the set of elements $w\in W_L$ such that $w$ is of minimal length in 
$wW_M$ and $wMw^{-1}$ is a standard Levi subgroup of $G$. Note that for $w\in W_L(M)$ and $w'\in W_L(wMw^{-1})$ we have $w'w\in W_L(M)$. By definition $W_L(M)$ is the disjoint union over standard Levi subgroups $M'$ of $G$ of the sets 
\[
W_L(M,M')=\{w\in W_L(M): wMw^{-1}=M'\}.
\]

Set $W(M)=W_G(M)$. In \cite[I.1.7, I.1.8]{MR1361168} the elementary symmetries $s_\alpha\in W(M)$ attached to each $\alpha\in \Delta_P$ are introduced and used in order to define a length function $\ell_M$ on $W(M)$. There is a unique element of $W_L(M)$ of maximal length for $\ell_M$ and we denote it by $w_M^L$. 

\subsection{Symmetric spaces}

By a symmetric pair $(G,\theta)$ we mean that $\theta$ is an involution of $\mathbb{G}$ defined over $F$.
We say that the symmetric pair $(G,\theta)$ is a Galois pair if $\mathbb{G}=\Res_{E/F}\mathbb{H}$ is the Weil restriction of scalars of a connected reductive group $\mathbb{H}$ defined over $F$, $E/F$ is a quadratic field extension and $\theta$ is the associated Galois action. 
Thus, for a Galois pair $(G,\theta)$ we have $G=\mathbb{H}(E)$ and $G^\theta=\mathbb{H}(F)$.

To a symmetric pair $(G,\theta)$ we associate the symmetric space 
\[
X=\{g\in G,\ \theta(g)=g^{-1}\},
\]
equipped with the twisted conjugation $G$-action  
\[
g\cdot x= g x \theta(g)^{-1}, \ \ \ g\in G,\,x\in X.
\] By \cite[Proposition 6.15]{MR1215304}, $X$ consists of a finite number of $G$-orbits. For $x\in X$ we denote by $\theta_x$ the involution of $G$ given by 
\[\theta_x(g)=x\theta(g)x^{-1}.\] For $x\in X$ and $Q$ a subgroup of $G$, we denote by $Q_x$ the stabilizer of 
$x$ in $Q$, so that 
\[Q_x=(Q\cap \theta_x(Q))^{\theta_x}.\]

Such an involution $\theta$ will induce linear involutions on different vector spaces. Whenever $\theta$ is an involution acting linearly on a vector space $V$ over $\R$ or $\C$, we write
\[
V=V_\theta^+ \oplus V_\theta^-
\]
for the corresponding decomposition into the $\theta$-eigenspaces $V_\theta^{\pm}$ with corresponding eigenvalues $\pm 1$.

\subsection{Parabolic induction and standard intertwining operators}

By a representation of a subgroup of $G$ we always mean a smooth complex valued representation.

\subsubsection{Induced representations} Let $P=M\ltimes U$ be a standard parabolic subgroup of $G$ and $\sigma$ a representation of $M$. 
We denote by $I_P^G(\sigma)$ the representation of $G$ by right translations on the space of functions $\varphi$ on $G$ with values in the space of $\sigma$ that are right invariant by some open subgroup of $G$ and satisfy
\[
\varphi(umg)=\delta_P(m)^{1/2}\sigma(m)\varphi(g),\ \ \ u\in U,\,m\in M,\,g\in G.
\]

\subsubsection{Holomorphic sections}
For $\lambda\in \aaa_{M,\C}^*$ we denote by $\sigma[\lambda]$ the representation of $M$ on the space of $\sigma$ given by $\sigma[\lambda](m)=e^{\sprod{\lambda}{H_M(m)}}\sigma(m)$.
In order to make sense of meromorphic families of linear forms on induced representations we realize all the representations $I_P^G(\sigma[\lambda])$ for $\lambda\in \aaa_{M,\C}^*$ on the space of $I_P^G(\sigma)$.

For $\lambda\in \aaa_{M,\C}^*$ and $\varphi\in I_P^G(\sigma)$ write $\varphi_\lambda(g)=e^{\sprod{\lambda}{H_M(g)}}\varphi(g)$ and let $I_P^G(\sigma,\lambda)$ be the representation of $G$ on the space $I_P^G(\sigma)$ defined by
\[
(I_P^G(g,\sigma,\lambda)\varphi)_{\lambda}(x)=\varphi_\lambda(xg).
\]
The map $\varphi\mapsto \varphi_\lambda$ is an isomorphism of representations $I_P^G(\sigma,\lambda)\rightarrow I_P^G(\sigma[\lambda])$.

\subsubsection{Intertwining operators}\label{sss intop}

Let $w\in W(M)$ and select a representative $n$ of $w$ in $N_G(T)$. Set $M_1=wMw^{-1}=nMn^{-1}$ and let $P_1=M_1\ltimes U_1$ be the standard parabolic subgroup of $G$ with Levi subgroup $M_1$. 
For $c\in \R$ let
\[
\D^{M,w}(c)=\{\lambda \in \aaa_{M,\C}^*: \langle \Re(\lambda), \alpha^\vee \rangle >c, \ \forall \alpha \in R(A_M,G),\,\alpha>0,\, w\alpha<0\}.
\] 
We denote by $n\sigma$ the representation of $M_1$ on the space of $\sigma$ defined by $(n\sigma)(nmn^{-1})=\sigma(m)$, $m\in M$. Note that the isomorphism class of $n\sigma$ is independent of the choice of representative $n$ for $w$ and we denote it by $w\sigma$. For $\lambda\in \aaa_{M,\C}^*$ we have $w\lambda\in \aaa_{M_1,\C}^*$ and the standard intertwining operator 
\[M(n,\sigma,\lambda):I_P^G(\sigma,\lambda)\rightarrow I_{P_1}^G(w\sigma,w\lambda)\] is the meromorphic continuation of 
the operator given by the following convergent integral for $\lambda\in \D^{M,w}(c_{\sigma})$ for some constant $c_{\sigma}$
\[
(M(n,\sigma,\lambda)\varphi)_{w\lambda}(g)= \int_{U_1\cap wUw^{-1}\backslash U_1}\varphi_\lambda (n^{-1}ug)\ du. 
\]
By definition, the convergence of the integral in the appropriate domain means that for every element $v^\vee$ in the smooth dual of $\sigma$ the scalar valued integral
\[
\int_{U_1\cap wUw^{-1}\backslash U_1} v^\vee[\varphi_\lambda (n^{-1}ug)]\ du
\]
converges.
We record the following simple consequence.
\begin{lemma}\label{lem conv l}
With the above notation, let $\lambda\in \D^{M,w}(c_\sigma)$ and $\ell$ a linear form on the space of $\sigma$ such that for every $\varphi\in I_P^G(\sigma)$ and $g\in G$ the integral
\[
\int_{U_1\cap wUw^{-1}\backslash U_1} \abs{\ell(\varphi_\lambda (n^{-1}ug))} du
\]
converges. Then 
\[
\ell((M(n,\sigma,\lambda)\varphi)_{w\lambda}(g))=\int_{U_1\cap wUw^{-1}\backslash U_1} \ell(\varphi_\lambda (n^{-1}ug)) du.
\]
\end{lemma}
\begin{proof}
Replacing $\varphi$ with $I_P^G(g,\sigma,\lambda)\varphi$ we assume without loss of generality that $g$ is the identity $e$ of $G$. Let $\M_1$ be a compact open subgroup of $M_1$ such that $\varphi$ is right $\M_1$-invariant and $\M=n^{-1}\M_1n$ a compact open subgroup of $M$. Let
\[
\ell_0(v)=\int_{\M}\ell(\sigma(m)v) \ dm
\] 
for every $v$ in the space of $\sigma$ where integration is over the probability measure of $\M$. Then $\ell_0$ lies in the smooth dual of $\sigma$. 
For $m\in \M$ we have
\[
\sigma(m)[(M(n,\sigma,\lambda)\varphi)_{w\lambda}(e)]=(M(n,\sigma,\lambda)\varphi)_{w\lambda}(nmn^{-1})=(M(n,\sigma,\lambda)\varphi)_{w\lambda}(e)
\]
and therefore
\[
\ell((M(n,\sigma,\lambda)\varphi)_{w\lambda}(e))=\ell_0((M(n,\sigma,\lambda)\varphi)_{w\lambda}(e))=\int_{U_1\cap wUw^{-1}\backslash U_1} \ell_0(\varphi_\lambda (n^{-1}u))\ du.
\]
Let $(\UU_k)_{k\in \N}$ be an increasing sequence of compact open subgroups of $U_1$ with union $U_1$ and such that $\M_1$ normalizes each $\UU_k$. (It is easy to see that such a sequence exists.) 
By the dominated convergence theorem we have
\[
\int_{U_1\cap wUw^{-1}\backslash U_1} \ell_0(\varphi_\lambda (n^{-1}u)) \ du=\lim_{k\to \infty} \int_{\UU_k\cap wUw^{-1}\backslash \UU_k} \ell_0(\varphi_\lambda (n^{-1}u)) \ du.
\]
Note that for every $k$ we have
\[
\int_{\UU_k\cap wUw^{-1}\backslash \UU_k} \ell_0(\varphi_\lambda (n^{-1}u))\ du=\int_{\UU_k\cap wUw^{-1}\backslash \UU_k} \int_{\M}\ell(\sigma(m)[\varphi_\lambda (n^{-1}u))])\ dm\ du
\]
\[
=\int_{\UU_k\cap wUw^{-1}\backslash \UU_k} \int_{\M}\ell(\varphi_\lambda (mn^{-1}u))\ dm\ du.
\]
After the change of variables $m\mapsto n^{-1}mn$ this becomes
\[
\int_{\UU_k\cap wUw^{-1}\backslash \UU_k} \int_{\M_1}\ell(\varphi_\lambda (n^{-1}mu))\ dm\ du.
\]
Changing order of integration, making the change of variables $u\mapsto m^{-1} u m$ and applying the right $\M_1$-invariance of $\varphi$ this becomes
\[
\int_{\UU_k\cap wUw^{-1}\backslash \UU_k} \ell(\varphi_\lambda (n^{-1}u))\ du.
\]
Applying the dominated convergence theorem again we have that
\[
\lim_{k\to \infty}\int_{\UU_k\cap wUw^{-1}\backslash \UU_k} \ell(\varphi_\lambda (n^{-1}u))\ du=\int_{U_1\cap wUw^{-1}\backslash U_1} \ell(\varphi_\lambda (n^{-1}u))\ du
\]
and the lemma follows.
\end{proof}

\subsection{Meromorphic families of linear forms} 
For a finite dimensional complex vector space $V$ and a complex vector space $\V$ we say that $L(\lambda)_{\lambda\in V}$ is a meromorphic family of linear forms on $\V$ if there is a non-zero Laurent polynomial $f$ in $\dim V$-variables such that $f(q^\lambda)L(\lambda)$ is a linear form on $\V$ for every $\lambda\in V$ and the map \[\lambda\mapsto f(q^\lambda)L(\lambda)\nu\] from $V$ to $\C$ is holomorphic for every $\nu\in \V$. Here $q$ is the size of the residual field of $F$ and $q^\lambda=(q^{\lambda_1},\dots,q^{\lambda_r})$ for some choice of coordinates $\lambda_i$ of $\lambda$.

\subsection{The result of Blanc and Delorme}\label{s bd}

For $(G,\theta)$ a symmetric pair, let $P=M\ltimes U$ be a parabolic subgroup of $G$ such that $P\cap \theta(P)=M$. Note that in this case $P^\theta=M^\theta$ is unimodular. Furthermore, since $M$ is $\theta$-stable, $\theta$ acts as a linear involution on $\aaa_M^*$ so that \[\aaa_M^*=(\aaa_M^*)_\theta^+ \oplus (\aaa_M^*)_\theta^-.\]  There is a similar decomposition to the dual space $\aaa_M$ so that $(\aaa_M)_\theta^{\pm}$ is dual to $(\aaa_M^*)_\theta^{\pm}$. In particular, we have
\[
\sprod{\lambda}{H_M(m)}=0, \ \ \ \lambda\in  (\aaa_M^*)_\theta^-,\,m\in M^\theta.
\]
For $c>0$ let
\[
\D_{M,\theta}(c)=\{\lambda\in (\aaa_M^*)_\theta^-:\sprod{\lambda}{\alpha^\vee}>c, \forall \alpha\in R(A_M,P)\}.
\]

 The following is a consequence  of \cite{MR2401221}. Let $\sigma$ be a representation of $M$ of finite length. There exists $c>0$ such that for all $\varphi\in I_P^G(\sigma)$, $\ell\in \Hom_{M^\theta}(\sigma,\triv_{M^\theta})$ and $\lambda\in \D_{M,\theta}(c)$ the integral 
 \[
 J_P^G(\varphi;\theta,\ell,\sigma,\lambda) =\int_{M^\theta\bs G^\theta} \ell(\varphi_\lambda(g))\ dg
 \]
 is absolutely convergent. Furthermore, it admits a meromorphic continuation to a meromorphic family of linear forms  $J_P^G(\theta,\ell,\sigma,\lambda)\in \Hom_{G^\theta}(I_P^G(\sigma,\lambda),\triv_{G^\theta})$, $\lambda\in (\aaa_{M,\C}^*)_\theta^-$.
 
\section{Intertwining periods}\label{sec int per}

Our goal in this section is to construct certain meromorphic families of invariant linear forms, local intertwining periods, on induced representations associated to a p-adic Galois pair. In the global setting, intertwining periods were introduced in \cite{MR1324453} and studied further in a more general setting in \cite{MR1625060} and \cite{MR2010737}. 

\subsection{An informal introduction}
The local intertwining periods emerge as follows. Consider a symmetric pair $(G,\theta)$. Let $\chi$ be a character of $G^\theta$. For a representation $\pi$ of $G$ we denote by $\Hom_{G^\theta}(\pi,\chi)$ the space of $(G^\theta,\chi)$-equivariant linear forms on the space of $\pi$. 

Consider a parabolic subgroup $P$ of $G$ with a $\theta$-stable Levi subgroup $M$ and let $U$ be the unipotent radical of $P$ so that $P^\theta=M^\theta\ltimes U^\theta$. Let $\sigma$ be a representation of $M$ and $\ell\in \Hom_{M^\theta}(\sigma, \delta_{P^\theta}\delta_P^{-1/2}\chi)$. For $\varphi\in I_P^G(\sigma)$, the map $g\mapsto \chi(g)^{-1}\ell(\varphi(g))$ is left $(P^\theta,\delta_{P^\theta})$-equivariant and therefore at least formally the integral  
\[
L(\varphi)=\int_{P^\theta\bs G^\theta} \chi(g)^{-1}\ell(\varphi(g))\ dg
\]
makes sense. If it converges it defines a linear form $L\in \Hom_{G^\theta}(I_P^G(\sigma),\chi)$. In general, however, the integral fails to converge. 
By considering unramified twists of $\sigma$ one hopes to define the linear form by a convergent integral on some cone of unramified twists and by meromorphic continuation in general. 

In this section we achieve this goal when $(G,\theta)$ is a Galois pair and $\chi$ is the trivial character.

\subsection{The strategy}  
Our proof of convergence in a cone and meromorphic continuation is based on the study of the geometry and combinatorics of the $P$-orbits on $G/G^\theta$ carried out in \cite[\S 3-4]{MR2010737}. It allows us to study the intertwining periods inductively along a certain directed graph associated to the $P$-orbits. Roughly speaking, vertices in the graph are associated to certain $P$-orbits and edges to a relation between them defined via twisted conjugation. For a directed edge one expresses the intertwining period associated with the origin as a double integral that formally corresponds with the composition of the intertwining period  associated with the destination and an intertwining operator. These are the local analogues of the simple functional equations obtained in \cite[Proposition 10.1.1]{MR2010737}. They allow reduction of the meromorphic continuation from origin to destination. In order to reduce convergence of the origin intertwining period to the destination one it is required to prove that the above composition expressed as a double integral is absolutely convergent. It is done in two steps. First, we prove it for unramified representations induced from real exponents. In this case the integrands at hand are positive and it is enough to prove convergence as an iterated integral. This reduction is carried out applying the Gindikin-Karpelevich formula and convergence of the standard intertwining operators.  In the second step we use the bounds obtained by Lagier in \cite[Theorem 4]{MR2381204} in order to reduce to the first case. This inductive process along the graph allows us to reduce convergence in a cone and meromorphic continuation to certain minimal cases. By a local analogue of \cite[Lemma 5.4.1]{MR2010737}, the minimal case is reduced to the main results established by Blanc and Delorme in \cite{MR2401221}. The entire proof is a local analogue of the one carried out by Lapid and Rogawski in \cite{MR2010737}. We remark that Lapid and Rogawski assume that $G$ admits a $\theta$-stable minimal parabolic subgroup. Thanks to the analysis of $P$-orbits on $G/G^\theta$ for a general involution $\theta$ of $G$ in \cite{MR3541705} we do not need to make this assumption.

\subsection{The set up}
For the rest of this section assume that $(G,\theta)$ is a Galois symmetric pair. Fix the minimal parabolic subgroup $P_0$ of $G$. According to \cite[Lemma 2.4]{MR1215304} there exists a $\theta$-stable maximal split torus $T$ of $G$ contained in $P_0$. Thus, $P_0$ admits a Levi decomposition $P_0=M_0\ltimes U_0$ where $M_0=Z_G(T)$ is a $\theta$-stable Levi subgroup of $P_0$.

\subsection{A graph of involutions}\label{s graph}

For a standard Levi subgroup $M$ of $G$ let 
\[
X[M]=\{x\in X:\theta_x(M)=M\}.
\]
Note for example that $X[M_0]=X\cap N_G(T)$. 

We define a directed labeled graph $\G=\G(G,\theta,P_0,T)$ as follows.  
The vertices of $\G$ are the pairs $(M,x)$ where $M$ is a standard Levi subgroup of $G$  and $x\in X[M]$. 
The edges of $\G$ are given by 
\[
(M,x) \overset{n}\searrow  (M_1,x_1)
\]
if there is $\alpha\in \Delta_P$ with $-\alpha\ne \theta_x(\alpha) <0$ such that $n\in s_\alpha M$ where $s_\alpha\in W(M)$ is the elementary symmetry associated to $\alpha$, $M_1=nMn^{-1}$ and $x_1=n\cdot x$. 
Note that $(M_1)_{x_1}=n M_x n^{-1}$.

\subsection{Minimal vertices}\label{s min} Recall that the Weyl group $W_G$ acts (by conjugation) simply transitively on the set of minimal semi-standard parabolic subgroups of $G$. There is therefore a unique $\tau\in W_G$ such that $\theta(P_0)=\tau P_0 \tau^{-1}$.

Let $(M,x)$ be a vertex in $\G$ and $P$ the standard parabolic subgroup of $G$ with Levi subgroup $M$. Set $P'=\tau^{-1}\theta(P)\tau$. Then $P'= M' \ltimes U'$ is a standard parabolic subgroup of $G$ with standard Levi subgroup $M'=\tau^{-1} \theta(M)\tau=\tau^{-1}x^{-1} Mx\tau$ and unipotent radical $U'=\tau^{-1} \theta(U)\tau$. We say that $(M,x)$ is minimal if there exists a standard parabolic subgroup $Q=L\ltimes V$ of $G$ containing $P$ such that
\begin{itemize}
\item $\theta(L)=\tau L\tau^{-1}$;
\item $Mx\tau M'=Mw_{M'}^L M'$;
\item $\theta_x(\alpha)=-\alpha$, $\alpha\in \Delta_P^Q$.
\end{itemize}      
This definition coincides with  \cite[Definition 6.6]{MR3541705}. Note that the conditions $M\subseteq L$ together with $\theta(L)=\tau L\tau^{-1}$ imply that $M'\subseteq L$ and therefore $w_{M'}^L$ makes sense. 
The following is a straightforward consequence of the definition.
\begin{lemma}\label{lem min}
Let $(M,x)$ be a minimal vertex in $\G$. In the above notation $Q$, $L$ and $V$ are $\theta_x$ stable. 
\end{lemma}
\begin{proof}
Let $n\in N_G(T)$ be a representative of $\tau$. It follows from the second point of the definition that $xn\in L$ and from the first that $L=n^{-1}\theta(L) n$. Conjugating with $xn$ we conclude that $L=\theta_x(L)$. Since $n^{-1} \theta(Q)n$ is a standard parabolic subgroup of $G$ with Levi subgroup $n^{-1}\theta(L)n=L$ we conclude that $n^{-1}\theta(Q)n=Q$ and conjugation with $xn$ similarly shows that $\theta_x(Q)=Q$. Since $V$ is the unipotent radical of $Q$ it immediately follows that $\theta_x(V)=V$.
\end{proof}

As a consequence of \cite[Lemma 3.2.1 and Proposition 3.3.1]{MR2010737} (see \cite[Corollary 6.7]{MR3541705} ) we have the following.
\begin{proposition}\label{prop min}
Let $(M,x)$ be a vertex in the graph $\G$. There exists a path 
\[
(M,x) \overset{n_1}\searrow (M_1,\theta_1) \overset{n_2}\searrow\cdots\overset{n_k}\searrow (M_k,x_k)
\] 
in $\G$ such that $(M_k,x_k)$ is a minimal vertex. 
\qed
\end{proposition}

\subsection{The cone of convergence}\label{s cone}
Let $(M,x)$ be a vertex in $\G$. Recall that $\theta_x$ acts as an involution on $\aaa_M^*$ and let
\[
(\aaa_M^*)_x^-=(\aaa_M^*)_{\theta_x}^-.\]

For $c>0$ let 
\[
\D_{M,x}(c)=\{\lambda\in (\aaa_M^*)_x^-:\sprod{\lambda}{\alpha^\vee}>c, \forall \alpha\in R(A_M,G), \,\alpha>0,\,\theta_x(\alpha)<0\}.
\]
For an edge $(M,x)\overset{n}\searrow (M_1,x_1)$ in $\G$ let $\alpha\in \Delta_P$ be such that $n\in s_\alpha M$. By \cite[Lemma 5.2.1]{MR2010737}
we have 
\[
\D_{M,x}(c)=s_\alpha^{-1} \D_{M_1,x_1}(c)\cap \{\lambda\in (\aaa_M^*)_x^-:\sprod{\lambda}{\alpha^\vee}>c\}.
\]

\subsection{Admissible orbits}\label{ss adm}
Let $P=M\ltimes U$ be a standard parabolic subgroup of $G$. For $x\in X$ let $w\in W_G$ be such that $P x\theta(P)=Pw\theta(P)$. (By the Bruhat decomposition, the double coset $W_MwW_{\theta(M)}$ is uniquely determined by $x$.) 

We say that $x$ (or its $P$-orbit in $X$) is \emph{$P$-admissible} if $M=w\theta(M)w^{-1}$. If a $P$-orbit $\O$ in $X$ is $P$-admissible then it contains an element of $X[M]$. In fact $\O\mapsto \O\cap X[M]$ is a bijection from the admissible $P$-orbits in $X$ to the $M$-orbits in $X[M]$. (See \cite[Lemma 3.2]{MR3541705}).
It further follows from \cite[Corollary 6.9]{MR3541705} that 
\begin{equation}\label{eq eq mod}
\delta_P^{1/2}|_{P_x}=\delta_{P_x}, \ \ \ x\in X[M].
\end{equation}

\begin{remark}
Equation \eqref{eq eq mod} is also satisfied by symmetric pairs of Prasad and Takloo-Bighash type (see Remark \ref{rem applications to PTB pairs}). However, it is not satisfied for general symmetric pairs. For example for the symmetric pairs $(\GL_n(F),\GL_a(F)\times \GL_b(F))$ where $a+b=n$ (see for example \cite[Proposition 3.6]{MR3430877}) the character $\delta_{P_x}$ does not, in general, extend to $P$. As pointed out by the referee, determining the symmetric pairs for which Equation (\ref{eq eq mod}) is satisfied is an interesting problem. 
\end{remark}

As a consequence, for  a representation $\sigma$ of $M$, $x\in X[M]$, $\ell\in \Hom_{M_x}(\sigma,\triv_{M_x})$, $\varphi\in I_P^G(\sigma)$ and $\lambda\in (\aaa_{M,\C}^*)_x^-$ we have
\begin{equation}\label{eq triv mod}
\ell(\varphi_\lambda(pg))=\delta_{P_x}(p)\ell(\varphi_\lambda(g)),\ \ \ p\in P_x,\,g\in G.
\end{equation}

\subsection{The intertwining periods modulo convergence}
It follows from \eqref{eq triv mod} in its notation that the following integral formally makes sense
\[
J_P^G(\varphi;x,\ell,\sigma,\lambda)=\int_{P_x\bs G_x} \ell(\varphi_\lambda(g))\ dg.
\]
If $\lambda$ is such that the integral is convergent for all $\varphi\in I_P^G(\sigma)$ then it defines a linear form $J_P^G(x,\ell,\sigma,\lambda)\in \Hom_{G_x}(I_P^G(\sigma,\lambda),\triv_{G_x})$.

Assuming that $\sigma$ is of finite length and with bounded matrix coefficients we prove that there exists $c>0$ such that the integrals are absolutely convergent as long as $\Re(\lambda)\in \D_{M,x}(c)$ and admit meromorphic continuation to $\lambda\in (\aaa_{M,\C}^*)_x^-$. We will prove this by induction along a path as in Proposition \ref{prop min} in the graph $\G$.

\subsection{Intertwining periods for minimal vertices}

Our first step is the observation that the results of Blanc and Delorme (\S \ref{s bd}) imply the base of the induction. 

\begin{corollary}\label{cor bd}
Let $(M,x)$ be a minimal vertex in $\G$ and $\sigma$ a representation of $M$ of finite length. Then for any $\ell\in \Hom_{M_x}(\sigma,\triv_{M_x})$ the linear form $J_P^G(x,\ell,\sigma,\lambda)$ is defined by a convergent integral for $\lambda\in \D_{M,x}(c)$ and admits a meromorphic continuation to $\lambda\in (\aaa_{M,\C}^*)_x^-$.
\end{corollary}
\begin{proof}
We use the notation of \S \ref{s min} for a minimal vertex and integrate $J_P^G(\varphi;x,\ell,\sigma,\lambda)$ in stages.
We have
\[
J_P^G(\varphi;x,\ell,\sigma,\lambda)=\int_{P_x\bs G_x} \ell(\varphi_\lambda(g))\ dg=\int_{Q_x\bs G_x} \int_{P_x\bs Q_x}\delta_{Q_x}(m)^{-1} \ell(\varphi_\lambda(mg))\ dm  \ dg.
\]
Recall that $Q$, $L$ and $V$ are $\theta_x$-stable (Lemma \ref{lem min}). It follows that $Q_x\bs G_x$ is compact (see \cite[Lemma 3.1]{MR3490774}). By smoothness, the outer integral may therefore be replaced by a finite sum and it is enough to prove convergence and meromorphic continuation of the inner integral. 

It further follows from  \cite[Lemma 6.3]{MR3541705} that $Q_x=L_x\ltimes V_x$ and $P_x=M_x\ltimes U_x$. Since $U=(U\cap L)\ltimes V$ we conclude that $U_x=(U_x\cap L)\ltimes V_x$.  By the definition of minimality $U_x\cap L=\{e\}$ and therefore $P_x=M_x\ltimes V_x$. 
The inner integral is therefore
\[
\int_{M_x\bs L_x}\delta_{Q_x}(m)^{-1} \ell(\varphi_\lambda(mg))\ dm.
\]
Recall that by \eqref{eq eq mod} we have $\delta_Q^{1/2}|_{Q_x}=\delta_{Q_x}$ and that the function $m\mapsto \delta_Q^{-1/2}(m)\varphi(mg)$, $m\in L$ lies in $I_{P\cap L}^L(\sigma)$. Thus $\theta_x$ is an involution on $L$ and by minimality of the vertex, the parabolic subgroup $P\cap L$ of $L$ satisfies $\theta_x(P\cap L)\cap (P\cap L)=M$. Note further that since $V$ is $\theta_x$-stable we have $\D_{M,x}(c)=\D_{M,\theta_x|_{L}}(c)$ where the right hand side is defined in \S \ref{s bd}. The result therefore follows from \S \ref{s bd}.
\end{proof}

\subsection{Convergence in the unramified case}

Let $P=M\ltimes U$ be a standard parabolic subgroup of $G$ and $\triv_M$ the trivial representation of $M$. For any $\lambda\in \aaa_M^*$ we consider the induced representation $I_P^G(\triv_{M},\lambda)$ realized on the space $I_P^G(\triv_M)$. Let $x\in X[M]$ and $\lambda\in (\aaa_M^*)_x^-$.
For $\varphi\in I_P^G(\triv_M)$ consider the integral
\[
J_P^G(\varphi;x,\lambda)=\int_{P_x\bs G_x} \varphi_\lambda(g)\ dg.
\]
Note that $\varphi_\lambda(pg)=\delta_P^{1/2}(p)\varphi_\lambda(g)$, $p\in P_x$ and $g\in G$ and by \eqref{eq eq mod}  the integral formally makes sense. 
\begin{proposition}\label{prop unram}
There exists $c>0$ such that $J_P^G(\varphi;x,\lambda)$ is defined by a convergent integral for every standard parabolic subgroup $P=M\ltimes U$ of $G$, $\varphi\in I_P^G(\triv_M)$, $x\in X[M]$ and $\lambda\in \D_{M,x}(c)$.
\end{proposition}
\begin{proof}
Let $\mathcal{M}$ be the finite set of standard Levi subgroups of $G$. For a parabolic subgroup $P$ of $G$ there are finitely many $P$-orbits in $X$, \cite[Proposition 6.15]{MR1215304}. Consequently, there are finitely many $M$-orbits in $X[M]$ for every $M\in \mathcal{M}$ (see \S\ref{ss adm}). 
Note further that for a standard parabolic subgroup $P=M\ltimes U$, $x\in X[M]$ and $m\in M$ we have
\[
J_P^G(\varphi;m\cdot x,\lambda)=e^{\sprod{\lambda+\rho_P}{H_M(m)}} J_P^G(I_P^G(m^{-1},\triv_M,\lambda)\varphi; x,\lambda)
\]
where the left hand side is defined by a convergent integral if and only if the right side is.

It therefore follows from Corollary \ref{cor bd} that there exists $c>0$ such that the integral $J_P^G(\varphi,x,\lambda)$ converges for all standard parabolic subgroups $P=M\ltimes U$ of $G$, $\varphi\in I_P^G(\triv_M)$, $x\in X[M]$ such that $(M,x)$ is a minimal vertex in $\G$ and $\lambda\in \D_{M,x}(c)$. 
We may further assume that $c$ is large enough so that the intertwining operators $M(n,\triv_M,\lambda)$ are defined by a convergent integral for every standard parabolic subgroup $P=M\ltimes U$ of $G$, $\alpha\in \Delta_P$, $n\in s_\alpha M$ and $\lambda\in \D^{M,s_\alpha}(c)$. Note further that if $x\in X[M]$ is such that $-\alpha\ne \theta_x(\alpha)<0$ then $\D_{M,x}(c)\subseteq \D^{M,s_\alpha}(c)$  (see \S\ref{sss intop} and \S \ref{s cone}).

It follows from Proposition \ref{prop min} that it is enough to prove the following. Assume that $(M,x) \overset{n}\searrow (M_1,x_1)$ is an edge on the graph $\G$ and $\alpha\in \Delta_P$ is such that $n\in s_\alpha M$. If 
$J_{P_1}^G(\varphi;x_1,s_\alpha\lambda)$ is defined by a convergent integral for every $\varphi\in I_{P_1}^G(\triv_{M_1})$ and $\lambda\in \D_{M_1,x_1}(c)$
then $J_P^G(\varphi;x,\lambda)$ is defined by a convergent integral for every $\varphi\in I_P^G(\triv_M)$ and $\lambda\in \D_{M,x}(c)$.

Let $\xi$ be the spherical vector in $I_P^G(\triv_M)$ normalized by taking the value one at the identity. Thanks to the Iwasawa decomposition, for every $\varphi \in I_P^G(\triv_M)$ there is a positive constant $A$ such that 
$|\varphi|\leq A\xi$ hence $|\varphi_\lambda|\leq A\xi_\lambda$ for all $\lambda\in \aaa_M^*$. It therefore suffices to consider
$\varphi=\xi$.

Let $P_1=M_1\ltimes U_1$ be the standard parabolic subgroup with Levi subgroup $M_1$ and $Q=L\ltimes V$ the standard parabolic subgroup containing $P$ such that $\Delta_P^Q=\{\alpha\}$.  For $\lambda\in \D_{M,x}(c)$ we have 
\[J_P^G(\xi;x,\lambda)=\int_{P_x\bs G_x} \xi_\lambda(g)\ dg= \int_{nP_{x}n^{-1}\bs G_{x_1}} \xi_\lambda(n^{-1}gn)\ dg.\]
It follows from \cite[Lemma 6.4]{MR3541705} that $nU_x n^{-1}\subseteq (U_1)_{x_1}$ and from \cite[Lemma 6.3]{MR3541705} that $P_x=M_x\ltimes U_x$ and $(P_1)_{x_1}=(M_1)_{x_1}\ltimes (U_1)_{x_1}$.

Since $(M_1)_{x_1}=n M_x n^{-1}$ we deduce that $nP_x n^{-1}\subseteq (P_1)_{x_1}$ and the inclusion $(U_1)_{x_1}\subseteq (P_1)_{x_1}$ induces an isomorphism $nU_{x}n^{-1}\bs (U_1)_{x_1} \simeq nP_x n^{-1}\bs (P_1)_{x_1}$. Applying \cite[Lemma 6.4]{MR3541705} again the projection from $Q$ to $L$ defines an isomorphism $nU_{x}n^{-1}\bs (U_1)_{x_1}\simeq L\cap U_1\simeq (U_1\cap n U n^{-1})\bs  U_1$. The local analogues of \cite[Lemma 4.3.1 (3) and (4)]{MR2010737} easily follow from \cite[Lemma 6.4]{MR3541705}. Proceeding as in the proof of \cite[Proposition 10.1.1]{MR2010737} we get that 
\begin{equation}\label{eq dblint}
J_P^G(\xi;x,\lambda)=\int_{(P_1)_{x_1}\bs G_{x_1}} \int_{(U_1\cap n U n^{-1})\bs  U_1}\xi_{\lambda}(n^{-1}ugn))\ du\ dg.
\end{equation}
By our assumption on $c$, the inner integral is absolutely convergent and we have
\[
 \int_{(U_1\cap n U n^{-1})\bs  U_1}\xi_{\lambda}(n^{-1}ugn))\ du=(M(n,\triv_M,\lambda)\xi)_{s_\alpha\lambda}(gn).
\]
By the Gindikin-Karpelevich formula we have $(M(n,\triv_M,\lambda)\xi)_{s_\alpha\lambda}(g)=c(\alpha,\lambda)\xi_{s_\alpha\lambda}(g)$, $g\in G$ where
\[
c(\alpha,\lambda)=\frac{L(\sprod{\lambda}{\alpha^\vee})}{L(1+\sprod{\lambda}{\alpha^\vee})},
\]
$L(s)=(1-q^{-s})^{-1}$ and $q$ is the size of the residual field of $F$.
By assumption the integral 
\[
\int_{(P_1)_{x_1}\bs G_{x_1}} \xi_{s_\alpha\lambda}(gn) \ dg
\]
converges (see \S \ref{s cone}) and by positivity of the integrand it follows that the double integral \eqref{eq dblint} converges.
The proposition follows.
\end{proof}
\begin{remark}\label{rmk domain}
It follows from the above proposition and its proof that there exists $c=c(G,P_0)> 0$ such that the integrals defining
\begin{itemize}
\item the linear forms $J_P^G(x,\lambda)$ and
\item the intertwining operators $M(n,\triv_M,\lambda)$
\end{itemize}
are absolutely convergent for every standard parabolic subgroup $P=M\ltimes U$ of $G$, $x\in X[M]$, $\lambda\in \D_{M,x}(c)$, $\alpha\in \Delta_P$ such that $-\alpha\ne \theta_x(\alpha)<0$ and $n\in s_\alpha M$. 
For the next section we fix such a $c$ and write $\D_{M,x}=\D_{M,x}(c)$. 
\end{remark}
\subsection{The main result}
Next, we show that the intertwining periods can be defined in the Galois setting. The method of proof also provides us with the simple functional equations that they satisfy. 

\begin{theorem}\label{thm merom}
Let $P=M\ltimes U$ be a standard parabolic subgroup of $G$, $\sigma$ a representation of $M$ of finite length with bounded matrix coefficients, $x\in X[M]$ and $\ell\in \Hom_{M_x}(\sigma,\triv_{M_x})$. For every $\lambda\in (\aaa_{M,\C}^*)_x^-$ such that $\Re(\lambda)\in \D_{M,x}$ and $\varphi\in I_P^G(\sigma)$ the integral
\[
J_P^G(\varphi;x,\ell,\sigma,\lambda)=\int_{P_x\bs G_x} \ell(\varphi_\lambda(g))\ dg
\]
converges. Furthermore, the linear form $J_P^G(x,\ell,\sigma,\lambda)$ admits a meromorphic continuation to $\lambda\in (\aaa_{M,\C}^*)_x^-$ and satisfies the following functional equations. Whenever $\alpha\in \Delta_P$ and $n\in s_\alpha M$ are such that $(M,x)\overset{n}\searrow (M_1,x_1)$ in $\G$ we have
\[
J_P^G(\varphi;x,\ell,\sigma,\lambda)=J_{P_1}^G(M(n,\sigma,\lambda)I_{P}^G(n,\sigma,\lambda)\varphi;x_1,\ell,s_\alpha \sigma, s_\alpha\lambda).
\]
\end{theorem}
\begin{proof}
We first prove the convergence for $\lambda\in (\aaa_{M,\C}^*)_x^-$ such that $\Re(\lambda)\in \D_{M,x}$. Taking the absolute integral we may assume that $\lambda=\Re(\lambda)$. 
For every $g\in G$ write $g=u_g m_gk_g$ with $u_g\in U$, $m_g\in M$ and $k_g\in K$. Then 
\[
\ell(\varphi_\lambda(g))=e^{\sprod{\lambda+\rho_P}{H_M(m_g)}}\ell(\sigma(m_g)\varphi(k_g))=e^{\sprod{\lambda+\rho_P}{H_M(g)}}\ell(\sigma(m_g)\varphi(k_g)).
\]
It follows from \cite[Th\'eor\`eme 4]{MR2381204} that the function $m\mapsto \ell(\sigma(m)v)$ is bounded for every $v$ in the space of $\sigma$. Since $\varphi$ takes finitely many values on $K$ it follows from the Iwasawa decomposition that there exists $C>0$ such that
\[
\int_{P_x\bs G_x} \abs{\ell(\varphi_\lambda(g))}\ dg\le C\int_{P_x\bs G_x} e^{\sprod{\lambda+\rho_P}{H_M(g)}}\ dg.
\]
The convergence therefore follows from Proposition \ref{prop unram}.

If $(M,x)$ is a minimal vertex in $\G$ the meromorphic continuation follows from Corollary \ref{cor bd}.
Now that convergence is granted, we show that if $\alpha\in \Delta_P$ and $n\in s_\alpha M$ are such that $(M,x)\overset{n}\searrow (M_1,x_1)$ in $\G$ and $P_1=M_1\ltimes U_1$ is the standard parabolic subgroup of $G$ with standard Levi subgroup $M_1$ then the meromorphic continuation of $J_{P_1}^G(x_1,\ell,\lambda)$ implies 
that of  $J_P^G(x,\ell,\lambda)$. In the process we establish the desired functional equation (which already appeared in the proof of the unramified case). In fact the computation hereunder is a completely parallel generalization of that done in Proposition \ref{prop unram}. We will use the notation of its proof and will not repeat all the arguments. In particular, $Q=L\ltimes V$ is the standard parabolic subgroup containing $P$ such that $\Delta_P^Q=\{\alpha\}$.

Since $\sigma$ has bounded matrix coefficients, it is straightforward that the integral defining $M(n,\sigma,\lambda)$ is absolutely convergent for $\lambda\in \D_{M,x}$ (see Remark \ref{rmk domain}).   
If $\Re(\lambda)\in \D_{M,x}$ the arguments in the proof of Proposition \ref{prop unram} imply that 
\[
J_P^G(\varphi;x,\ell,\sigma,\lambda)=\int_{P_x\bs G_x} \ell(\varphi_\lambda(g))\ dg= \int_{nP_{x}n^{-1}\bs G_{x_1}} \ell(\varphi_\lambda(g))\ dg\]
\[
=\int_{(P_1)_{x_1}\bs G_{x_1}} \int_{(U_1\cap n U n^{-1})\bs  U_1}\ell(\varphi_{\lambda}(n^{-1}ugn)))\ du\ dg
\] 
and by Lemma \ref{lem conv l} this equals
\[
\int_{(P_1)_{x_1}\bs G_{x_1}} \ell[(M(n,\sigma,\lambda)\varphi)_{s_\alpha\lambda}(gn)] \ dg.
\]
The theorem now follows by induction based on Proposition \ref{prop min}.

\end{proof}

\subsection{Applications to distinction}\label{section applications to distinction}
In this section we continue to denote by $(G,\theta)$ a Galois symmetric pair and by $X=\{x\in G: \theta(x)=x^{-1}\}$ the associated symmetric space.  Fix a standard parabolic subgroup $P=M\ltimes U$ of $G$ and a representation $\sigma$ of $M$ and let $\pi=I_P^G(\sigma)$. 

For every $P$ orbit $\O$ in $X$ there exists a standard parabolic subgroup $Q=L\ltimes V$ of $G$ contained in $P$ and $x\in \O$ such that $M\cap \theta_x(M)=L$. We say that $\O$ contributes to $\sigma$ if the normalized Jacquet module $r_{L,M}(\sigma)$ is $L_x$-distinguished. This property depends only on $\O$.

Let $x\in X$. As a consequence of \cite[Theorem 4.2 and Corollary 6.9]{MR3541705}, applied to the Galois symmetric pair $(G,\theta_x)$, if $\pi$ is $G_x$-distinguished then there exists a $P$-orbit in $G\cdot x\subseteq X$ that contributes to $\sigma$.
The following is a partial converse for this necessary condition for distinction.

\begin{proposition}\label{prop lift}
With the above notation, if $\sigma$ is of finite length with bounded matrix coefficients and there exists $x'\in G\cdot x$ such that $\theta_{x'}(M)=M$ and $\sigma$ is $M_{x'}$-distinguished then $I_P^G(\sigma,\lambda)$ is $G_x$-distinguished for every $\lambda\in (\aaa_{M,\C}^*)^-_{x'}$. 
\end{proposition}
\begin{proof}
We construct a meromorphic family of linear forms $L_\lambda\in \Hom_{G_x}(I_P^G(\sigma,\lambda),\triv_{G_x})$, $\lambda\in (\aaa_{M,\C}^*)_{x'}^-$ on $I_P^G(\sigma)$ as follows. Let $\eta\in G$ be such that $x'=\eta\cdot x$  (so that $G_{x'}=\eta G_x\eta^{-1}$), $0\ne \ell\in \Hom_{M_{x'}}(\sigma,\triv_{M_{x'}})$ and $L_\lambda=J_P^G(x',\ell,\sigma,\lambda)\circ I_P^G(\eta,\sigma,\lambda)$.
By Theorem \ref{thm merom}, indeed, $L_\lambda$ is a meromorphic family of linear forms as required and it is clearly non zero. The proposition now follows by taking a leading term of $L_\lambda$ along a generic direction. That is, for $\lambda\in (\aaa_{M,\C}^*)_{x'}^-$ there exist $\lambda_0\in (\aaa_{M,\C}^*)_{x'}^-$ (a generic direction) and a non-negative integer $m$ such that $s\mapsto s^m L_{\lambda+s\lambda_0}\in \Hom_{G_x}(I_P^G(\sigma,\lambda+s\lambda_0),\triv_{G_x})$ is holomorphic from $\C$ to the linear dual of $I_P^G(\sigma)$ and non-zero at $s=0$. Its value at $s=0$ gives a non-zero linear form in $\Hom_{G_x}(I_P^G(\sigma,\lambda),\triv_{G_x})$.
\end{proof}

As a consequence, for a representation of $G$ induced from cuspidal, contribution of a $P$-orbit characterizes distinction.

\begin{corollary}\label{cor cusp dist}
Let $x\in X$ and $\sigma$ a cuspidal representation of $M$ of finite length. Then $I_P^G(\sigma)$ is $G_x$-distinguished if and only if there exists $x'\in G\cdot x$ such that $\theta_{x'}(M)=M$ and $\sigma$ is $M_{x'}$-distinguished.
\end{corollary}
\begin{proof}
If $L$ is a proper Levi subgroup of $M$ then $r_{L,M}(\sigma)=0$. The only if direction is therefore immediate from the necessary condition obtained by \cite[Theorem 4.2 and Corollary 6.9]{MR3541705}. 

Assume now that $x'\in G\cdot x\cap X[M]$ and $\sigma$ is $M_{x'}$-distinguished. 
By cuspidality $\sigma$ is a direct sum of irreducible representations and it is enough to assume that $\sigma$ is irreducible. 

Let $\omega_\sigma$ be the central character of $\sigma$ and note that $\omega_\sigma$ is trivial on $(A_M)_{x'}$. Therefore, there exists $\mu\in (\aaa_M^*)_{x'}^-$ such that $\abs{\omega_\sigma}=e^{\sprod{-\mu}{H_M(\cdot)}}$ and therefore the unramified twist $\sigma[\mu]$ of $\sigma$ is unitary and in particular has bounded matrix coefficients. Furthermore, $\Hom_{M_{x'}}(\sigma,\triv_{M_{x'}})=\Hom_{M_{x'}}(\sigma[\mu],\triv_{M_{x'}})$. Since $I_P^G(\sigma)\simeq I_P^G(\sigma[\mu],-\mu)$ the corollary follows from Proposition \ref{prop lift} applied to $\sigma[\mu]$. 
\end{proof}

\begin{remark}\label{rem stand conj}
In Corollary \ref{cor cusp dist}, the cuspidality of $\sigma$ guarantees that only $P$-admissible orbits contribute to distinction. It would be interesting to explore to what extent this fact generalizes. It may be the case that in the context of Galois pairs, Corollary \ref{cor cusp dist}
remains true for $\sigma$ a unitary discrete series representation, or more optimistically a discrete series such that the induced representation is a standard module (as suggested by \cite[Lemma 3.3]{MR3421655} in the general linear case).
\end{remark}

\begin{remark}\label{rem counter-ex} 
If the necessary condition for distinction is satisfied by a non-admissible $P$-orbit, it does not guarantee distinction of the induced representation. We provide such an example.
The Galois pair in question is $(\GL_5(E),\GL_5(F))$ (and the example easily generalizes to $(\GL_{5n}(E),\GL_{5n}(F))$). Use the symbol $\times$ for normalized parabolic induction. Let $\pi_1$ be the trivial representation of $\GL_1(E)$, $\pi_2$ the irreducible submodule of 
$|\cdot|_E^0\times |\cdot|_E$, namely the character $|\det|_E^{1/2}$, and $\pi_3$ the irreducible quotient of $|\cdot|_E^{-1}\times |\cdot|_E^{0}$, namely the twist of the Steinberg representation of $\GL_2(E)$ by 
$|\det|_E^{-1/2}$. The normalized parabolic induction 
$\pi_1\times \pi_2 \times \pi_3$ (which is in fact irreducible by \cite[Theorem 0.1]{MR3237446}) is not $\GL_5(F)$-distinguished, however it has a Jacquet module distinguished by a stabilizer corresponding to a (necessarily) non admissible orbit. 

Indeed, the Jacquet module is that with respect to the Borel, as it has as a quotient $|\cdot|_E^0 \otimes |\cdot|_E^0\otimes |\cdot|_E \otimes |\cdot|_E^0 \times |\cdot|_E^{-1}$. And this quotient is distinguished with respect to the stabilizer corresponding to the double coset attached, in the notations of \cite{MR2755483}, to the sequence in $I(1,2,2)$ given by $n_{1,1}=n_{1,2}=0$, $n_{1,3}=1, n_{2,1}=0$, $n_{2,2}=n_{2,3}=1$, and $n_{3,1}=n_{3,2}=1$ and $n_{3,3}=0$. On the other hand, the induced representation $\pi_1\times \pi_2 \times \pi_3$ is not distinguished, it is the quotient of the standard module $|\cdot|_E\times |\cdot|_E^0\times |\cdot|_E^0\times \pi_3$, which is not distingushed according to \cite[Proposition 3.4]{MR3421655}.
\end{remark}

We end this section with a somewhat weaker variant of Proposition \ref{prop lift} that, nevertheless, removes the bounded hypothesis on the inducing representation and avoids the use of intertwining periods. For irreducibly induced representations it provides the same sufficient condition for distinction.

\begin{proposition}\label{prop lift 2}
Suppose that $x'\in G\cdot x$ is such that $\theta_{x'}(M)=M$ and $\sigma$ is a representation of $M$ of finite length that is $M_{x'}$-distinguished. Then there exists $w\in W(M)$ such that $I_{P_1}^G(w \sigma)$ is $G_x$-distinguished. (Here $P_1$ is the standard parabolic subgroup of $G$ with Levi subgroup $wMw^{-1}$.)
In particular, if in addition $I_P^G(\sigma)$ is irreducible then it is $G_x$-distinguished.  
\end{proposition}
\begin{proof}
Note that for $w\in W(M)$ and a representative $n$ of $w$ in $G$ we have 
\[
\Hom_{(wMw^{-1})_{n\cdot x'}}(n\sigma,\triv_{(wMw^{-1})_{n\cdot x'}})=\Hom_{M_{x'}}(\sigma,\triv_{M_{x'}}). 
\]
Consequently, applying Proposition \ref{prop min} it is enough to assume that $(M,x')$ is in fact minimal and prove that $I_P^G(\sigma)$ is $G_x$-distinguished. 
But then let $L$ and $Q$ be as in Section \ref{s min}. The representation 
$\sigma'=I_{P\cap L}^{L}(\sigma)$ is $L_{x'}$-distinguished thanks to the proof of \cite[Proposition 7.2]{MR3541705}. Now 
$I_P^G(\sigma)\simeq I_Q^G(\sigma')$ is $G_{x'}$-distinguished, hence $G_x$-distinguished, by the proof of \cite[Proposition 7.1]{MR3541705}.
\end{proof}

\begin{remark}\label{rem applications to PTB pairs}
The results of \S \ref{sec int per} also hold for symmetric pairs of Prasad and Takloo--Bighash type, i.e. those of the form $(\GL_m(D),C_E(\GL_m(D)))$ where 
$D$ is an $F$-central division algebra of index $d$ with $md$ even, and $C_E(\GL_m(D))$ is the centralizer 
of an embedding of $E$ in the space $M_m(D)$ of $m\times m$ matrices with entries in $D$. This follows from the fact that Equation (\ref{eq eq mod}) holds in this setting (see \cite[Equation (5.3) and Remark 5.4]{10.1093/imrn/rnz254}).
\end{remark}

\begin{remark}\label{rem descent of linear form}
In particular, Proposition \ref{prop lift} and Corollary \ref{cor cusp dist} become more powerful when $I_P^G(\sigma)$ is reducible. Note that any irreducible representation $\pi$ of $G$ is the quotient of a representation of the form $I_P^G(\sigma)$ for $\sigma$ as in Proposition \ref{prop lift} or Corollary \ref{cor cusp dist}, and if $\pi$ is $G_x$-distinguished, certainly $I_P^G(\sigma)$ is as well. Moreover, Proposition \ref{prop lift} provides a particular $G_x$-invariant linear form $L$ on $I_P^G(\sigma)$. In the situation where $\Hom_{G_x}(I_P^G(\sigma),\triv_{G_x})$ is one dimensional, this means that $\pi$ is distinguished if and only if $L$ descends to $\pi$. Such an observation is in practice very useful and has been used to study distinction of Langlands quotients in a variety of situations, see \cite{MR2930996}, \cite{MR3421655}, \cite{MR3719517} or \cite{matringe2017gamma}. 

If $I_P^G(\sigma,\lambda)$ satisfies that $\Hom_{G_x}(I_P^G(\sigma,\lambda),\triv_{G_x})$ is one dimensional for a generic choice of $\lambda$ then the corresponding intertwining periods should satisfy functional equations more general then those obtained in this work. 
Namely, for any $w\in W(M)$ (represented by $n\in G$) and $P_1$ the standard parabolic subgroup of $G$ with Levi subgroup $wMw^{-1}$, the intertwining period $J_P^G(x',\ell, \sigma, \lambda)$ satisfies an identity of the form
\[
J_{P_1}^G(n\cdot x',\ell, w \sigma,w \lambda)\circ M(n,\sigma,\lambda)\circ I_P^G(n,\sigma,\lambda)=\gamma_w(\sigma,\lambda)\ J_P^G(x',\ell, \sigma, \lambda)
\] 
for some proportionality constant $\gamma_w$ that is meromorphic in $\lambda$. In some cases $\gamma_w$ can be explicitly expressed in terms of local invariants associated to $\sigma$, and again this has turned out to be very helpful to understand when the linear form $L$ descends to $\pi$, see \cite{MR2930996}, \cite{matringe2017gamma} or \cite{SX}.
\end{remark}

\section{Classical Galois pairs}

In the next section we explicate Corollary \ref{cor cusp dist} for a Galois symmetric pair $(G,\theta)$ where $G$ is a classical group. In this section we introduce the setting and classify the $G$-orbits in $X$. This classification is new and of interest for its own sake.

\subsection{Further notation and preliminaries}\label{sec prel}
Let $k$ be a non Archimedean local field of characteristic zero, $K/k$ a field extension of degree one or two and $\rho$ a generator of the Galois group $\Gal(K/k)$. 
Let 
\[
(K/k)_1=\{a\in K^*: \ aa^{\rho}=1\} \ \ \ \text{and}  \ \ \ N(K/k)=\{aa^{\rho}: \ a\in K^*\}.
\]
That is, $(K/k)_1$ is the subgroup $\ker(N_{K/k})$ if $K/k$ is a quadratic extension and $N_{K/k}:K^*\rightarrow k^*$ is the norm map associated with $K/k$ and the group $\{\pm 1\}$ otherwise, while $N(K/k)$ is the subgroup $N_{K/k}(K^*)$ of $k^*$ if $K/k$ is a quadratic extension and the group $k^{*2}$ of squares in $k^*$ otherwise.

For $\sgn\in \{\pm 1\}$ we denote by
\[
Y_n(K/k,\sgn)=\{y\in \GL_n(K):{}^t y^\rho=\sgn y\}
\]
the space of $\sgn$-hermitian invertible matrices with respect to the field extension $K/k$.
It comes naturally equipped with the following right $\GL_n(K)$-action
\[
y\star g={}^t g^\rho y g,\ \ \ y\in Y_n(K/k,\sgn),\ g\in \GL_n(K).
\]
We denote by $G(y,K/k)$ the stabilizer of $y\in Y_n(K/k,\sgn)$ in $\GL_n(K)$.

This set up is a standard unified notation for considering symplectic, orthogonal and unitary groups. We recall the classification of $\GL_n(K)$-orbits in $Y_n(K/k,\sgn)$ and explicate the stabilizers in each case.
\begin{enumerate}
\item {\bf The symplectic case: $K=k$, $\sgn=-1$ and $n$ is even.} The space $Y_n(k/k,-1)$ is not empty if and only if $n$ is even and in that case it is the space of alternating matrices in $\GL_n(k)$. For $m\in \N$ we have that $Y_{2m}(k/k,-1)=\sm{}{I_m}{-I_m}{}\star \GL_{2m}(k)$ is a unique $\GL_{2m}(k)$-orbit.

The stabilizer $G(y,k/k)$ is the symplectic group $\Sp(y,k)$ associated to an element $y=-{}^t y$ in $\GL_n(k)$. 

\item {\bf The orthogonal case: $K=k$ and $\sgn=1$.} The space $Y_n(k/k,1)$ is the space of symmetric matrices in $\GL_n(k)$.  The $\GL_n(k)$-orbits in $Y_n(k/k,1)$ are determined by the discriminant and the Hasse invariant. The discriminant of $y\in Y_n(k/k,1)$ is the square class $\det y\, k^{*2}$. 
The Hasse invariant is defined as follows. There exists $g\in \GL_n(k)$ such that $y\star g=\diag(a_1,\dots,a_n)$ is diagonal and the Hasse invariant is defined by
\[
\Hasse_k(y)=\prod_{1\le i<j\le n} (a_i,a_j)_k
\]
where $(\cdot,\cdot)_k$ is the quadratic Hilbert symbol on $k^*\times k^*$.
We have that $y_1,\,y_2\in Y_n(k/k,1)$ are in the same $\GL_n(k)$-orbit if and only if $\det y_1 \,k^{*2}=\det y_2 \,k^{*2}$ and $\Hasse_k(y_1)=\Hasse_k(y_2)$.

For $n=1$ we have that $Y_1(k/k,1)=k^*$ and the $\GL_1(k)=k^*$-orbits are the  square classes in $k^*/k^{*2}$.
For $n=2$ there is a unique orbit of discriminant $-k^{*2}$ (it has Hasse invariant one) and two orbits for every other discriminant. For $n\ge 3$ there are two orbits for every discriminant class. 

The stabilizer $G(y,k/k)$ is the orthogonal group $\O(y,k)$ associated to an element $y={}^t y$ in $\GL_n(k)$.

\item {\bf The unitary case: $K/k$ is a quadratic extension.} The space $Y_n(K/k,\sgn)$ is the space of $\sgn$-hermitian matrices in $\GL_n(K)$ with respect to $K/k$. 
We have that $y_1,\,y_2\in Y_n(K/k,\sgn)$ are in the same $\GL_n(K)$-orbit if and only if $ \det  y_1 N(K/k)= \det y_2 N(K/k)$. By local class field theory $N(K/k)$ is an index two subgroup of $k^*$ and therefore $Y_n(K/k,\sgn)$ consists of two $\GL_n(K)$-orbits. 
The stabilizer $G(y,K/k)$ is the unitary group $\U(y,K/k)$ associated to the element $y=\sgn {}^t y^\rho$ in $\GL_n(K)$.
\end{enumerate}

Let $y\in Y_n(K/k,\sgn)$.
In the symplectic case, every element of $G(y,K/k)$ has determinant one. In either the unitary or orthogonal case $\det$ maps $G(y,K/k)$ to  $(K/k)_1$.
We observe below that this map is in fact surjective. We apply the well known fact that in both cases every $\GL_n(K)$-orbit in $Y_n(K/k,\sgn)$ contains a diagonal element.
\begin{observation}\label{obs det}
In the orthogonal and unitary cases for every $y\in Y_n(K/k,\sgn)$ and $a\in (K/k)_1$ there exists $h\in G(y,K/k)$ such that $\det h=a$. Furthermore, in the orthogonal case there exists such an $h$ that satisfies $h^2=I_n$. 
\end{observation}
Indeed, $y\star g$ is diagonal for some $g\in \GL_n(K)$ and we can take $h=g\diag(a,I_{n-1})g^{-1}$.

\subsection{Classical Galois pairs-The set up}
Fix a quadratic extension of non Archimedean local fields $E/F$ of characteristic zero and let $\sigma(x)=\bar x$ be the associated Galois action. In particular, $E=F[\imath]$ for some $\imath\in E\setminus F$ such that $\imath^2\in F$. 

We would like to consider distinction for unitary, special orthogonal and symplectic Galois pairs. In order to unify notation let $F'=F$ and $\tau$ be the identity on $F'$ in the symplectic or orthogonal case while $F'$ is a quadratic extension of $F$ such that $E\cap F'=F$ and $\tau$ is the Galois action associated to $F'/F$ in the unitary case. Let $E'=EF'$. Note that in the unitary case $\sigma$ (resp. $\tau$) has a unique extension to an automorphism of $E'$ over $F'$ (resp. over $E$) that, by abuse of notation, we denote by $\sigma$ (resp. $\tau$) and that $E'/F$ is a Klein extension (that is, a Galois extension such that $\Gal(E'/F)$ is isomorphic to the Klein group $(\Z/2\Z)^2$). As a consequence, for $x\in E'$ we can write $\bar x^\tau$ for $\sigma(\tau(x))=\tau(\sigma(x))$ with no ambiguity.

\subsubsection{The space of $\sgn$-hermitian forms}
Fix $n\in \N$ and $\sgn\in\{\pm 1\}$ with the convention that if $\sgn=-1$ then $F'=F$ and $n$ is even and set $Y_n(E)=Y_n(E'/E,\sgn)$ and $Y_n(F)=Y_n(F'/F,\sgn)$.
\begin{remark}
Note that if $F'/F$ is a quadratic extension and $\jmath \in F'$ is such that $\tau(\jmath)=-\jmath$ then $y\mapsto \jmath y$ is a stabilizer preserving bijection between $Y_n(E'/E,1)$ and $Y_n(E'/E,-1)$ and therefore in the unitary case it is enough to consider $\sgn=1$.
\end{remark}

\subsubsection{The general linear Galois symmetric space}\label{ss Hilbert90}
We consider another symmetric space
\[
Z_n(E)=\{g\in \GL_n(E'):g\bar g=I_n\}
\]
with the left $\GL_n(E')$-action $g\cdot z=gz\bar g^{-1}$ by twisted conjugation. Since the first cohomology set $\coh(\Gal(E'/F'),\GL_n(E'))$ is trivial, we have that $Z_n(E)=\GL_n(E')\cdot I_n$ is a unique $\GL_n(E')$-orbit.  

\subsubsection{The symmetric space associated to a classical Galois pair}

Fix $\j\in Y_n(F)$ and let $G^\j$ be the algebraic group, defined over $F$, of isometries of the $\sgn$-hermitian form $\j$. That is
\[
G^\j(F)=\{g\in \GL_n(F'): {}^t g^\tau \j g=\j\}
\]
and
\[
G^\j(E)=\{g\in \GL_n(E'): {}^t g^\tau \j g=\j\}.
\]
Consider the $G^\j(E)$-space 
\[
X^\j=X^\j(E/F)=\{x\in G^\j(E):x\bar x=I_n\}
\]
with the $G^\j(E)$ action
\[
g\cdot x=gx\bar g^{-1},g\in G^\j(E),\ x\in X^\j
\]
by twisted conjugation. 
Note that $X^\j$ is a $G^\j(E)$-invariant subspace of $Z_n(E)$.

Let $R_{E/F}(G^\j_E)$ be the Weil restriction of scalars from $E$ to $F$ of the base change of $G^\j$ to $E$. The symmetric space $X^\j$ is the set of $F$-points of a $R_{E/F}(G^\j_E)$-variety. We denote by $G^\j_x$ the stabilizer in $R_{E/F}(G^\j_E)$ of $x\in X^\j$. It is an algebraic group defined over $F$ such that 
\[
G^\j_x(F)=\{g\in G^\j(E): g\cdot x=x\}.
\]

\subsection{A classification of orbits}\label{ss orbit classification}

We parameterize $G^\j(E)$-orbits in $X^\j$ in terms of certain $\GL_n(F')$-orbits in $Y_n(F)$ and enterprete stabilizers accordingaly.
\begin{lemma}\label{lem orb iso}
For $x\in X^\j$ let $z\in \GL_n(E')$ be such that $x=z\cdot I_n$ (see \S\ref{ss Hilbert90}) and let $y=\j\star z\in Y_n(E)$. 
\begin{enumerate}
\item Then, in fact, $y\in Y_n(F)$ and the assignment $G^\j(E)\cdot x\mapsto y\star\GL_n(F')$ is well defined and defines a bijection between the $G^\j(E)$-orbits in $X^\j$ and the $\GL_n(F')$-orbits in $\j\star\GL_n(E')\cap Y_n(F)$.
\item Furthermore, 
\[
G_x^\j(F)=zG^y(F)z^{-1}.
\] 
\end{enumerate}
\end{lemma}
\begin{proof}
Note first that for $z\in \GL_n(E')$ we have that $z\cdot I_n=z\overline{z}^{-1}\in G^\j(E)$ if and only if 
$\j\star z={}^tz^\tau \j z\in Y_n(F)$. Moreover, for $z,\,z'\in \GL_n(E')$ we have that  $z \cdot I_n= z' \cdot I_n$ if and only if $z^{-1}z'\in \GL_n(F')$ if and only if $\j\star z$ and $\j\star z'$ belong to the same $\GL_n(F')$-orbit. Hence the formula $p(z\cdot I_n)=(\j\star z)\star \GL_n(F')$ well defines a map $p$ from $X^\j$ to the $\GL_n(F')$-orbits in $\j\star\GL_n(E')\cap Y_n(F)$.
For $x=z\cdot I_n$ and $g\in G^\j(E)$ we have $g\cdot x=(gz)\cdot I_n$ and $\j\star z=\j\star (gz)$ and therefore $p(x)=p(g\cdot x)$. Consequently, the map $p$ descends to a map $\overline{p}$ from the $G^\j(E)$-orbits in $X^\j$ to the $\GL_n(F')$-orbits in $\j\star\GL_n(E')\cap Y_n(F)$.

Assume now that $z,\,z'\in \GL_n(E')$ are such that $\j\star z=\j\star z'\in Y_n(F)$. Then $z'z^{-1}\in G^\j(E)$ and therefore $z\cdot I_n$ and $z'\cdot I_n$ are in the same $G^\j(E)$-orbit in $X^\j$. Consequently, the formula $q(\j\star z)=G^\j(E)\cdot (z\cdot I_n)$ well defines a map $q$ from $\j\star\GL_n(E')\cap Y_n(F)$ to the $G^\j(E)$-orbits in $X^\j$.
It is again straightforward that $q$ is constant on $\GL_n(F')$-orbits and therefore descends to a map $\overline{q}$ from the $\GL_n(F')$-orbits in $\j\star\GL_n(E')\cap Y_n(F)$ to
the $G^\j(E)$-orbits in $X^\j$. It is also straightforward to verify that $\overline{p}$ and $\overline{q}$ are inverses of one another and the bijection of orbits follows.  

For the second part, let $z\in \GL_n(E')$ be such that $x=z\cdot I_n \in X^\j$ and let $y=\j \star z$. For $g\in G^\j(E)$ we have $g\in G^\j_x(F)$ if and only if $g\cdot x=x$ if and only if $(z^{-1}gz)\cdot I_n=I_n$ if and only if $z^{-1}gz\in \GL_n(F')$.
Since $z^{-1} G^\j(E)z=G^y(E)$ we conclude that $z^{-1}G_x^\j(F)z=G^y(F)$ and
the lemma follows.
\end{proof}

Let $SG^\j(E)=G^\j(E)\cap \SL_n(E')$. We indicate the relation between orbits in $X^\j$ for the actions of $G^\j(E)$ and $SG^\j(E)$. In the symplectic case $G^\j(E)$ is a subgroup of $\SL_n(E')$. Therefore, the following lemma is only relevant to the case $\sgn=1$.

Next we observe that determinant is the only invariant of an $SG^\j(E)$-orbit in a $G^\j(E)$-orbit in $X^\j$.  
\begin{lemma}\label{lem slorb}
For every $x\in X^\j$ we have $SG^\j(E)\cdot x=\{x'\in G^\j(E)\cdot x: \det x'=\det x\}$. In particular, in the orthogonal case $SG^\j(E)\cdot x=G^\j(E)\cdot x$.
\end{lemma}
\begin{proof}
Clearly, all elements of $SG^\j(E)\cdot x$ have the same determinant $\det x$ and the inclusion $\subseteq$ follows. For the other inclusion, note that if $g\in G^\j(E)$ is such that $\det(g\cdot x)=\det x$ then $\det g\in F'\cap (E'/E)_1=(F'/F)_1$. It follows from Lemma \ref{lem orb iso} and Observation \ref{obs det} that for any $a\in (F'/F)_1$ there exists $h\in G^\j_x(F)$ with $\det h=a$. 
For $a=\det g^{-1}$ and such $h$ we have $g\cdot x=(gh)\cdot x$ and $gh\in SG^\j(E)$. The lemma follows. 
\end{proof}

We record separately consequences of Lemmas \ref{lem orb iso} and \ref{lem slorb} in the three cases we consider.

\subsubsection{The symplectic case}
Recall that in the symplectic case $Y_n(F)=\j\star \GL_n(F)$ is a unique $\GL_n(F)$-orbit. As an immediate consequence of Lemma \ref{lem orb iso} we obtain
\begin{corollary}\label{cor symp orb}
In the symplectic case $X^\j=G^\j(E)\cdot I_n$ is a unique $G^\j(E)$-orbit.  \qed
\end{corollary}

\subsubsection{The unitary case}
We begin with the following simple observation.
\begin{lemma}\label{lem klein ext}
Let $L/F$ be a Klein extension of local fields and let $E$ be a quadratic extension of $F$ in $L$. Then $F^*\subseteq N(L/E)$.
\end{lemma}
\begin{proof}
Let $E_1$ and $E_2$ be the two different quadratic extensions of $F$ contained in $L$ and different from $E$. Restriction to $E_i$ defines an isomorphism from $\Gal(L/E)$ to $\Gal(E_i/F)$ and therefore $N_{L/E}(E_i^*)=N(E_i/F)$ and in particular $N(E_i/F)$ is a subgroup of $N(L/E)$, $i=1,2$. By local class field theory $N(E_i/F)$, $i=1,2$ are two different subgroups of $F^*$ of index two and therefore their  product equals $F^*$. The lemma follows.
\end{proof}
Note that any $a\in (E'/E)_1\cap (E'/F')_1$ satisfies $a^\sigma=a^{-1}=a^\tau$ and it is therefore easy to see that
\[
(E'/E)_1\cap (E'/F')_1=((E')^{\sigma\tau}/F)_1.
\]
Let
\begin{equation}\label{eq def gamma}
\Gamma=\{a^{1-\sigma}: a\in (E'/E)_1\}
\end{equation}
and note that $\Gamma$ is a subgroup of $((E')^{\sigma\tau}/F)_1$. Note further that $\det$ maps $X^\j$ to $((E')^{\sigma\tau}/F)_1$.
\begin{observation}\label{obs x}
In the unitary case, for every $a\in ((E')^{\sigma\tau}/F)_1$ there exists $x\in X^\j$ such that $\det x=a$.
\end{observation}
Indeed, if $\j$ is diagonal then $\diag(a,I_{n-1})\in X^\j$ for every $a\in ((E')^{\sigma\tau}/F)_1$. Otherwise, there exists $g\in \GL_n(F')$ such that $\j\star g$ is diagonal and we then have  $X^\j =g X^{\j\star g} g^{-1}$ so that $\det(X^\j)=\det(X^{\j\star g})$. 

\begin{corollary}\label{cor 2orb}
In the unitary case $Y_n(F)\subseteq \j\star\GL_n(E')$. In particular, $X^\j$ consists of exactly two $G^\j(E)$-orbits. Namely, $\Gamma$ is of index two in $((E')^{\sigma\tau}/F)_1$ and the two $G^\j(E)$-orbits in $X^\j$ are $\{x\in X^\j:\det x\in \Gamma\}=G^\j(E)\cdot I_n$ and $\{x\in X^j:\det x\not\in \Gamma\}$.

\end{corollary}
\begin{proof}
Consider first the case $n=1$.
Note that $Y_1(F) =F^*\subseteq Y_1(E)=E^*$ and $\j\star\GL_1(E')=\j N(E'/E)$ so that 
\[
\j\star\GL_1(E')\cap Y_1(F)=\j N(E'/E)\cap F^*=\j(N(E'/E)\cap F^*).
\]
It follows from Lemma \ref{lem klein ext} that $\j\star\GL_1(E')\cap Y_1(F)=\j F^*$.
Note further that $\GL_1(F')$ acts on $\j\star\GL_1(E')\cap Y_1(F)$ by $y\star g=yN_{F'/F}(g)$ and since $N(F'/F)$ is of index two in $F^*$ it follows that $\j\star\GL_1(E')\cap Y_1(F)$ consists of two $\GL_1(F')$-orbits. 

As a consequence of Lemma \ref{lem orb iso} we obtain that $X^\j$ consists of two $G^\j(E)$-orbits. 
Note that $\j\in F^*$ and $G^\j(E)=(E'/E)_1$ acts on $X^\j=((E')^{\sigma\tau}/F)_1$ by $g\cdot x=g^{1-\sigma} x$.
It follows that, indeed, $\Gamma$ has index two in $((E')^{\sigma\tau}/F)_1$.

We now consider a general $n$.
We first show that $\j\star \GL_n(E')\cap Y_n(F)$ contains two different $\GL_n(F')$-orbits. After replacing $\j$ with an appropriate element of $\j\star\GL_n(F')$, we may assume without loss of generality that $\j$ is diagonal. Now let $a\in (E')^*$ be such that $N_{E'/E}(a)=aa^\tau\in F^*\setminus N(F'/F)$ be given by Lemma \ref{lem klein ext}. Then $y=\j\star\diag(a,I_{n-1})\in \j\star \GL_n(E')\cap Y_n(F)$ but $y\not\in \j\star\GL_n(F')$. 
Since $Y_n(F)$ consists of two $\GL_n(F')$-orbits it follows that $Y_n(F)=\j\star \GL_n(E')\cap Y_n(F)$. By Lemma \ref{lem orb iso} we get that $X^\j$ consists of two $G^\j$-orbits. 

It follows from Observations \ref{obs det}  (resp. \ref {obs x}) that $\det (G^\j\cdot x)=
\Gamma\det x$ (resp. $\det(X^\j)=((E')^{\sigma\tau}/F)_1$). The corollary follows.
\end{proof}


\subsubsection{The orthogonal case}

Consider the orthogonal case and let $SX^\j=X^\j\cap SG^\j(E)$. Note that $SX^\j$ (and its complement $X^\j\setminus SX^\j$) are $G^\j(E)$-invariant spaces since the $G^\j(E)$-action preserves determinant. We observe the following refinement of Lemma \ref{lem orb iso}.
\begin{lemma}\label{lem orth orb} Let $\GL_n^D(E)=\{g\in \GL_n(E):\det g\in D\}$ for a subset $D$ of $E$. Consider the orthogonal case. 
\begin{enumerate}
\item\label{part y} The bijection defined by Lemma \ref{lem orb iso} restricts to a bijection between $SG^\j(E)$-orbits in $SX^\j$ and $GL_n(F)$-orbits in $\j\star \GL_n^F(E)\cap Y_n(F)$ and to a bijection between $SG^\j(E)$-orbits in $X^\j\setminus SX^\j$ and $GL_n(F)$-orbits in $\j\star \GL_n^{\imath F}(E)\cap Y_n(F)$.
\item\label{part x} Each of $SX^\j$ and $X^\j\setminus SX^\j$ contains at most two $SG^\j(E)$-orbits. 
For $n=2$ if  $\det j\in -F^{*2}$ then $SX^\j$ is a unique $SG^\j(E)$-orbit and if $\det \j\in -\imath^2F^{*2}$ then $X^\j\setminus SX^\j$ is a unique $SG^\j(E)$-orbit. 
\end{enumerate}
\end{lemma}
\begin{proof}
Note first that $a\in E^*$ satisfies $a^2\in F^*$ if and only if $a\in F^* \sqcup \imath F^*$. As a consequence $E^{*2}\cap F^*=F^{*2}\sqcup \imath^2 F^{*2}$ and therefore comparing determinants we obtain the disjoint union
\[
\j\star \GL_n(E)\cap Y_n(F)=(\j\star \GL_n^F(E)\cap Y_n(F))\sqcup (\j\star \GL_n^{\imath F}(E)\cap Y_n(F)).
\]
Let $x\in X^\j$ and $z\in \GL_n(E)$ be such that $x=z\cdot I_n$. Recall that $G^\j(E) \cdot x$ corresponds to $(\j\star z)\star \GL_n(F))$ under the bijection of Lemma \ref{lem orb iso}. Note that if $x\in SX^\j$ then $1=\det x=\det z\bar z^{-1}$ and therefore $z\in GL_n^F(E)$ and if $x\in X^\j\setminus SX^\j$ then $-1=\det x=\det z\bar z^{-1}$ and therefore $z\in \GL_n^{\imath F}(E)$. By Lemma \ref{lem slorb} every $SG^\j(E)$-orbit is a $G^\j(E)$-orbit. The first part of the lemma follows. 

Note further that every element in $\j\star \GL_n^F(E)\cap Y_n(F)$ has the fixed discriminant $\det j F^{*2}$ and every element of $\j\star \GL_n^{\imath F}(E)\cap Y_n(F)$ has the fixed discriminant $\imath ^2\det \j F^{*2}$. The second part of the lemma follows from the first part and the classification of $\GL_n(F)$-orbits in $Y_n(F)$. 
\end{proof}

\begin{proposition}\label{prop orth orbs}
\begin{enumerate}
\item The space $SX^\j$ is a unique $SG^\j(E)$-orbit if and only if there is a unique $\GL_n(F)$-orbit in $Y_n(F)$ with discriminant $\det \j\, F^{*2}$. That is, if and only if either $n=1$ or $n=2$ and $\det \j\in -F^{*2}$. 

\item The space $X^\j\setminus SX^\j$ is a unique $SG^\j(E)$-orbit if and only if there is a unique $\GL_n(F)$-orbit in $Y_n(F)$ with discriminant $\imath^2\det \j \,F^{*2}$. That is, if and only if either $n=1$ or $n=2$ and $\det \j\in -\imath^2 F^{*2}$. 

\item\label{part 3} If either $n\ge 3$ or $n=2$ and $\det j\not\in -F^{*2}$ (resp. $n\ge 3$ or $n=2$ and $\det j\not\in -i^2 F^{*2}$) then $SX^\j$ (resp. $X^\j \setminus SX^\j$) consists of exactly two orbits. 
\end{enumerate}
\end{proposition}
\begin{proof}
By Lemma \ref{lem orth orb}\eqref{part y} the number of $SG^\j(E)$-orbits in $SX^\j$ equals the number of $\GL_n(F)$-orbits in $\j\star \GL_n^F(E)\cap Y_n(F)$ and the number of $SG^\j(E)$-orbits in $X^\j\setminus SX^\j$ equals the number of $\GL_n(F)$-orbits in $\j\star \GL_n^{\imath F}(E)\cap Y_n(F)$.
Note that 
\[
\GL_n^{\imath F}(E)=\diag(\imath,I_{n-1})\GL_n^F(E)=\GL_n^F(E)\diag(\imath ,I_{n-1}). 
\]

Replacing $\j$ by an appropriate element of $\j\star\GL_n(F)$ we assume without loss of generality that $\j$ is diagonal. When this is the case 
\[
\j\star \GL_n^{\imath F}(E)\cap Y_n(F)=(\j \diag(\imath^2,I_{n-1}))\star \GL_n^F(E)\cap Y_n(F).
\]
As a consequence, the second part of the proposition follows from the first and Lemma \ref{lem orth orb}\eqref{part y}. The third part follows from the first two and Lemma \ref{lem orth orb}\eqref{part x}.
It remains to prove the first part.

When $n=1$ we have that $SG^\j(E)=\{\pm 1\}$ acts trivially on $SX^\j=\{1\}$ and the statement is obvious. 
Consider the case $n=2$. If $\det \j\in -F^{*2}$ the statement follows from Lemma \ref{lem orth orb}\eqref{part x}. Assume that $\det \j\not\in -F^{*2}$. Without loss of generality assume that $\j$ is diagonal and write $\j=\diag(a,b)$. 
Let $K$ be a quadratic extension of $F$ such that $-a^{-1}b=u^2$ for some $u\in K^*$ and $\rho$ the Galois action of $K/F$.
Note that  for $\eta=\sm{u}{-u}{1}{1}\in \GL_2(K)$ we have
\[
\j\star \eta=\sm{0}{2b}{2b}{0}\ \ \ \text{and} \ \ \ \eta^{-1}\eta^\rho=\sm{0}{1}{1}{0}.
\]

Consider first the case that $-ab\in E^{*2}$ so that $K=E$ and $\rho=\sigma$. 
Then
\[
SG^\j(E)=\eta SG^{\j\star \eta}(E)\eta^{-1}=\eta \{t_a: a\in E^*\}\eta^{-1}
\]
where $t_a=\diag(a,a^{-1})$, $a\in E^*$.
For $a\in E^*$ and $g=\eta t_a\eta^{-1}$ we have 
\[
g\bar g=\eta t_{a\bar a^{-1}} \eta^{-1}.
\]
Hence 
\[
SX^\j=\eta\{t_a:a\in F^*\}\eta^{-1}. 
\]
Let $a\in E^*$ and $b\in F^*$ and write $g=\eta t_a\eta^{-1}\in SG^\j(E)$ and
$x=\eta t_b\eta^{-1}\in SX^\j$. Then 
\[
g\cdot x=\eta t_{a\bar a b} \eta^{-1}
\]
and since $N(E/F)$ is of index two in $F^*$ it follows that $SX^\j$ consists of two different $SG^\j(E)$-orbits.

Assume now that $-ab\not\in E^{*2}$ and let $L=KE$. Then $L/F$ is a Klein extension. 
By abuse of notation we continue to denote by $\rho$ the Galois action of $L/E$ and by $\sigma$ the Galois action of $L/K$. Note that $\eta^\sigma=\eta$. 

It follows from the previous case (with $L/E$ replacing $E/F$) that 
\[
SG^\j(L)=\eta \{t_a: a\in L^*\}\eta^{-1}.
\]
For $a\in L^*$ we have
\[
(\eta t_a \eta^{-1})^\rho=\eta t_{a^{-\rho}}\eta^{-1} 
\] 
and therefore
\[
SG^\j(E)=SG^\j(L)\cap \GL_2(E)=\eta\{t_a: a\in (L/E)_1\}\eta^{-1}.
\]
For $a\in (L/E)_1$ and $g=\eta t_a \eta^{-1}$ we have
\[
g \bar g=\eta t_{a a^\sigma}\eta^{-1}
\]
and therefore 
\[
SX^\j=\eta \{t_a: a \in (L/E)_1\cap (L/K)_1\}\eta^{-1}.
\]
For $a\in (L/E)_1$ and $b\in (L/E)_1\cap (L/K)_1$ set $g=\eta t_a\eta^{-1}\in SG^\j(E)$ and $x=\eta t_b\eta^{-1}\in SX^\j(E)$.
It follows that
\[
g\cdot x=\eta t_{a^{1-\sigma} b}\eta^{-1}.
\]
It follows from Corollary \ref{cor 2orb} that $\{a^{1-\sigma}:a\in (L/E)_1\}$ is of index two in $(L/E)_1\cap (L/K)_1$ and therefore $SX^\j$ consists of exactly two $SG^\j(E)$-orbits.

Next we show that for $n=3$, the space $\j\star \GL_3^F(E)\cap Y_3(F)$ contains two different $\GL_3(F)$-orbits. Note that for $a\in F^*$ the map $y\mapsto ay$ defines a $\GL_3(F)$-equivariant bijection between $\j\star \GL_3^F(E)\cap Y_3(F)$ and $(a\j)\star \GL_3^F(E)\cap Y_3(F)$. We may therefore assume without loss of generality that $\det j\in -F^{*2}$. Let $u\in F^*\setminus N(E/F)$ and recall that $(u,\imath^2)_F=-1$. It follows that $y_1=\diag(1,1,-1)$ and $y_2=\diag(\imath^2,u,-\imath^2u)$ both have discriminant $-F^{*2}$, are not in the same $\GL_3(F)$-orbit but are in the same $\GL_3(E)$-orbit. There is therefore $z\in \GL_3(E)$ such that $y_2=y_1\star z$ and comparing determinants it follows that $z\in \GL_3^F(E)$. Since $\j\star \GL_3(F)=y_i\star \GL_3(F)$ for some $i\in \{1,2\}$ our claim follows.

Next we show that for $n> 3$ the space $\j\star \GL_n^F(E)\cap Y_n(F)$  contains two different $\GL_n(F)$-orbits. 
We may again assume without loss of generality that $\j$ is diagonal. Write $\j=\diag(\j_1,\j_2)$ where $\j_1\in Y_3(F)$. It follows from the $n=3$ case that there exists $z\in \GL_3^F(E)$ such that $\j_1\star z \not\in \j_1\star \GL_3(F)$. It follows from Witt's cancellation theorem that $\j\star \diag(z,I_{n-3})\not\in \j\star \GL_n(F)$ while clearly $\diag(z,I_{n-3})\in \GL_n^F(E)$. All together this shows that for $n\ge 3$ it is always the case that  $\j\star \GL_n^F(E)\cap Y_n(F)$  contains two different $\GL_n(F)$-orbits. 
The proposition now follows from Lemma \ref{lem orth orb}\eqref{part y}.
\end{proof}

\subsubsection{Invariants for the $SG^\j(E)$-orbits in $X^\j$ in the orthogonal case}\label{ss xinv}
Continue to consider the orthogonal case. It will be convenient to introduce invariants that classify the $SG^\j(E)$-orbits in $X^\j$ based on the above results.
Let $x\in X^\j$, $z\in \GL_n(E)$ such that $x=z\cdot I_n$ and $y=\j\star z\in Y_n(F)$. Note that
\[
\det y F^{*2}=\begin{cases} \det \j F^{*2} & x\in SX^\j \\ \imath^2\det \j F^{*2} & x\in X^\j\setminus SX^\j \end{cases}
\]
and $\Hasse_F(y)$ are invariants of $SG^\j(E) \cdot x$ that uniquely determine it. 
Let 
\[
\partial^{\j}(x)=\det y F^{*2} \ \ \ \text{and}\ \ \ \hbar^\j(x)=\Hasse_F(y).
\]
For $x,\,x'\in X^\j$ we have $SG^\j(E)\cdot x=SG^\j(E)\cdot x'$ if and only if $(\partial^\j(x),\hbar^\j(x))=(\partial^\j(x'),\hbar^\j(x'))$.

\subsection{Stabilizers}
Next, we analyze conjugacy classes of stabilizers. 
In the orthogonal case let $SG^\j=\SL_n\cap G^\j$ and for $x\in X^j$ let $SG_x^\j=R_{E/F}(\SL_n)\cap G_x^\j$ be the algebraic group defined over $F$ so that $SG_x^\j(F)=\SL_n(E)\cap G_x^\j(F)$. 
In the sequel even or odd cases refer to the parity of $n$. 
\begin{lemma}\label{lem exp stab}
Let $H$ denote the algebraic $F$-group $SG^\j$ in the orthogonal case and $G^\j$ otherwise. 
\begin{enumerate}
\item In either the symplectic or the odd unitary case for every $x\in X^j$ the group $G_x^\j(F)$ is $G^\j(E)$-conjugate to $H(F)$.
\item In the even unitary case fix $x_0\in X^\j$ such that $\det x_0\not \in \Gamma$ (see \eqref{eq def gamma} and Observation \ref{obs x}) and let $H'=G_{x_0}^\j$.
For $x\in X^\j$ we have 
\begin{itemize}
\item if $\det x\in \Gamma$ then the group $G_x^\j(F)$ is $G^\j(E)$-conjugate to $H(F)$
\item and if $\det x\not\in \Gamma$  then the group $G_x^\j(F)$ is $G^\j(E)$-conjugate to $H'(F)$.
\end{itemize}
\item In the orthogonal case if either $n\ge 3$ or $n=2$ and $\det \j\not\in -F^{*2}$ fix $x_1\in SX^\j$ such that $\hbar^\j(x_1)\ne \Hasse_F(\j)$ and let $H_1=SG_{x_1}^\j$.

In the following three cases and only in these cases the group $SG_x^\j(F)$ is $SG^\j(E)$-conjugate to $H(F)$ for all $x\in SX^\j$: $n=1$, $n=2$ and $\det \j\in -F^{*2}$,  $n$ is even and $(\imath^2,(-1)^{\frac n2}\det \j)_F=-1$. Otherwise
for $x\in SX^\j$ we have
\begin{itemize}
\item if $\hbar^\j(x)=\Hasse_F(\j)$ then the group $SG_x^\j(F)$ is $SG^\j(E)$-conjugate to $H(F)$
\item otherwise the group $SG_x^\j(F)$ is $SG^\j(E)$-conjugate to $H_1(F)$.
\end{itemize}
\item In the odd orthogonal case for $x\in X^\j\setminus SX^\j$ we have
\begin{itemize}
\item if $\hbar^\j(x)=(\imath^2,-1)_F^{\frac{n-1}2}\Hasse_F(\j)$ then the group $SG_x^\j(F)$ is $SG^\j(E)$-conjugate to $H(F)$
\item otherwise the group $SG_x^\j(F)$ is $SG^\j(E)$-conjugate to $H_1(F)$.
\end{itemize}
\item In the even orthogonal case fix $x_2\in X^\j\setminus SX^\j$ and set $H_2=SG_{x_2}^\j$ and unless $n=2$ and $\det \j\in -\imath^2F^{*2}$ fix $x_3\in X^\j\setminus SX^\j$ such that $\hbar^\j(x_3)\ne \hbar^\j(x_2)$ and set $H_3=SG_{x_3}^\j$. 

In the following three cases and only in these cases the group $SG_x^\j(F)$ is $SG^\j(E)$-conjugate to $H_2(F)$ for all $x\in X^\j\setminus SX^\j$: $n=1$, $n=2$ and $\det j\in -\imath^2F^{*2}$, 
$n$ is even and  $(\imath^2,(-1)^{\frac n2+1}\det j)_F=-1$ . Otherwise for $x\in X^\j\setminus SX^\j$ we have
\begin{itemize}
\item if $\hbar(x)=\hbar(x_2)$ then the group $SG_x^\j(F)$ is $SG^\j(E)$-conjugate to $H_2(F)$
\item otherwise the group $SG_x^\j(F)$ is $SG^\j(E)$-conjugate to $H_3(F)$.
\end{itemize}

\end{enumerate}
\end{lemma}
\begin{proof}
Clearly, $G_{g\cdot x}^\j(F)=gG_x^\j(F)g^{-1}$, $g\in G^\j(E)$, $x\in X^\j$ and $G_{I_n}^j(F)=H(F)$. 
In the symplectic case, the statement is now obvious from Corollary \ref{cor symp orb} and in the even unitary case from Corollary \ref{cor 2orb}. 

In the odd unitary case note that for $a\in (E'/E)_1\cap (E'/F')_1$ and $x\in X^\j$ we have $ax\in X^\j$ and clearly $G_{ax}^\j=G_x^\j$. Note morever that since, in the notation of the Corollary \ref{cor 2orb}, the group $\Gamma$ is of index two in $(E'/E)_1\cap (E'/F')_1$ we have $\det(ax)\Gamma=a\det x\,\Gamma$. This case now also follows from Corollary \ref{cor 2orb} .

In the orthogonal case we apply the fact that for $a\in F^*$ and $y\in Y_n(F)$ we have $\Hasse_F(ay)=(a,(-1)^{\lfloor \frac n2 \rfloor}\det y^{n-1})_F\Hasse_F(y)$.
Note that for $x\in X^\j$ we have that $-x\in X^j$ and  if $x=z\cdot I_n$ then $-x=z'\cdot I_n$ where $z'=\imath z$. 
Also, in the odd case $x\in SX^\j$ if and only if $-x\in X^\j\setminus SX^\j$ while in the even case $x\in SX^\j$ if and only if $-x\in SX^\j$.
It easily follows that 
\[
\partial^\j(-x)=\begin{cases} \partial^\j(x) & n\text{ is even} \\ \imath^2 \partial^\j(x) & n\text{ is odd} \end{cases} \ \ \ \text{and}\ \ \  \hbar^\j(-x)=\begin{cases} (\imath^2,(-1)^{\frac{n}2}\partial^\j(x))_F\hbar^\j(x) & n\text{ is even} \\ (\imath^2,-1)_F^{\frac{n-1}2}\hbar^\j(x) & n\text{ is odd.}  \end{cases}
\]
Furthermore, 
\[
\partial^\j(x)=\begin{cases} \det \j F^{*2} & x\in SX^\j \\ \imath^2\det \j F^{*2}& x\in X^\j\setminus SX^\j. \end{cases}
\]
Since we have $SG_{-x}^\j=SG_x^\j$ the orthogonal case now follows from Proposition \ref{prop orth orbs} and \S\ref{ss xinv}. 
\end{proof}

\section{Distinction for classical Galois pairs-statement of the main result}\label{section Galois distinction and induction}

Let $\j\in Y_{n_0}(F)$ be anisotropic (that is, for $0\ne v\in (F')^{n_0}$ we also have ${}^t v^\tau\j v\ne 0$). 
This implies that in the symplectic case $n_0=0$, in the unitary case $n_0\in \{0,1,2\}$ and in the orthogonal case $n_0\in \{0,1,2,3,4\}$. For every $n\in \Z_{\ge 0}$ let $w_n=(\delta_{i,n+1-j})\in \GL_n(F)$, 
\[
\j[n]=\begin{pmatrix} & &  w_n \\ & \j & \\\sgn w_n& & \end{pmatrix}, \ \ \ G_n=G^{\j[n]}(E), \ X_n=X^{\j[n]} \ \text{and} \ H_n=G^{\j[n]}(F).
\]

The group $G_0$ is compact. In the orthogonal case let $G_n^\circ=G_n\cap \SL_{n_0+2n}(E)$ be the special orthogonal group. In order to unify notation set $G_n^\circ=G_n$ in the unitary and symplectic cases. Thus $G_n^\circ$ is a connected reductive $p$-adic group. 
For any subset $S$ of $G_n$ let $S^\circ=S\cap G_n^\circ$.

In the unitary case if $n_0=2$ let $x_0 \in X_0\setminus (G_0\cdot I_2)$ and 
\[
H_n'=G_{\diag(I_n,x_0,I_n)}^{\j[n]}(F),\ \ \ n\ge 0
\]
and if $n_0=0$ let $x_1\in X_1\setminus (G_1\cdot I_2)$ and
\[
H_n'=G_{\diag(I_{n-1},x_1,I_{n-1})}^{\j[n]}(F),\ \ \ n\ge 1.
\]

Consider the orthogonal case. If $n_0\ge 2$ let $x_{0,1}\in X_0^\circ$ be such that $\hbar^\j(x_0)\ne \Hasse_F(\j)$ 
(note that if $n_0=2$ then the anisotropic assumption means that $\det \j\not\in -F^{*2}$).
If in addition $n_0=2$ and $\det \j\not\in-\imath^2F^{*2}$ or $n_0=4$ let $x_{0,\ell}\in X_0\setminus X_0^\circ$, $\ell=2,3$ be such that $\hbar^\j(x_{0,2})\ne \hbar^\j(x_{0,3})$. 

If $n_0=0$ let $x_{2,1}\in X_2^\circ \setminus G_2\cdot I_4$ and if $n_0=1$ let $x_{1,1}\in X_1^\circ\setminus G_1\cdot I_3$. 
Furthermore, if either $n_0=0$ or $n_0=2$ and $\det \j\in-\imath^2F^{*2}$ let $x_{1,\ell}\in X_1\setminus X_1^\circ$, $\ell=2,3$ such that $\hbar^{\j[1]}(x_{1,2})\ne \hbar^{\j[1]}(x_{1,3})$. 

Whenever $x_{t,i}$ is defined above for some $t=0,1,2$ and $i=1,2,3$ (note that for a given $i$ there is at most one such $t$) set
\[
H_{i,n}=G_{\diag(I_{n-t},x_{t,i},I_{n-t})}^{\j[n]}(F),\ \ \ n\ge t, \ \ \ i=1,2,3.
\]
When $n_0=2$ and $\det j\in-\imath^2F^{*2}$ we also define $H_{2,0}=G_{y_{0,2}}^{\j}(F)$ for a fixed $y_{0,2}\in X_0\setminus X_0^\circ$.

Fix $n\in \N$ and let $G=G_n$, $X=X_n$ and $H=H_n$. 
Let $N=n_0+2n$ so that $\j[n]\in Y_N(F)$ and $G$ is a subgroup of $\GL_N(E')$.

\subsection{The standard parabolic subgroups of $G^\circ$}\label{ss parab}
 
 For a decomposition $\alpha=(N_1,\dots,N_k)$ of $N$ let $Q_\alpha=L_\alpha\ltimes V_\alpha$ be the standard parabolic subgroup of $\GL_N(E')$ of type $\alpha$ with standard Levi subgroup $L_\alpha$ and unipotent radical $V_\alpha$. That is, $Q_\alpha$ is the unique parabolic subgroup containing all upper-triangular matrices in $\GL_N(E')$ with Levi subgroup
 \[
 L_\alpha=\{\diag(g_1,\dots,g_k): g_i\in \GL_{N_i}(E'),\ i=1,\dots,k\}.
 \]

For the definition of a parabolic subgroup of $G$ and the following analysis (particularly in the orthogonal case when $G$ is not connected) we refer to \cite[CHAPITRE 1, III\ 2]{MR1041060}. 

The standard minimal parabolic subgroup of $G$ is $P_0=M_0\ltimes U_0$ with Levi part
\[
M_0=\{\diag(a_1,\dots,a_n,h,a_n^{-\tau},\dots,a_1^{-\tau}):a_1,\dots,a_n\in (E')^*,\ h\in G_0\}
\]
and unipotent radical $U_0=V_{(1^{(n)},n_0,1^{(n)})}\cap G$ consisting of block upper triangular unipotent matrices. (Here $1^{(n)}$ is the $n$-tuple $(1,\dots,1)$.)

Let $A$ be the set of all tuples $\alpha$ of non-negative integers of the form $\alpha=(n_1,\dots,n_k;r)$ where $k
,r\in \Z_{\ge 0}$, $n_i\in \N$, $1\le i \le k$ and $n=n_1+\cdots+n_k+r$. When $r=0$ we will simply write $\alpha=(n_1,\dots,n_k)\in A$ and denote by $A_0$ the subset of all tuples in $A$ with $r=0$.

For $m\in \N$ and $g\in \GL_m(E')$ let $g^*=w_m {}^t g^{-\tau}w_m$. 
For $\alpha=(n_1,\dots,n_k;r)\in A$ let $\inj=\inj_\alpha:\GL_{n_1}(E')\times\cdots\times \GL_{n_k}(E')\times G_r \rightarrow G$ be the imbedding
\[
\iota(g_1,\dots,g_k;h)=\diag(g_1,\dots,g_k,h,g_k^*,\dots,g_1^*).
\]
When $\alpha\in A_0$ we simply write $\iota(g_1,\dots,g_k)=\diag(g_1,\dots,g_k,g_k^*,\dots,g_1^*)$.
Denote by $M_\alpha$ the image of $\inj_\alpha$ and let 
\[
U_\alpha=V_{(n_1,\dots,n_k,n_0+2r,n_k,\dots,n_1)}\cap G
\]
and $P_\alpha=M_\alpha\ltimes U_\alpha$. The groups $P_\alpha$, $\alpha\in A$ are parabolic subgroups of $G$ containing $P_0$. 
The map $\alpha\mapsto P_\alpha$ is a bijection from $A$ to a complete set of representatives of the $G$-conjugacy classes of parabolic subgroups of $G$. A parabolic subgroup of $G$ is called standard if it is of the form $P_\alpha$ for some $\alpha\in A$. 

Excluding the split even orthogonal case (that is, when $F'=F$, $\sgn=1$ and $n_0=0$) the map $\alpha\mapsto P_\alpha^\circ$ is a bijection from $A$ to the parabolic subgroups of $G^\circ$ containing $P_0^\circ$. In order to describe the standard parabolic subgroups of $G^\circ$ in the split even orthogonal case we introduce some further notation. 

Consider the split even orthogonal case and let $\kappa=\kappa_n=\iota(I_{n-1};w_2)$. We set $\inj_\alpha'=\Ad(\kappa)\circ \inj_\alpha$ and let $M_\alpha'=\kappa M_\alpha \kappa$ be the image of $\inj_\alpha'$. When $\alpha$ is clear from the context we also write $\inj'=\inj_\alpha'$. We further write $U_\alpha'=\kappa U_\alpha \kappa$ and $P_\alpha'=\kappa P_\alpha \kappa=M_\alpha'\ltimes U_\alpha'$. Then $P_\alpha'$ is a parabolic subgroup of $G$ containing $P_0$. Note that for $\alpha=(n_1,\dots,n_k;r)\in A$ we have $P_\alpha=P_\alpha'$ except for $\alpha$ such that $r=0$ and $n_k\ne 1$. Furthermore, for $\alpha,\,\beta\in A$ we have that the two sets $\{P_\beta, P_\beta'\}$ and $\{P_\alpha,\, P_\alpha'\}$ intersect if and only if $\beta=\alpha$.

A parabolic subgroup of $G^\circ$ is called standard if it contains $P_0$. Every standard parabolic subgroup of $G^\circ$ is either of the form $P_\alpha^\circ$ or $(P_\alpha')^\circ$. However, there are some repetitions. We have  $P_\alpha^\circ=(P_\alpha')^\circ$ except for $\alpha$ such that $r=0$ and $n_k\ne 1$. Furthermore, if $n_1+\cdots+n_m+1=n$ we also have $P_{(n_1,\dots,n_m,1)}^\circ=P_{(n_1,\dots,n_m;1)}^\circ$. These cases provide the only possible repetitions. 

In order to unify notation, back to the general case, let $A^\circ$ equal $A$ unless we are in the split even orthogonal case in which case we define $A^\circ$ to be the disjoint union of 
\[
\{\alpha=(n_1,\dots,n_k;r)\in A:  r>1 \ \ \ \text{or}\ \ \ r=0, \  n_k=1\} 
\]
and 
\[
\{[\alpha,i]: \alpha=(n_1,\dots,n_k)\in A_0,\ n_k\ne 1, i\in\{\pm 1\}\}.
\]
For such $[\alpha,i]$ we write
\[
P_{[\alpha,i]}=\begin{cases} P_\alpha & i=1 \\ P_\alpha' & i=-1. \end{cases},\ \ M_{[\alpha,i]}=\begin{cases} M_\alpha & i=1 \\ M_\alpha' & i=-1. \end{cases} \ \ \ \text{and} \ \ \ U_{[\alpha,i]}=\begin{cases} U_\alpha & i=1 \\ U_\alpha' & i=-1. \end{cases}
\]
The map $a\mapsto P_a^\circ$ is a bijection from $A^\circ$ to the set of standard parabolic subgroups of $G^\circ$.

\subsection{Distinction and cuspidal induction for classical Galois pairs}\label{ss mainthm class}

Let $a\in A^\circ$ and $P=M\ltimes U$ with $M=M_a$ and $U=U_a$. That is, $P^\circ=M^\circ\ltimes U$ is a standard parabolic subgroup of $G^\circ$. 
If $a\in A\cap A^\circ$ let $\alpha=a$. Otherwise, write $a=[\alpha,\epsilon]$ where $\alpha=(n_1,\dots,n_k)\in A_0$ and $\epsilon\in\{\pm 1\}$ and set $r=0$.
If either $a\in A\cap A^\circ$ or $\epsilon=1$ denote by $\inj_a^\circ$ the restriction of $\inj_\alpha$ to $\GL_{n_1}(E')\times\cdots\times\GL_{n_k}(E')\times G_r^\circ$.
Otherwise, let $\inj_a^\circ=\inj_\alpha'$. 
Thus, 
\[
\inj_a^\circ: \GL_{n_1}(E')\times\cdots\times\GL_{n_k}(E')\times G_r^\circ \rightarrow M^\circ
\] 
is an isomorphism.

Let $\pi_i$ be a representation of $\GL_{n_i}(E')$, $i\in [1,k]$ and $\pi_0$ a representation of $G_r^\circ$. (If $n_0=r=0$, by $\pi_0$ we will always mean the trivial representaion of the trivial group).
Let $\pi=(\pi_1\otimes \cdots\otimes \pi_k \otimes \pi_0)\circ \inj_a^\circ$ be the associated representation of $M^\circ$. We apply the following standard notation for parabolic induction
\[
\pi_1\times\cdots\times\pi_k\ltimes \pi_0=I_{P^\circ}^{G^\circ}(\pi).
\]

Our main result in this section is the characterization of distinction for classical symmetric pairs of representations parabolically induced from cuspidal.

For $\rho\in S_k$ and a subset $\set$ of $[1,k]$ let 
\begin{itemize}
\item $I(\rho\set)=\{i\in \set: \rho(i)=i\}$;
\item $\odd(\set)=\odd_M(\set)$ be the number of $i\in \set$ such that $n_i$ is odd;
\item $N(\rho\set)=N_M(\rho\set)=\sum_{i\in I(\rho\set)} n_i$.
\end{itemize}

\begin{theorem}\label{thm Hdist ind class}
With the above notation, let $\pi_i$ be an irreducible cuspidal representation of $\GL_{n_i}(E')$, $i\in [1,k]$ and $\pi_0$ a cuspidal representation of $G_r^\circ$. The representation 
\[
\pi_1\times\cdots\times\pi_k\ltimes \pi_0
\]
of $G^\circ$ is $H^\circ$-distinguished if and only if there exist an involution $\rho\in S_k$ and a subset $\set$ of $[1,k]$ such that $\rho(\set)=\set$ and $n_{\rho(i)}=n_i$ such that 
\begin{equation}\label{eq rhocond}
\begin{array}{ll}
\pi_{\rho(i)}\simeq\bar \pi_i^\vee & i\not\in \set,\ \rho(i)\ne i \\
\pi_i \text{ is }\GL_{n_i}(F')\text{-distinguished} & i\not\in\set,\ \rho(i)=i \\
\pi_{\rho(i)}\simeq \pi_i^{\tau\sigma} & i\in  \set 
\end{array}
\end{equation}
and furthermore 
\begin{itemize}
\item in either the symplectic or the odd unitary case or if $I(\rho \set)$ is empty then $\pi_0$ is $H_r^\circ$-distinguished
\item if $I(\rho\set)$ is not empty then
\begin{itemize}
\item in the even unitary case $\pi_0$ is either $H_r$-distinguished or $H_r'$-distinguished
\item  in the odd orthogonal case if $n_0+2r>1$ then $\pi_0$ is either $H_r$-distinguished or $H_{1,r}$-distinguished
\item in the even orthogonal case if $n_0+2r=0$ then $\odd(\set)$ is even and if $n_0+2r>0$ then
\begin{itemize}
\item if $\odd(\set)$ is even then $\pi_0$ is either $H_r$-distinguished or $H_{1,r}$-distinguished
\item if $\odd(\set)$ is odd then
\begin{itemize}
\item if $n_0=2$, $r=0$ and $\det \j\in-\imath^2 F^{*2}$ then $\pi_0$ is $H_{2,0}$-distinguished
\item otherwise $\pi_0$ is either $H_{2,r}$-distinguished or $H_{3,r}$-distinguished.
\end{itemize}
\end{itemize}
\end{itemize}
\end{itemize}
\end{theorem}

\section{Distinction for classical Galois pairs-proof of the main result}\label{s main pf}

\subsection{A reduction step} We begin the proof of the theorem with a reduction to the case where $P$ is a standard parabolic subgroup of $G$, that is, either $a\in A$ or $\epsilon=1$ in the above notation. This step is only relevant to the split even orthogonal case.

Consider the split even orthogonal case and let $\alpha=(n_1,\dots,n_k)\in A_0$. 
Recall that $\kappa=\inj(I_{n-1};w_2)$ and for a representation $\pi$ of $G^\circ$ let $\pi^\kappa$ be the representation of $G^\circ$ on the space of $\pi$ defined by $\pi^\kappa(g)=\pi(\kappa g\kappa^{-1})$. More generally, for a representation $\pi$ of $M_\alpha$ denote by $\pi^\kappa$ the representation of $M_\alpha'=\kappa M_\alpha \kappa^{-1}$ defined similarly.
\begin{lemma}\label{lem kappa}
\begin{enumerate}
Let $x\in X$.
\item For a representation  $\pi$ of $G^\circ$ we have
\[
\Hom_{G_x^\circ}(\pi,\triv)\simeq\Hom_{G_x^\circ}(\pi^\kappa,\triv).
\]

\item Let $\alpha\in A_0$ and $\pi$ a representation of $M_\alpha$. Then 
\[
\Hom_{G_x^\circ}(I_{P_\alpha}^{G^\circ}(\pi),\triv)\simeq \Hom_{G_x^\circ}(I_{P_\alpha'}^{G^\circ}(\pi^\kappa),\triv).
\]
\end{enumerate}
\end{lemma}
\begin{proof}
For the first part, note that $\kappa G_x^\circ \kappa^{-1}=G_{\kappa\cdot x}^\circ$ and therefore 
\[
\Hom_{G_x^\circ}(\pi^\kappa,\triv)\simeq\Hom_{G_{\kappa\cdot x}^\circ}(\pi,\triv).
\]
It follows from Lemma \ref{lem slorb} that $G_{\kappa\cdot x}^\circ$ is $G^\circ$-conjugate to $G_x^\circ$ and therefore 
\[
\Hom_{G_{\kappa\cdot x}^\circ}(\pi,\triv) \simeq \Hom_{G_x^\circ}(\pi,\triv).
\]
The first part follows. The second part follows from the first part and the observation that  $I_{P_\alpha'}^{G^\circ}(\pi^\kappa)\simeq I_{P_\alpha}^{G^\circ}(\pi)^\kappa$. Indeed, for $\varphi\in I_{P_\alpha}^{G^\circ}(\pi)$ let $\varphi^\kappa(g)=\varphi(\kappa g\kappa^{-1})$, $g\in G^\circ$. Then $\varphi\mapsto \varphi^\kappa$ realizes this equivalence of representations.
\end{proof}

We continue with some further notation and preliminaries.

\subsection{Admissible orbits}\label{sec adm}
For a standard parabolic subgroup $P=M\ltimes U$ of $G$ we recall some facts on the $P^\circ$-orbits in $X^\circ$ following \cite[Section 3]{MR3541705}. For $x\in X^\circ$ the $P^\circ$-orbit $P^\circ\cdot x$ is contained in the Bruhat cell $P^\circ xP^\circ$. By the Bruhat decomposition it corresponds to a double coset in $W_{M^\circ} \bs W_{G^\circ}/ W_{M^\circ}$. If $w$ is the element of minimal length in this double coset, then $w$ is an involution. We say that $P^\circ\cdot  x$ is $M^\circ$-admissible if $w$ normalizes $M^\circ$.  
Intersection with $N_{G^\circ}(M^\circ)$ defines a bijection from $M^\circ$-admissible $P^\circ$-orbits in $X^\circ$ to $M^\circ$-orbits in $X\cap N_{G^\circ}(M^\circ)$.  In the sequel we study $M^\circ$-orbits in $X\cap N_{G^\circ}(M^\circ)$. 

\subsection{The signed permutation group}
Let $\weyl_k$ be the signed permutation group on $k$ elements realized as $\weyl_k=S_k\ltimes \Xi_k$ where $\Xi_k\simeq (\Z/2\Z)^k$ is the group of subsets of the integer interval $[1,k]$ with the operation of symmetric difference and the group $S_k$ of permutations of $[1,k]$  acts naturally on $\Xi_k$ (In $\weyl_k$ we have $\rho\set\rho^{-1}=\rho(\set)$, $\rho\in S_k$, $\set\in \Xi_k$. That is, conjugation of an element $\set$ in $\Xi_k$ by the permutation $\rho$ is the image of the set $\set$ under the permutation $\rho$). Denote by $S_k[2]$ the set of involutions in $S_k$ and by $\weyl_k[2]$ the set of involutions in $\weyl_k$. Note that 
\[\weyl_k[2]=\{\rho \set: \rho\in S_k[2],\ \set \in  \Xi_k,\ \rho(\set)=\set\}.\]
When we write $w=\rho\set\in \weyl_k$ we always mean that $\rho\in S_k$ and $\set\in \Xi_k$.


\subsection{The split even orthogonal case}

Some of the arguments in the sequel will require further justification in the case where $F'=F$, $n_0=0$ and $\sgn=1$ henceforth the split even orthogonal case. 

In the split even orthogonal case we define the Weyl group $W_G$ of the non-connected group $G$ as $W_G=N_G(T)/T$. 
Clearly $W_{G^\circ}$ is a subgroup of $W_G$.
Note that $T=\{\inj(a_1,\dots,a_n): a_i\in F^*,\ i\in [1,n]\}$ is a maximal split torus of $G^\circ$ contained in $M_0$ and let $e_i\in X^*(T)$ be defined by $e_i(\inj(a_1,\dots,a_n))=a_i$, $i\in [1,n]$. 
The split connected group $G^\circ$ has a root system
\[
R(T,G^\circ)=\{\pm (e_i\pm e_j):1\le i\ne j\le n\}
\]
of type $D_n$ with a basis of simple roots
\[
\Delta_0^\circ=\{e_i-e_{i+1}:i\in [1,n-1]\}\sqcup \{e_{n-1}+e_n\}.
\]
The set
\[
R(T,G)\footnote{the notation $R(T,G)$ is somewhat artificial, note that there are no root subgroups for the roots $2e_i$}=R(T,G^\circ)\sqcup \{\pm 2 e_i: i\in [1,n]\}
\]
is a root system of type $C_n$  with basis
\[
\Delta_0=(\Delta_0^\circ\setminus \{e_{n-1}+e_n\})\sqcup \{2e_n\}.
\]
It is easy to see that the action of $N_G(T)$ on $T$ by conjugation identifies $W_G$ as the Weyl group of the root system $R(T,G)$ and in particular, defines an isomorphism  
\[
W_G\simeq \weyl_n.
\] 
In particular, $W_G$ admits a length function with respect to the elementary reflections associated with the simple roots in $\Delta_0$.

For the rest of this section let $\alpha=(n_1,\dots,n_k;r)\in A\cap A^\circ$, $M=M_\alpha$, $U=U_\alpha$ and $P=M\ltimes U$. 
\subsection{The set $W_G(M)$ and elementary symmetries.}\label{sec elm sym}
Let $\{s_\alpha:\alpha\in \Delta_{M^\circ}\}$ be the set of elementary symmetries in $W_{G^\circ}(M^\circ)$ and $\ell=\ell_{M^\circ}: W_{G^\circ}(M^\circ)\rightarrow \Z_{\ge 0}$ the length function defined in \cite[I.1.7]{MR1361168}. 

For the sake of unified notation set $W_G(M)=W_{G^\circ}(M^\circ)$, $\ell_M=\ell_{M^\circ}$, $\Delta_P=\Delta_{P^\circ}$ and $R(A_M,G)=R(A_{M^\circ},G^{\circ})$ except if $r=0$ in the split even orthogonal case. 

Assume that $r=0$ in the split even orthogonal case and note that $M=M^\circ$.  
Let $W_G(M)$ be the subset of elements $w\in W_G$ such that $w$ has minimal length in $w W_M$ and $wMw^{-1}$ is a standard Levi subgroup of $G$ (that is, of the form $M_\beta$ for some $\beta\in A$). 
Let $R(A_M,G)$ be the set of non-zero restrictions to $A_M$ of the roots in $R(T,G)$ and let $\Delta_P$ be the non-zero restrictions to $A_M$ of the roots in $\Delta_0$. For $\alpha\in R(A_M,G)$ we write $\alpha>0$ if it is the restriction of a positive root in $R(T,G)$ with respect to $\Delta_0$ and $\alpha<0$ otherwise. Note that 
\[
A_M=\{\inj(a_1I_{n_1},\dots,a_k I_{n_k}): a_i\in F^*,\,i\in [1,k]\}
\]
and $\Delta_P=\{\alpha_1,\dots,\alpha_k\}$ where 
\[
\alpha_i(\inj(a_1I_{n_1},\dots,a_k I_{n_k}))=a_ia_{i+1}^{-1}, \ \ \ i\in [1,k-1]
\] 
and 
\[
\alpha_k(\inj(a_1I_{n_1},\dots,a_k I_{n_k}))=a_k^2. 
\]
Let 
\[
s_i=s_{i,M}=\inj(I_{n_1+\cdots+n_{i-1}},\sm{}{I_{n_{i+1}}}{I_{n_i}}{},I_{n_{i+2}+\cdots+n_k}),\ \ \ i\in [1,k-1]
\]
and
\[
s_k=s_{k,M}=\inj(I_{n_1+\cdots+n_{k-1}};\sm{}{I_{n_k}}{I_{n_k}}{}).
\]
Note that $s_i T\in W_G(M)$ and set $s_{\alpha_i}=s_i T$. The elements $\{s_\alpha: \alpha\in \Delta_P\}$ serve the role of elementary symmetries for $W_G(M)$. We define the length function $\ell_M$ on $W_G(M)$ by letting $\ell_M(w)$ be the cardinality of $\{\alpha\in R(A_M,G): \alpha>0,\,w\alpha<0\}$.
With this set up the results of \cite[I.1.7 and I.1.8]{MR1361168} extend to $W_G(M)$ without modification.

\subsection{A choice of representatives in $G$ for $W_G(M)$}\label{ss WMG}
If $n_0>0$ fix once and for all an element $\eta_0\in H_0$ as follows. 
In the unitary case set $\eta_0=I_{n_0}$ and in the orthogonal case let $\eta_0\in H_0\setminus H_0^\circ$ be such that $\eta_0^2=I_{n_0}$ (see Observation \ref{obs det}).

For $m\in \N$ we define an involution $\eta_m\in G_m$ as follows.
If $n_0>0$ set $\eta_m=\iota(I_m;\eta_0)$. If $n_0=0$, in the (split even) orthogonal case set $\eta_m=\inj(I_{m-1};w_2)$ otherwise set $\eta_m=I_{2m}$.

For every $w\in \weyl_k$ we define an element $t_w=t_{w,M}\in G$ as follows.
For $\rho\in S_k$ let $w_\rho\in \GL_n(F)$ be the unique permutation matrix such that
\[
w_\rho\diag(g_1,\dots,g_k)w_\rho^{-1}=\diag(g_{\rho^{-1}(1)},\dots,g_{\rho^{-1}(k)}),\ \ \ g_i\in \GL_{n_i}(F), i\in[1,k]
\]
and let 
\[
t_{\rho,M}=\inj(w_\rho;I_{n_0+2r}).
\]
For $r\le m\in \N$ let $\eta_{m,M}=\eta_m$ except if $r=0$ in the split even orthogonal case where we set $\eta_{m,M}=I_{2m}$. Note that $\eta_{m,M}=\inj(I_{m-r};\eta_{r,M})$.
For $i\in [1,k]$ let
\[
t_{i,M}=\begin{pmatrix} I_{n_1+\cdots+n_{i-1}} & & & &  \\ & & & I_{n_i} & \\ & & \eta_{n_{i+1}+\cdots+n_k+r,M}^{n_i}  & & \\ & \sgn I_{n_i} & & & \\ & & & & I_{n_1+\cdots+n_{i-1}}  \end{pmatrix}
\]
and note that $t_{i,M}\in G^\circ$ except if $n_i$ is odd and $r=0$ in the split even orthogonal case in which case $t_{i,M}\in G\setminus G^\circ$.
Note further that $t_{1,M},\dots,t_{k,M}$ commute with each other. For $\set\in \Xi_k$ let 
\[
t_{\set,M}=\prod_{i\in \set} t_{i,M}
\]
and recall that $\odd(\set)=\odd_M(\set)$ is the number of $i\in \set$ such that $n_i$ is odd.  
Thus, $t_{\set,M}\in G^\circ$ except if $\odd_M(\set)$ is odd and $r=0$ in the split even orthogonal case in which case $t_{\set,M}\in  G\setminus G^\circ$.
Finally, for $w=\rho\set\in \weyl_k$ let
\[
t_w=t_{w,M}=t_{\rho,M}t_{\set,M}.
\]
Note that $t_w\in N_G(T)\cap H$ and in fact $t_w\in G^\circ$ except if $\odd_M(\set)$ is odd and $r=0$ in the split even orthogonal case.
Furthermore, $t_w$ represents an element in $W_G(M)$ (that is, $t_w M_0\in W_G(M)$) and for $g_i\in \GL_{n_i}(E')$, $i\in [1,k]$, $h\in G_r$ and $m=\inj(g_1,\dots,g_k;h)\in M$ we have
 \begin{equation}\label{eq wconj}
t_w m t_w^{-1}=\imath(g'_1,\dots,g'_k;h') \ \ \ 
\text{where} \ \ \
g_i'=\begin{cases} g_{\rho^{-1}(i)} & i\not\in \set \\ g_{\rho^{-1}(i)}^*& i\in \set \end{cases}  \ \ \ \text{and} \ \ \ h'=\eta_r^{\odd(\set)}h \eta_r^{-\odd(\set)}.
\end{equation}
In particular, $t_w Mt_w^{-1}=M_{w\alpha}$ where $w\alpha=(n_{\rho^{-1}(1)},\dots,n_{\rho^{-1}(k)};r)$.
Note further that for $w_1,\,w_2\in \weyl_k$, $t_{w_1}=t_{w_1,M_{w_2\alpha}}$, $t_{w_2}=t_{w_2,M}$ and $t_{w_1w_2}=t_{w_1w_2,M}$ we have
\[
t_{w_1}t_{w_2} m (t_{w_1}t_{w_2})^{-1}=t_{w_1w_2}  m t_{w_1w_2}^{-1},\ \ \ m\in M.
\]
That is, $t_{w_1w_2}^{-1}t_{w_1}t_{w_2}$  lies in the center of $M$ and in particular in $M_0$. 

Let $s_1,\dots,s_k$ be the standard simple reflections in $\weyl_k$, that is, $s_i=(i,i+1)\in S_k$, $i\in [1,k-1]$ and $s_k=\{k\}\in \Xi_k$. Note that
$t_{s_1},\dots,t_{s_k}$ are representatives of the elementary symmetries in $W_G(M)$. It is now a simple consequence of the results in \cite[I.1.7 and I.1.8]{MR1361168}  and the discussion in \S \ref{sec elm sym} that the map $w\mapsto t_w M_0:\weyl_k\rightarrow W_G(M)$ is bijective. We denote by $j_M:W_G(M)\rightarrow \weyl_k$ its inverse. For $w\in W_G(M)$ and $w'\in W_G(M')$ where $M'=wMw^{-1}$ we have
\[
j_M(w'w)=j_{M'}(w')j_M(w).
\]
\subsubsection{}\label{sec sp orth} In this paragraph assume that $r=0$ in the even orthogonal case. Note that neither of $W_G(M)$ and $W_{G^\circ}(M^\circ)$ is necessarily contained in the other. In fact, it is easy to see that
\[
W_G(M)\cap W_{G^\circ}(M^{\circ})=\{w\in W_G(M): j_M(w)=\rho\set \  \text{where} \ \odd(\set) \ \text{is even}\}.
\]
Furthermore, clearly $W_{G^\circ}(M^{\circ},M^{\circ})\subseteq W_G(M)\cap W_{G^\circ}(M^{\circ})$ and we conclude that
\[
W_{G^\circ}(M^{\circ},M^{\circ})=\{w\in W_G(M): j_M(w)=\rho\set \  \text{where} \ \odd(\set) \ \text{is even and}\ n_{\rho(i)}=n_i,\,i\in [1,k]\}.
\]
In particular, $t_{j_M(w)}\in G^\circ$, $w\in W_{G^\circ}(M^{\circ},M^{\circ})$. 
\subsection{Orbits and stabilizers in $X\cap N_{G^\circ}(M^\circ)$}\label{subsec orb and stab}
Let 
\[
\weyl_k[2,M]=\{w=\rho\set\in \weyl_k[2]: n_{\rho(i)}=n_i,\,i\in [1,k]\}.
\]
We begin with two lemmas that consist of simple computations that allow us to compute $M^\circ$-orbits and stabilizers in $X\cap N_{G^\circ}(M^\circ)$.

\begin{lemma}\label{lem in x}
Let $w=\rho\set\in \weyl_k[2,M]$ and $m=\inj(g_1,\dots,g_k;h)\in M$ where $g_i\in \GL_{n_i}(E')$, $i\in [1,k]$ and $h\in G_r$. Then
\[
t_w m \overline{t_w m}=\inj(a_1,\dots,a_k;b) \ \ \ \text{where} \ \ \ a_i=\begin{cases} g_{\rho(i)} \bar g_i & i\not\in \set \\ \sgn g_{\rho(i)}^* \bar g_i & i\in \set \end{cases}  \ \ \ \text{and} \ \ \ b=\eta_r^{\odd(\set)}h \eta_r^{\odd(\set)}\bar h.
\]
In particular, $t_w m\in X$ if and only if
\[
\begin{array}{ll}
g_{\rho(i)}\bar g_i=I_{n_i} & i\not\in \set,\ \rho(i)\ne i \\
g_i\in Z_{n_i}(E) & i\not\in\set,\ \rho(i)=i \\
g_{\rho(i)}^*\bar g_i=\sgn I_{n_i} & i\in  \set, \ \rho(i)\ne i \\
w_{n_i} \bar g_i\in Y_{n_i}(E'/(E')^{\sigma\tau},\sgn) & i\in \set,\ \rho(i)=i \\
\eta_r^{\odd(\set)}h\in X_r.& 
\end{array}
\]

\end{lemma}
\begin{proof}
Note that $t_w\in N_G(M)$ and recall that $\bar t_w=t_w$. It follows from the above discussion that $t_w^2$ lies in the center of $M$ and therefore $t_w m \overline{t_w m}=t_w^2 (t_w mt_w^{-1}) \bar m$. In fact, it is easy to compute explicitly that
\[
t_w^2=\inj(u_1,\dots,u_k;I_{n_0+2r}) \ \ \ \text{where} \ \ \ u_i=\begin{cases} I_{n_i} & i\not\in \set \\ \sgn I_{n_i} & i\in \set. \end{cases}
\]
Combined with \eqref{eq wconj} the lemma follows.
\end{proof}
\begin{lemma}\label{lem stab}
Let $w=\rho\set\in \weyl_k[2,M]$ and $m=\inj(g_1,\dots,g_k;h),\,x=\inj(x_1,\dots,x_k,y)\in M$ where $g_i,\,x_i\in \GL_{n_i}(E')$, $i\in [1,k]$ and $h,\,y\in G_r$. Then
\[
m t_w x\overline m^{-1}=t_w \inj(a_1,\dots,a_k;b) \ \ \ \text{where} \ \ \ a_i=\begin{cases} g_{\rho(i)} x_i\bar g_i^{-1} & i\not\in \set \\ g_{\rho(i)}^* x_i\bar g_i^{-1} & i\in \set \end{cases}  \ \ \ \text{and} \ \ \ b=\eta_r^{\odd(\set)}h \eta_r^{-\odd(\set)}y\bar h^{-1}.
\]
\end{lemma}
\begin{proof}
Since $t_w^2$ is in the center of $M$ we have  $mt_w=t_w^2 mt_w^{-1}$ and therefore $m t_w x\overline m^{-1}=t_w (t_wmt_w^{-1})x\overline m^{-1}$. The lemma is therefore straightforward from  \eqref{eq wconj}.
\end{proof}

If $r=0$ in the split even orthogonal case let 
\[
\weyl_k^\circ[2,M]=\{w=\rho\set\in \weyl_k[2,M]: \odd(\set) \ \text{is even}\}. 
\]
For the sake of unified notation set $\weyl_k^\circ[2,M]=\weyl_k[2,M]$ otherwise. 

Recall that for $w=\rho\set\in \weyl_k[2]$ we have
\[
I(w)=\{i\in \set:\rho(i)=i\}.
\]
For $w=\rho\set\in \weyl_k^\circ[2,M]$, $y_i\in Y_{n_i}(E'/(E')^{\sigma\tau},\sgn)$, $i\in I(w)$ and $z\in X_r$ let 
\begin{equation}\label{eq def x}
x_w(\{y_i\}_{i\in I(w)},\,z)=t_w\inj(u_1,\dots,u_k;\eta_r^{\odd(\set)}z) \ \ \ \text{where} \ \ \ u_i=\begin{cases} I_{n_i} & i\not\in I(w) \\ w_{n_i}\bar y_i & i\in I(w). \end{cases}
\end{equation}
\begin{proposition}\label{prop adm orbs}
\begin{enumerate}
\item\label{part orb decomp} The disjoint $M^\circ$-orbit decomposition of $X\cap N_{G^\circ}(M^\circ)$ is given by
\[
X\cap N_{G^\circ}(M^\circ)=\sqcup_{w\in \weyl_k^\circ[2,M]} \sqcup_{\{y_i\}_{i\in I(w)},\,z} M^\circ\cdot x_w(\{y_i\}_{i\in I(w)},\,z)
\]
where $y_i$ ranges over a choice of representatives for the two $\GL_{n_i}(E')$-orbits in $Y_{n_i}(E'/(E')^{\sigma\tau},\sgn)$, $i\in I(w)$ and if $n_0+2r>0$ then $z$ ranges over a set of representatives for the (at most two) $G_r^\circ$-orbits in 
\[
X_r\cap \eta_r^{\odd(\set)}G_r^\circ=\begin{cases} X_r\setminus X_r^\circ & \odd(\set) \ \text{is odd in the orthogonal case}   \\ \ X_r^\circ & \text{otherwise.} \end{cases} 
\]
\item In particular, for $w\in \weyl_k^\circ[2,M]$ the number of $M^\circ$-orbits in $X\cap t_wM^\circ$ is $2^{\abs{I(w)}+\delta_{w,M}}$ where $\delta_{w,M}\in \{0,1\}$ is determined as follows: 
\begin{itemize}
\item In the symplectic case $\delta_{w,M}=0$.
\item In the unitary case $\delta_{w,M}=\begin{cases} 1 & n_0+2r>0 \\ 0 & n_0=r=0.\end{cases}$
\item In the orthogonal case $\delta_{w,M}=0$ if 
\begin{itemize}
\item $n_0\in \{0,1\}$ and $r=0$ or 
\item $n_0=2$, $r=0$ and $\det \j$ is in $-F^{*2}$ if $\odd(\set)$ is even and in $-\imath^2F^{*2}$ if $\odd(\set)$ is odd. 
\end{itemize} Otherwise, $\delta_{w,M}=1$. 
\end{itemize}
\item \label{part stab}The stabilizer of $x_w(\{y_i\}_{i\in I(w)},\,z)$ in $M^\circ$ consists of elements $\inj(g_1,\dots,g_k;h)$ such that
\[
\begin{array}{ll}
g_{\rho(i)}=\bar g_i\in \GL_{n_i}(E') & i\not\in \set,\ \rho(i)\ne i \\
g_i\in \GL_{n_i}(F') & i\not\in\set,\ \rho(i)=i \\
g_{\rho(i)}^*=\bar g_i\in \GL_{n_i}(E') & i\in  \set, \ \rho(i)\ne i \\
\bar g_i\in \U(y_i,(E'/(E')^{\sigma\tau}) & i\in \set,\ \rho(i)=i \\
h\in (G_r^\circ)_z.& 
\end{array}
\]
\end{enumerate}

\end{proposition}
\begin{proof}

Recall that the map $w\mapsto wM^\circ : W_{G^\circ}(M^\circ,M^\circ) \rightarrow N_{G^\circ}(M^\circ)/M^\circ$ is a group isomorphism. 
It follows from the discussion in \S\ref{sec adm} that if $x\in X\cap N_{G^\circ}(M^\circ)$ then $xM=wM$ where $w\in W_{G^\circ}(M^\circ,M^\circ)$ is an involution.
Furthermore, $W_{G^\circ}(M^\circ,M^\circ)\subseteq W_G(M)\cap W_{G^\circ}(M^\circ)$ and  $t_{j_M(w)}\in G^\circ$ for $w\in W_{G^\circ}(M^\circ,M^\circ)$ (see \S \ref{sec sp orth} if $r=0$ in the split orthogonal case). Note that $j_M$ maps the set of involutions in $W_{G^\circ}(M^\circ,M^\circ)$ to $\weyl_k^\circ[2,M]$. We therefore have the disjoint union
\[
X\cap N_{G^\circ}(M^\circ)=\sqcup_{w\in \weyl_k^\circ[2,M]} X \cap t_w M^\circ.
\]
The disjoint decomposition 
\[
X \cap t_w M^\circ=\sqcup_{\{y_i\}_{i\in I(w)},\,z} M^\circ\cdot x_w(\{y_i\}_{i\in I(w)},\,z)
\]
as well as the explication of the stabilizer $M_{x_w(\{y_i\}_{i\in I(w)},\,z)}^\circ$ are immediate from Lemmas \ref{lem in x} and \ref{lem stab} and the facts that $Y_m(E'/(E')^{\sigma\tau},\sgn)$ consists of two $\GL_m(E')$-orbits (see the unitary case in  \S \ref{sec prel}) and $Z_m(E)=\GL_m(E')\cdot I_m$ (see \S\ref{ss Hilbert90}).
The formula for the number of $M^\circ$-orbits in $X\cap t_w M^\circ$ further follows from the classification of $G_r^\circ$-orbits in $X_r^\circ$ and in $X_r\setminus X_r^\circ$ (see  Corollary \ref{cor symp orb} in the symplectic case, Corollary \ref{cor 2orb} in the unitary case and Proposition \ref{prop orth orbs} in the orthogonal case).

\end{proof}

Finally it will be relevant for us to determine, for $x\in X^\circ$, when is the representative \eqref{eq def x} in $G^\circ\cdot x$. 
\begin{lemma}\label{lem in eorb}
Let $x\in X^\circ$, $w=\rho\set\in \weyl_k^\circ[2,M]$, $y_i\in Y_{n_i}(E'/(E')^{\sigma\tau},\sgn)$, $i\in I(w)$ and $z\in X_r\cap \eta_r^{\odd(\set)}G_r^\circ$. Set $x_w=x_w(\{y_i\}_{i\in I(w)},\,z)\in X^\circ$. 
\begin{enumerate}
\item In the symplectic case $x_w\in G^\circ\cdot x$.
\item In the unitary case $x_w\in G^\circ\cdot x$ if and only if 
\[
(-1)^{\odd(\set)} \det z\prod_{i\in I(w)} \det y_i^{\tau-1}\in \det x\,\Gamma.
\]
(Here, if $n_0=r=0$ replace $\det z$ by one. See \eqref{eq def gamma} for the definition of $\Gamma$.)
\item In the orthogonal case
\begin{itemize}
\item for $x_1,\,x_2\in X^\circ$ we have $G^\circ \cdot x_1=G^\circ \cdot x_2$ if and only if $\hbar^{\j[n]}(x_1)=\hbar^{\j[n]}(x_2)$;
\item $x_w\in G^\circ\cdot x$ if and only if 
\[
((-1)^{r\odd(\set)+{N(\rho\set)\choose 2}}(2\det \j)^{\odd(\set)}\prod_{j\in I(w)} (\det y_j,\imath^2)_F \frac{\hbar^{\j[r]}(z)}{\hbar^{\j[r]}(I_{n_0+2r})}=\frac{\hbar^{\j[n]}(x)}{\hbar^{\j[n]}(I_N)}.
\]
(See \S \ref{ss xinv} for the definition of the invariant $\hbar^{\j[r]}$. Here if $n_0=0$ replace $\det \j$ by one and if $n_0+2r=0$ replace the quotient on left hand side of the equality by one.)

\end{itemize}

\end{enumerate} 
\end{lemma}
\begin{proof}
In the symplectic case the result of the lemma is immediate from Corollary \ref{cor symp orb}. In the unitary case it is immediate from Corollary \ref{cor 2orb} and the observation that 
\[
\det x_w(\{y_i\}_{i\in I(w)};z)=(-1)^{\odd(\set)} \det z\prod_{i\in I(w)} \det y_i^{\tau-1}. 
\]
The orthogonal case is more elaborate and requires a careful computation of invariants. We point out all the ingredients of our computation. 

It follows from Lemma \ref{lem orth orb} \eqref{part y} that $\partial^{\j[n]}(x_1)=\partial^{\j[n]}(x_2)$. 
The first part immediately follows (see \S \ref{ss xinv}).
For the second part, it follows that $x_w\in G^\circ\cdot x$ if and only if $\hbar^{\j[n]}(x)=\hbar^{\j[n]}(x_w)$. 

We now explicate the computation of $\hbar^{\j[n]}(x_w)$.
Let $z_w\in \GL_N(E)$ be such that $x_w=z_w\cdot I_N$ and $y_w=\j[n]\star z_w$. By definition
$\hbar^{\j[n]}(x_w)=\Hasse_F(y_w)$. 

We make the following observations. Let $\zeta_m =\sm{I_m}{\imath I_m}{I_m}{-\imath I_m}$ and $\upsilon_m=w_{2m}\star \zeta_m$. We observe that
\[
\zeta_m\cdot I_{2m}=\sm{}{I_m}{I_m}{}
\ \ \ \text{and} \ \ \
\upsilon_m=\diag(2w_m,-2\imath^2w_m).
\]
Note further that 
\[
w_{4m} \star \diag(\zeta_{2m},\zeta_{2m})=\sm{}{\eta_m}{\eta_m}{}
\]
and that $\eta_{2m},\,\sm{}{\eta_m}{\eta_m}{}\in w_{4m}\star \GL_{4m}(F)=I_{4m}\star \GL_{4m}(F)$.

For $y\in Y_m(E/F,1)$ let $x(y)=\sm{}{y^{-1}w_m}{w_m \bar y}{}\in X^{w_{2m}}$, $z(y)=\sm{I_m}{\imath w_m}{w_m\bar y}{-\imath w_m \bar y w_m}\in \GL_{2m}(E)$ and $\eta(y)=w_{2m}\star z(y)$.
Note that $z(y)\cdot I_{2m}=x(y)$ and therefore $\eta(y)\in Y_{2m}(F/F,1)$. If we assume further that $y\in Y_m(F/F,1)=Y_m(E/F,1)\cap \GL_m(F)$ then 
\[
\eta(y)=\diag(2y,-2\imath^2 w_m yw_m).
\]
We can deduce that
\[
\frac{\det \eta(y)}{\det w_{2m}}\in \imath^{2m}F^{*2}
\]
and
\[
\frac{\Hasse_F(\eta(y))}{\Hasse_F(w_{2m})}=(\det y,\imath^2)_F(2,\imath^2)_F^m(-1,\imath^2)_F^{m\choose 2}.
\]
Indeed, note that $\eta(y)\star \GL_{2m}(F)$ (hence $\det \eta(y)F^{*2}$ and $\Hasse_F(\eta(y))$) depends only on $y\star \GL_m(E)$ and therefore we can replace $y$ with $\diag(I_{m-1},\det y)$. In this case $\eta(y)$ is diagonal and its invariants are easy to compute. We further recall that $\det w_{2m}=(-1)^m$.

Based on the above observations we have that $\diag(u_1,\dots,u_k,u)\in y_w\star \GL_N(F)$ where
\[
u_i=\begin{cases} \eta(y_i) & i\in I(w) \\ I_{2n_i} & \text{otherwise} \end{cases}\ \ \ \text{and}\ \ \ u=\j[r]\star \zeta \text{ where } z=\zeta\cdot I_{n_0+2r}.
\]
It therefore follows that
\begin{multline*}
\frac{\Hasse_F(y_w))}{\Hasse_F(\j[n])}=\frac{\hbar^{\j[r]}(z)}{\Hasse_F(\j[r])} \left[\prod_{i\in I(w)}\frac{(\partial^{\j[r]}(z), \det \eta(y_i))_F}{(\det \j[r],\det w_{2n_i})_F} \frac{\Hasse_F(\eta(y_i))}{\Hasse_F(w_{2n_i})} \right] \times \\ \prod_{\{i,\,j\in I(w): i<j\}} \frac{(\det \eta(y_i),\det \eta(y_j))_F}{(\det w_{2n_i},\det w_{2n_j})_F}.
\end{multline*} 
Let $\odd(w)=\odd_M(w)$ be the number of $i\in I(w)$ such that $n_i$ is odd. 
If $n_0+2r>0$ then $\det \eta_r=-1$ and therefore $\det z=(-1)^{\odd(\set)}$.
It follows (see \S \ref{ss xinv}) that
\[
\frac{\partial^{j[r]}(z)}{\det \j[r]}\in \imath^{2\odd(\set)}F^{*2}.
\]

It follows from the above computations that
\[
\prod_{\{i,\,j\in I(w): i<j\}} \frac{(\det \eta(y_i),\det \eta(y_j))_F}{(\det w_{2n_i},\det w_{2n_j})_F}=(-1,\imath^2)_F^{\odd(w)\choose 2}
\]
and 
\[
\prod_{i\in I(w)}\frac{\Hasse_F(\eta(y_i))}{\Hasse_F(w_{2n_i})}=(2,\imath^2)_F^{\odd(w)}\prod_{j\in I(w)}(-1,\imath^2)_F^{n_j\choose 2} (\det y_j,\imath^2)_F.
\]
Recall that $N(w)=N_M(w)=\sum_{i\in I(w)} n_i$ and note that
\[
\sum_{j\in I(w)} {n_j\choose 2}\equiv {N(w)\choose 2} +{\odd(w) \choose 2} (\mod 2).
\]
It follows that 
\[
\prod_{\{i,\,j\in I(w): i<j\}} \frac{(\det \eta(y_i),\det \eta(y_j))_F}{(\det w_{2n_i},\det w_{2n_j})_F} \times \prod_{i\in I(w)}\frac{\Hasse_F(\eta(y_i))}{\Hasse_F(w_{2n_i})}=(2,\imath^2)_F^{\odd(w)}(-1,\imath^2)_F^{N(w)\choose 2} (\prod_{j\in I(w)} \det y_j,\imath^2)_F.
\]
Recall further that $\det \j[r]=(-1)^r\det \j$. Hence
\[
\prod_{i\in I(w)}\frac{(\partial^{\j[r]}(z), \det \eta(y_i))}{(\det \j[r],\det w_{2n_i})_F} =(-1,\imath^2)_F^{r \odd(w)}(\det j,\imath^2)_F^{\odd(w)}.
\]
We note further that $\odd(w)\equiv \odd(\set)(\mod 2)$. We conclude that
\[
\frac{\Hasse_F(y_w))}{\Hasse_F(\j[n])}=((-1)^r2\det \j,\imath^2)_F^{\odd(\set)}(-1,\imath^2)_F^{N(w)\choose 2} (\prod_{j\in I(w)} \det y_j,\imath^2)_F\frac{\hbar^{\j[r]}(z)}{\Hasse_F(\j[r])}.
\]
The lemma follows. 
\end{proof}


\subsection{The contribution of an admissible orbit}

Let 
\[
\Delta(\GL_m(E'))=\{(g,g):g\in \GL_m(E')\}
\] 
be the image of the diagonal imbedding of $\GL_m(E')$ in $\GL_m(E')\times \GL_m(E')$.
The following is an immediate consequence of the explication of the stabilizers in Proposition \ref{prop adm orbs} \eqref{part stab}.

\begin{corollary}\label{cor dist adm}
Let $\alpha=(n_1,\dots,n_k;r)\in A$ and $M=M_\alpha$. Let $\pi_i$ be a representation of $\GL_{n_i}(E')$, $i\in [1,k]$ and $\pi_0$ a representation of $G_r^\circ$. 
Let $w=\rho\set\in \weyl_k^\circ[2,M]$,  $y_i\in Y_{n_i}(E'/(E')^{\sigma\tau},\sgn)$, $i\in I(w)$, $z\in X_r\cap \eta_r^{\odd(\set)}G_r^\circ$  and $x_w(\{y_i\}_{i\in I(w)},z)\in X^\circ$ be defined as in \eqref{eq def x}.  The representation $(\pi_1\otimes\cdots\otimes \pi_k\otimes \pi_0)\circ \inj_\alpha$ of $M^\circ$ is $M_{x_w(\{y_i\}_{i\in I(w)},z)}^\circ$-distinguished if and only if 
\[
\begin{array}{ll}
\pi_i\otimes \bar\pi_{\rho(i)} \text{ is }\Delta(\GL_{n_i}(E'))\text{-distinguished}& i\not\in \set,\ \rho(i)\ne i \\
\pi_i \text{ is }\GL_{n_i}(F')\text{-distinguished} & i\not\in\set,\ \rho(i)=i \\
\pi_i\otimes \bar\pi_{\rho(i)}^* \text{ is }\Delta(\GL_{n_i}(E'))\text{-distinguished}& i\in  \set, \ \rho(i)\ne i \\
\pi_i \text{ is }\U(y_i,E'/(E')^{\sigma\tau})-\text{distinguished} & i\in \set,\ \rho(i)=i \\
\pi_0 \text{ is } (G_r^\circ)_z-\text{distinguished.}& 
\end{array}
\]\qed
\end{corollary}
\begin{remark}\label{rmk gk}
For irreducible representations $\pi_1$ and $\pi_2$ of $\GL_m(E')$ we clearly have that $\pi_1\otimes \pi_2$ is $\Delta(\GL_{n_i}(E'))$-distinguished if and only if $\pi_2\simeq \pi_1^\vee$. Furthermore, for an irreducible representation $\pi$ of $\GL_m(E')$ we have $\pi^*\simeq (\pi^\tau)^\vee$ (see \cite{MR0404534}). 
As a consequence, in the above lemma (and in its notation) if $\rho(i)\ne i$ and $\pi_i$ and $\pi_{\rho(i)}$ are irreducible then $\pi_i\otimes \bar\pi_{\rho(i)} $ is $\Delta(\GL_{n_i}(E'))$-distinguished if and only if $\pi_{\rho(i)}\simeq \bar\pi_i^\vee$ and $\pi_i\otimes \bar\pi_{\rho(i)} ^*$ is $\Delta(\GL_{n_i}(E'))$-distinguished if and only if $\pi_{\rho(i)}\simeq \pi_i^{\sigma\tau}$.
\end{remark}
\begin{remark}\label{rmk flo}
For an irreducible cuspidal representation $\pi$ of $\GL_m(E')$, it follows from \cite[Theorem 6.1 and Theorem 13.1]{MR2930996} that the following are equivalent:
\begin{enumerate}
\item $\pi\simeq \pi^{\sigma\tau}$
\item $\pi$ is $U(y,E'/(E')^{\sigma\tau})$-distinguished for some $y\in Y_m(E'/(E')^{\sigma\tau})$
\item $\pi$ is $U(y,E'/(E')^{\sigma\tau})$-distinguished for all $y\in Y_m(E'/(E')^{\sigma\tau})$.
\end{enumerate}
\end{remark}

\subsection{Explication of Corollary \ref{cor cusp dist} for classical symmetric pairs}

\begin{theorem}\label{thm dist ind class}
Let $x\in X^\circ$, $\pi_i$ an irreducible cuspidal representation of $\GL_{n_i}(E')$, $i\in [1,k]$ and $\pi_0$ a cuspidal representation of $G_r^\circ$. The representation 
\[
\pi_1\times\cdots \times \pi_k\ltimes \pi_0
\]
of $G^\circ$ is $G^\circ_x$-distinguished if and only if there exist an involution $\rho\in S_k$ and a subset $\set$ of $[1,k]$ such that $w=\rho\set\in \weyl_k^\circ[2,M]$ and there exist $y_i\in Y_{n_i}(E'/(E')^{\sigma\tau},\sgn)$, $i\in I(w)$ and $z\in X_r\cap \eta_r^{\odd(\set)}G_r^\circ$ such that $x_w(\{y_i\}_{i\in I(w)},\,z)\in G^\circ\cdot x$ (see Lemma \ref{lem in eorb}) and furthermore
\[
\begin{array}{ll}
\pi_{\rho(i)}\simeq\bar \pi_i^\vee & i\not\in \set,\ \rho(i)\ne i \\
\pi_i \text{ is }\GL_{n_i}(F')\text{-distinguished} & i\not\in\set,\ \rho(i)=i \\
\pi_{\rho(i)}\simeq \pi_i^{\tau\sigma} & i\in  \set \\
\pi_0\text{ is }(G_r^\circ)_z-\text{distinguished}. 
\end{array}
\]
\end{theorem}
\begin{proof}
It follows from Lemma \ref{lem kappa} that we may assume without loss of generality that $P=P_\alpha$ for some $\alpha=(n_1,\dots,n_k;r)\in A$ (with $r\ne 1$ in the split even orthogonal case). As a consequence of Corollary \ref{cor cusp dist} we have that $\pi_1\times\cdots\times\pi_k\ltimes \pi_0$ is $G^\circ_x$-distinguished if and only if there exists $x'\in G\cdot x \cap N_{G^\circ}(M^\circ)$ such that $(\pi_1\otimes\cdots\otimes \pi_k\otimes \pi_0)\circ \inj_\alpha$ is $M^\circ_{x'}$-distinguished. Note that for a representation $\upsilon$ of $M^\circ$ and $m\in M^\circ$ we have that  $\ell\mapsto \ell\circ \upsilon(m)$ is an isomorphism from $\Hom_{M^\circ_{m\cdot x'}}(\upsilon,\triv)$ to $\Hom_{M^\circ_{x'}}(\upsilon,\triv)$ and in particular, $\upsilon$ is $M^\circ_{x'}$-distinguished if and only if it is $M^\circ_{m\cdot x'}$-distinguished. The theorem is therefore now immediate from Proposition \ref{prop adm orbs}\eqref{part orb decomp}, Lemma \ref{lem in eorb}, Corollary \ref{cor dist adm} and Remarks \ref{rmk gk} and \ref{rmk flo}.
\end{proof}

\begin{remark}\label{rem existence dist cusp local} 
The statement of Theorem \ref{thm dist ind class} is not empty. It is proved in \cite{BP} that there exist representations $\pi_i$, $i=0,1,\dots,k$ as in Theorem \ref{thm dist ind class} that are distinguished in the stated manner.
In fact, more generally it is proved in \cite{BP} that if $(G,H)$ is a Galois pair with $G$ quasi-split then there exists a non-degenerate character $\psi$ of the unipotent radical $N$ of a Borel subgroup of $G$, trivial on $N\cap H$, such that there exists an irreducible cuspidal representation of $G$ that is simultaneously $H$-distinguished and $\psi$-generic. Moreover, if $H$ is quasi-split, this holds for any such choice of $\psi$.
\end{remark}

\subsection{A necessary condition for distinction}
As a consequence of Theorem \ref{thm dist ind class} we obtain a necessary condition for distinction that is easier to formulate and is useful for applications. 

We continue to assume that $n_1+\cdots+n_k+r=n$. For $w=\rho\set\in\weyl_k$ and a tuple $\pi=(\pi_1,\dots,\pi_k;\pi_0)$ such that $\pi_0$ is a representation of $G_r^\circ$ and $\pi_i$ is a representation of $\GL_{n_i}(E')$ let
\begin{equation}\label{eq w on pi}
w\pi=(\pi_1',\dots,\pi_k';\pi_0) \ \ \ \text{where} \ \ \ \pi_i'=\begin{cases} \pi_{\rho^{-1}(i)} &i\not\in\set \\  \pi^*_{\rho^{-1}(i)}  & i\in \set.\end{cases}
\end{equation}

\begin{corollary}\label{cor dist inv}
Let $x\in X^\circ$, $\pi_i$ an irreducible cuspidal representation of $\GL_{n_i}(E')$, $i\in [1,k]$ and $\pi_0$ a cuspidal representation of $G_r^\circ$. If the representation 
\[
\pi_1\times\cdots\times\pi_k\ltimes \pi_0
\]
of $G^\circ$ is $G_x^\circ$-distinguished and $w=\rho\set\in \weyl_k^\circ[2,M]$ satisfies \eqref{eq rhocond} (such $w$ is guaranteed by Theorem \ref{thm dist ind class}) then 
\[
w(\pi_1,\dots,\pi_k;\pi_0)=(\bar\pi_1^\vee,\dots,\bar\pi_k^\vee;\pi_0).
\]
\end{corollary}
\begin{proof}
Write $w\pi=(\pi_1',\dots,\pi_k';\pi_0)$ and recall that for an irreducible representation $\upsilon$ of $\GL_m(E')$ we have $\upsilon^*\simeq(\upsilon^\vee)^\tau$ by \cite{MR0404534}.
The identity $\pi_i'=\bar\pi_i^\vee$ is therefore clear whenever either $i\in\set$ or $\rho(i)\ne i$. For $i\not\in \set$ such that $\rho(i)=i$ it further follows from \cite[Proposition 12]{MR1111204} and the corollary follows. 
\end{proof}

\subsection{Completion of the proof of Theorem \ref{thm Hdist ind class}} Note that for $w=\rho\set \in\weyl_k^\circ[2,M]$ if $I(w)$ is empty then $\odd(\set)$ is even and $N(w)=0$. The theorem is therefore a consequence of Theorem \ref{thm dist ind class} with $x=I_N$ and Lemmas \ref{lem exp stab} and \ref{lem in eorb}.
\qed

\section{Distinction and conjugate-self-duality in the split odd orthogonal case}\label{sec selfdual}

It was recently proved by Beuzart-Plessis \cite{BP} that for any quasi-split reductive group $\mathbb{H}$ defined over $F$ and irreducible, generic representation $\pi$ of $\mathbb{H}(E)$ that is $\mathbb{H}(F)$-distinguished we have $\pi^\sigma=\pi^\vee$. 

When $\mathbb{H}$ is the split odd special orthogonal group we apply our previous results in order to prove a global analogue as well as several local generalizations of this result. In particular, we show that in the above statement $\mathbb{H}(F)$-distinction can be replaced by $\mathbb{H'}(F)$-distinction where $\mathbb{H'}$ is a non-quasi-split inner form of $\mathbb{H}$ such that $\mathbb{H'}(E)=\mathbb{H}(E)$. 

Let $\SO_{2n+1}=\SO(w_{2n+1})$ be the split odd special orthogonal group. Over any field $k$ we have
\[
\SO_{2n+1}(k)=\{g\in \SL_{2n+1}(k): {}^t gw_{2n+1} g=w_{2n+1}\}. 
\]
The following is well known. Since we do not know a proper reference we include the argument.
\begin{observation}\label{obs}
\begin{enumerate}
\item For a $p$-adic field $F$ and an irreducible, generic representation $\pi$ of $\SO_{2n+1}(F)$ we have $\pi^\vee\simeq \pi$. 
\item For a number field $k$ and an irreducible, generic, cuspidal automorphic representation $\Pi$ of $\SO_{2n+1}(\A_k)$ we have $\Pi^\vee=\Pi$. 
\end{enumerate}
\end{observation}
\begin{proof}
Consider first the local statement and assume further that $\pi$ is unramified. Then, $\pi$ is fully induced from a character $\chi$ of the standard Borel subgroup. Since $\chi^{-1}$ is in the Weyl orbit of $\chi$ it follows that $\pi$ is self-dual.  

In the global setting, this implies that $\Pi_v$ is self-dual for almost all places $v$ of $k$. The global statement is therefore immediate from the rigidity theorem of Jiang and Soudry \cite[Theorem 5.3]{MR1983781}.

Now a simple globalization argument such as \cite[Proposition 5.1]{MR1168488} implies that every irreducible, generic and cuspidal representation $\pi$ of $\SO_{2n+1}(F)$ is self-dual. 

For the general local case, in terms of its cuspidal support, $\pi^\vee$ is a subquotient of a representation induced from an irreducible, cuspidal and generic representation of a standard Levi subgroup. That is, there exists a decomposition $n=n_1+\cdots+n_k+r$, irreducible cuspidal representations $\pi_i$ of $\GL_{n_i}(F)$, $i\in [1,k]$ and an irreducible, generic and cuspidal representation $\pi_0$ of $G_r^\circ$ such that $\pi^\vee$ is a subquotient of $\pi_1\times \cdots\times \pi_k \ltimes \pi_0$.  It follows from the cuspidal case that $\pi_0$ is self-dual. Thus both $\pi$ and $\pi^\vee$ are subquotients of $\pi_1^\vee\times \cdots\times \pi_k^\vee \ltimes \pi_0^\vee$. However,  by a well known result of Rodier \cite{MR0354942}, $\pi_1^\vee\times \cdots\times \pi_k^\vee \ltimes \pi_0^\vee$ has a unique irreducible, generic subquotient and it follows that $\pi$ is self-dual.
\end{proof}

Let $\mathbb{G}$ be a semi-simple algebraic group and $\mathbb{H}$ a reductive subgroup of $\mathbb{G}$ both defined over a number field $k$. It follows from \cite[Proposition 1]{MR1233493} that the period integral 
\[
\int_{\mathbb{H}(k)\bs \mathbb{H}(\A_k)}\phi(h)\ dh
\] 
is absolutely convergent for every  cuspidal automorphic form $\phi$ on $\mathbb{G}(\A_k)$.
A cuspidal automorphic representation $\Pi$ of $\mathbb{G}(\A_k)$ is called $\mathbb{H}$-distinguished if this period integral is non-vanishing for some cusp form $\phi$ in the space of $\Pi$. 
\begin{theorem}\label{thm global}
Let $K/k$ be a quadratic extension of number fields with Galois action $\sigma$, $\mathbb{G}=\Res_{K/k}(\SO_{2n+1})$, $\mathbb{X}=\{x\in \mathbb{G}: x^\sigma=x^{-1}\}$ and $g\cdot x=gxg^{-\sigma}$ the action of $\mathbb{G}$ on $\mathbb{X}$ by twisted conjugation. Denote by $\mathbb{G}_x$ the stabilizer of $x\in \mathbb{X}(k)$ in $\mathbb{G}$, an algebraic group defined over $k$. Let $\Pi$ be an irreducible cuspidal generic representation of $\mathbb{G}(\A_k)=\SO_{2n+1}(\A_K)$ that is $\mathbb{G}_x$-distinugished. Then $\Pi^\sigma\simeq \Pi^\vee$.
\end{theorem}
\begin{proof}
Write $\Pi=\otimes_v \Pi_v$ where the restricted tensor product is over all places $v$ of $k$. Then $\Pi_v$ is $\mathbb{G}_x(k_v)$-distinguished for all places $v$ of $F$.

If $v$ is split in $K$ (that is $K_v\simeq k_v\times k_v$) then we may identify $\mathbb{G}(K_v)$ with $\SO_{2n+1}(k_v)\times \SO_{2n+1}(k_v)$ and $\sigma$ acts on it by interchanging the coordinates. Then $x_v=(x_0,x_0^{-1})$ for some $x_0\in \SO_{2n+1}(k_v)$ and 
\[
\mathbb{G}_x(k_v)=\{(g,x_0^{-1}gx_0): g\in \SO_{2n+1}(k_v)\}.
\]
Write $\Pi_v=\pi_1\otimes \pi_2$ for irreducible representations $\pi_1$ and $\pi_2$ of $\SO_{2n+1}(k_v)$. Since $\Pi_v$ is $\mathbb{G}_x(k_v)$-distinguished we see that $\pi_2\simeq\pi_1^\vee$, that is, $\Pi_v^\sigma\simeq\Pi_v^\vee$.

Since $\GL_{2n+1}$ has trivial Galois cohomology, there exists $z\in \GL_{2n+1}(K)$ such that $x=zz^{-\sigma}$. Note that $y={}^t zw_{2n+1} z$ is a symmetric matrix in $\GL_{2n+1}(k)$. Since $\det x=1$ we have $\det y\, k^{*2}=\det w_{2n+1}\,k^{*2}$. Since the Hasse invariant of $y$ is one at almost all places it follows that $y$ and $w_{2n+1}$ are in the same $\GL_{2n+1}(k_v)$-orbit for almost all places $v$ of $k$. 
It follows from Lemma \ref{lem orb iso} that $\mathbb{G}_x(k_v)$ is $\SO_{2n+1}(K_v)$-conjugate to $\SO_{2n+1}(k_v)$ for almost all places $v$ of $k$ that are inert in $K$. 
Let $v$ be a finite and inert place of $k$ such that $\Pi_v$ is unramified. Since $\Pi$ is generic we can write $\Pi_v\simeq\chi_1\times\cdots \times\chi_n\ltimes\triv_0$ 
as irreducibly induced from a character of the Borel subgroup. Thus, $\Pi_v\simeq\chi_{\rho(1)}'\times\cdots \times\chi_{\rho(n)}'\ltimes\triv_0$ whenever $\chi_i'\in \{\chi_i,\chi_i^{-1}\}$, $i\in [1,n]$ and $\rho\in S_n$. It now follows from Corollary \ref{cor dist inv} that $\Pi_v^\sigma=\Pi_v^\vee$.

We conclude that $\Pi_v^\sigma=\Pi_v^\vee$ for almost all places $v$ of $k$ and the theorem now follows from the rigidity theorem of Jiang and Soudry \cite[Theorem 5.3]{MR1983781}.
\end{proof}
For the rest of this section we go back to the notation of \S \ref{section Galois distinction and induction} in the split odd orthogonal setting. Thus $G^\circ=\SO_{2n+1}(E)$ and $X^\circ=\{x\in G^\circ: x^{-1}=x^\sigma\}$. Since $X=X^\circ\sqcup(-X^\circ)$ we will use without further mention the fact that $G^\circ_x=G^\circ_{-x}$, $x\in X$.

A representation $\pi$ of $\GL_m(E)$ is called essentially tempered if $\abs{\det}^{-e} \otimes \pi$ is tempered for some real number $e$. The number $e$ is uniquely determined by $\pi$ and we write $e(\pi)=e$. Every irreducible and cuspidal representation of $\GL_m(E)$ is essentially tempered.

\begin{theorem}\label{thm local sodd}
Let $\pi$ be an irreducible, generic representation of $G^\circ$. If $\pi$ is $G_x^\circ$-distinguished for some $x\in X$ then $\pi^\sigma\simeq\pi^\vee$.
\end{theorem}
\begin{proof}
If $\pi$ is cuspidal this is immediate from Theorem \ref{thm global} and the globalization of distinguished, generic representations in \cite{BP}. We apply the cuspidal case in order to complete the proof of the theorem for a general $\pi$. 

In terms of its cuspidal support we realize $\pi$ as a quotient of a representation induced from an irreducible, generic and cuspidal representation of a standard Levi subgroup of $G^\circ$. That is, there exists a decomposition $n=n_1+\cdots+n_k+r$, irreducible cuspidal representations $\pi_i$ of $\GL_{n_i}(E)$, $i\in [1,k]$ and an irreducible, generic and cuspidal representation $\pi_0$ of $G_r^\circ$ such that $\pi^\vee$ is a quotient of $I=\pi_1\times \cdots\times \pi_k \ltimes \pi_0$.  Hence $I$ is also $G_x^\circ$-distinguished.
Furthermore, $\pi$ is the unique generic irreducible subquotient of $I$ by \cite{MR0354942}.

By Theorem \ref{thm dist ind class} the representation $\pi_0$ is $(G_r^\circ)_z$-distinguished for some $z\in X_r$ and by the cuspidal case of this theorem we conclude that $\pi_0^\sigma\simeq \pi_0^\vee$.  It further follows from Corollary \ref{cor dist inv} that $(\pi_1,\dots,\pi_k)=((\pi_{\rho(1)}')^\sigma,\dots, (\pi_{\rho(k)}')^\sigma)$ for some involution $\rho\in S_k$ and some choice $\pi_i'\in\{\pi_i,\pi_i^\vee\}$, $i\in [1,k]$. Since $(\pi_{\rho(1)}')^{\sigma} \times \cdots \times (\pi_{\rho(k)}')^{\sigma}\ltimes (\pi_0^\vee)^{\sigma}$ has the same decomposition series as $(I^\vee)^{\sigma}$ we conclude that $I$ and $(I^\vee)^{\sigma}$ have the same decomposition series. Thus, $(\pi^\vee)^{\sigma}$ is a generic subquotient of $I$ and therefore $\pi^\sigma=\pi^\vee$.  
\end{proof}

\begin{remark} The existence of generic representations as in the statement of Theorem \ref{thm local sodd} follows either from Remark \ref{rem existence dist cusp local} in the cuspidal case, or from 
Theorem \ref{thm dist ind class} applied to irreducible principal series induced from characters of the upper triangular Borel subgroup. We also note that the existence of cuspidal generic distinguished automorphic representations as in the statement of Theorem \ref{thm global}  follows from the existence of local distinguished representations (Remark \ref{rem existence dist cusp local}, a result of Beuzart-Plessis) and Beuzart-Plessis' globalization theorem in \cite{BP}. 
\end{remark}

Recall that by generalized injectivity \cite{MR3986415}, if a standard module is generic and reducible then its Langlands quotient is not generic. We now extend our result and show that for a certain family of irreducible, degenerate representations distinction implies Galois invariance. 
\begin{proposition}\label{prop standinv}
Let $I$ be a standard module of $G^\circ$ of the form $I=\pi_1\times \cdots\times \pi_s \ltimes T$ where $\pi_i$ is irreducible cuspidal for $i\in [1,s]$ and such that $e(\pi_1)\ge \cdots \ge e(\pi_s)>0$ and $T$ is irreducible, generic and has unitary cuspidal support (in particular, $T$ is tempered). If either $I$ or $I^\vee$ is  $G_x^\circ$-distinguished for some $x\in X$ then $I^\sigma=I$. Consequently, if the Langlands quotient of $I$ is $G_x^\circ$-distinguished for some $x\in X$ then it is Galois invariant.
\end{proposition}
\begin{proof}
The representation $T$ is the unique generic direct summand of a semi-simple representation of the form $\pi_{s+1}\times\cdots\times \pi_k \ltimes \pi_0$ where $\pi_0$ is an irreducible, generic and cuspidal representation of $G_r^\circ$ for some $r$ and $\pi_i$ is irreducible cuspidal for $i\in [s+1,k]$. Therefore $I$ is a quotient of $\Pi=\pi_1\times\cdots\times \pi_k \ltimes \pi_0$ and $I^\vee$ is a quotient of $\Pi^\vee$.

If $I$ (resp. $I^\vee$) is $G_x^\circ$-distinguished then so is $\Pi$ (resp. $\Pi^\vee$). It follows from Theorem \ref{thm dist ind class} that $\pi_0$ (resp. $\pi_0^\vee$) is $(G_r^\circ)_z$-distinguished for some $z\in X_r$ and from Theorem \ref{thm local sodd} and Observation \ref{obs} that $\pi_0^\sigma=\pi_0$. It further follows from Corollary \ref{cor dist inv}, the inequalities $e(\pi_1)\ge \cdots \ge e(\pi_s)>0$ and the equalities $e(\pi_{s+1})=\cdots=e(\pi_k)=0$ that we have the following equalities of multi-sets of cuspidal representations
\[
\{\pi_1,\dots,\pi_s\}=\{\pi_1^\sigma,\dots,\pi_s^\sigma\} \ \ \ \text{and} \ \ \ \{\pi_{s+1}',\dots,\pi_k'\}=\{\pi_{s+1}^\sigma,\dots,\pi_k^\sigma\}
\]
for some choice of $\pi_i'\in \{\pi_i,\pi_i^\vee\}$, $i\in[s+1,k]$. By well known results on standard modules for the general linear group over $E$ the first identity implies the equality of the induced representations 
\[
\pi_1\times \cdots \times \pi_s=\pi_1^\sigma \times \cdots \times \pi_s^\sigma
\]
while the second identity combined with the fact that $\pi_0$ is Galois invariant  imply the equality of the semi-simple representations
\[
\pi_{s+1}\times \cdots \times \pi_k\ltimes \pi_0=\pi_{s+1}^\sigma \times \cdots \times \pi_k^\sigma \ltimes \pi_0^\sigma.
\]
Since $T^\sigma$ is a generic summand of the right hand side we conclude that $T=T^\sigma$. All together, this implies that $I$ is Galois invariant.
 
The second part of the proposition follows from the first part and the facts that Galois invariance of $I$ implies Galois invariance of its Langlands quotient and $G_x^\circ$-distinction of the Langlands quotient implies $G_x^\circ$-distinction of $I$.
\end{proof}
\begin{remark}
The proof of Proposition \ref{prop standinv} uses Theorem \ref{thm local sodd} only in the case where $\pi$ is cuspidal. In fact, the proposition can be used as an alternative way to reduce Theorem \ref{thm local sodd} to the cuspidal case as follows. Let $\pi$ be as in Theorem \ref{thm local sodd}. We freely use the notation of the proof of the proposition. We may realize $\pi^\vee$ as the unique generic irreducible subquotient of a representation of the form $\Pi$. We may now define $T$ to be the unique generic summand of the semi-simple representation $\pi_{s+1}\times\cdots\times \pi_k \ltimes \pi_0$ so the standard module $I$ is defined as in the proposition. Thus, $\pi^\vee$ is the unique generic irreducible subquotient of $I$ and by generalized injectivity \cite{MR3986415}, $\pi^\vee$ is a subrepresentation of $I$. Consequently, $\pi$ is a quotient of $I^\vee$ and therefore $I^\vee$ is $G_x^\circ$-distinguished. It therefore follows from the proposition that $I$ is Galois invariant. Hence$(\pi^\vee)^\sigma$ is an irreducible generic subquotients of $I$ and therefore equals $\pi^\vee$. It follows that $\pi$ is Galois invariant.  
\end{remark}

\section{On distinction and local Langlands correspondence}\label{s llc}

In this section we identify a certain family of irreducible representations $\pi$ of the classical group $G^\circ$ (in the notation of \S  \ref{section Galois distinction and induction}), parabolically induced from cuspidal, such that their functorial transfer to $\GL_N(E')$ is also irreducibly induced from cuspidal and Galois distinction of $\pi$ implies Galois distinction of its transfer. We only consider quasi-split classical groups and the family of representations will contain, in particular, all generic principal series representations (that is, irreducibly induced from a character of the Borel subgroup).

The results, namely Theorem \ref{thm dist and trans}, are simple applications of the necessary condition in Theorem  \ref{thm dist ind class}. However, we make the relation to Prasad's conjectures in \cite{DP} explicit and this requires lengthy and rather technical computations.

\subsection{On Galois distinction for $\GL_N$}\label{ss gl dist}

For $m=m_1+\cdots+m_t$ and representations $\pi_i$ of $\GL_{m_i}(E')$, $i\in [1,t]$ we denote by $\pi_1\times \cdots\times \pi_t$ the representation of $\GL_m(E')$ parabolically induced from $\pi_1\otimes \cdots \otimes \pi_t$ (with the usual normalized induction).
If $\pi_i$ is $\GL_{m_i}(F')$-distinguished for all $i\in [1,t]$ then $\pi_1\times \cdots\times \pi_t$ is $\GL_m(F')$-distinguished by a closed orbit argument (see e.g. \cite[Proposition 26]{MR1194271}). Furthermore, for any representation  $\pi$ of $\GL_m(E')$ of finite length, the representation $\pi^\vee\times \pi^\sigma$ of $\GL_{2m}(E')$ is $\GL_{2m}(F')$-distinguished by an open orbit argument (see e.g. \cite[Proposition 2.3]{MR3421655}). 

For a representation  $\pi$ of $\GL_m(E')$ and a character $\chi$ of $(E')^*$ let $\pi\chi=\pi\otimes (\chi\circ \det)$. 
Let $\eta_{E'/F'}$ be the unique character of $(F')^*$ with kernel, the index two subgroup $N(E'/F')$, and let $\eta$ be any extension of $\eta_{E'/F'}$ to a character of $(E')^*$. In particular, $\eta^\sigma=\eta^{-1}$. Clearly, $\pi$ is $\GL_m(F')$-distinguished if and only if $\pi\eta$ is $(\GL_m(F'),\eta_{E'/F'}\circ \det)$-distinguished.
Since, with the above notation, $(\pi_1\times \cdots\times \pi_t)\eta\simeq (\pi_1 \eta)\times \cdots\times (\pi_t \eta)$, the closed orbit argument also implies that if 
 $\pi_i$ is $(\GL_{m_i}(F'),\eta_{E'/F'}\circ \det)$-distinguished for all $i\in [1,t]$ then $\pi_1\times \cdots\times \pi_t$ is $(\GL_m(F'),\eta_{E'/F'}\circ \det)$-distinguished. Similarly,   $(\pi^\vee \times \pi^\sigma)\eta\simeq (\pi\eta^{-1})^\vee \times (\pi\eta^{-1})^\sigma$ and therefore, the open orbit argument also implies that for any representation  $\pi$ of $\GL_m(E')$ of finite length, the representation $\pi^\vee\times \pi^\sigma$ of $\GL_{2m}(E')$ is $(\GL_{2m}(F'),\eta_{E'/F'})$-distinguished.
In the next lemma we freely use both the open and the closed orbit argument.

We make the following simple observation that, in conjunction with Theorem  \ref{thm dist ind class}, will allow us to relate Galois distinction for a classical group and a general linear group.
\begin{lemma}\label{lem gldist}
Let $N=2(n_1+\cdots+n_k)+r$, $\Pi_0$ a representation of $\GL_r(E')$, $\pi_i$ an irreducible representation of $\GL_{n_i}(E')$, $i\in [1,k]$ and $\chi$ be either the trivial character of $(F')^*$ or $\eta_{E'/F'}$.
Assume that
\begin{enumerate}
\item the representation 
\[
\Pi=\pi_1\times\cdots \times  \pi_k \times \Pi_0 \times (\pi_k^\vee)^\tau \times \cdots \times (\pi_1^\vee)^\tau
\]
of $\GL_N(E')$ is irreducible, 
\item $\Pi_0$ is $(\GL_r(F'),\chi\circ \det)$-distinguished and 
\item there exist an involution $\rho\in S_k$ and a subset $\set$ of $[1,k]$ such that $\rho(\set)=\set$ and $n_{\rho(i)}=n_i$, $i\in [1,k]$ and $y_i\in Y_{n_i}(E'/(E')^{\sigma\tau},\sgn)$ for any  $i\in \set$ such that $\rho(i)=i$ such that 
\[
\begin{array}{ll}
\pi_{\rho(i)}\simeq (\pi_i^\vee)^\sigma & i\not\in \set,\ \rho(i)\ne i \\
\pi_i \text{ is }(\GL_{n_i}(F'),\chi\circ \det)\text{-distinguished} & i\not\in\set,\ \rho(i)=i \\
\pi_{\rho(i)}\simeq \pi_i^{\tau\sigma} & i\in  \set, \ \rho(i)\ne i \\
\pi_i \text{ is }\U(y_i,E'/(E')^{\sigma\tau})-\text{distinguished} & i\in \set,\ \rho(i)=i. 
\end{array}
\]
\end{enumerate}
Then $\Pi$ is $(\GL_N(F'),\chi\circ \det)$-distinguished. 
\end{lemma}
\begin{proof}
Since $\Pi$ is irreducible we may freely permute its factors.
For $i\not\in \set$ such that  $\rho(i)\ne i$ we have that
\[
\pi_i \times \pi_{\rho(i)}\simeq \pi_i\times (\pi_i^\vee)^\sigma \ \ \ \text{and}\ \ \  (\pi_i^\vee)^\tau \times (\pi_{\rho(i)}^\vee)^\tau\simeq (\pi_i^\vee)^\tau \times \pi_i^{\sigma\tau}.
\]
For $i\in  \set$ such that $\rho(i)\ne i$ we have
\[
\pi_i \times  (\pi_{\rho(i)}^\vee)^\tau \simeq \pi_i \times (\pi_i^\vee)^\sigma.
\]
For $i\in  \set$ such that $\rho(i)= i$ it follows from \cite[Theorem 6.1(1)]{MR2930996} that $\pi_i\simeq \pi_i^{\sigma\tau}$ and therefore that 
\[
\pi_i\times (\pi_i^\vee)^\tau\simeq \pi_i^{\sigma\tau}\times  (\pi_i^\vee)^\tau.
\]
By the open orbit argument, it follows that $\Pi$ can be expressed as a product of representations each of the appropriate $\GL_m(E')$ that is $(\GL_m(F'),\chi\circ\det)$-distinguished. Therefore, $\Pi$ is $(\GL_N(F'),\chi\circ \det)$-distinguished by the closed orbit argument.  
\end{proof}

\subsection{L-groups and transfer maps}\label{ss lgp}
Recall (see  \S \ref{section Galois distinction and induction}) that $G^\circ=\mathbb{H^\circ}(E)$ where $\mathbb{H^\circ}$ is the connected component of the algebraic group $\mathbb{H}=\{g\in \GL_N:{}^tg^\tau \j[n] g=\j[n]\}$ defined over $F$. We denote by $W_k$ the Weil group of the non-archimedean local field $k$ and by $W_k'=W_k\times \SL_2(\C)$ the Weil-Deligne group.

Let $\bG^\circ=\mathbb{H}^\circ_E$ be the base change of $\mathbb{H^\circ}$ to $E$. Assume for the rest of this work that $\bG^\circ$ is quasi-split. The Langlands dual group (L-group) of $\bG^\circ$ is $\widehat{\mathbb{G^\circ}}(\C)\ltimes W_E$ with connected component $\widehat{\mathbb{G^\circ}}(\C)$. For many applications it suffices to keep track of the action of $W_E$ on $\widehat{\mathbb{G^\circ}}(\C)$ and it is convenient to consider a simplified version of the $L$-group. For instance, if $\mathbb{G}^\circ$ is split then $W_E$ acts trivially on $\widehat{\mathbb{G^\circ}}(\C)$ and we only keep track of the connected component.

We denote by ${}^LG^\circ$ the $L$-group of $G^\circ$ defined as follows. In the symplectic, the odd orthogonal and the split even orthogonal cases, we take 
${}^LG^\circ=\widehat{\mathbb{G^\circ}}(\C)$. Explicitly, $\widehat{\mathbb{G^\circ}}(\C)$ equals $\SO_{2n+1}(\C)$, $\Sp_{2n}(\C)$ or $\SO_{2n}(\C)$ in the symplectic, odd orthogonal or split even orthogonal case respectively. In the non-split but quasi-split even orthogonal case, $\mathbb{G^\circ}$ splits over a quadratic extension $Q/E$ and we set 
\[
{}^LG^\circ=\SO_{2n+2}(\C)\ltimes \Gal(Q/E).
\] 
Sending the Galois involution $\theta_{Q/E}$ of $Q/E$ to 
$s=\diag(I_n,w_2,I_n)$ we obtain an identification 
\[
{}^LG^\circ=\O_{2n+2}(\C)
\] 
with the split orthogonal group defined  with respect to $w_{2n+2}$. 
Hence in all these cases we have a natural imbedding of ${}^LG^\circ$ into $\GL_m(\C)$ where $m=m(n)$ equals $2n+1$ in the symplectic case, $2n$ in the split orthogonal case and $2n+2$ in the quasi-split non-split orthogonal case.

In the unitary case 
\[{}^LG^\circ=\GL_N(\C)\ltimes \Gal(E'/E)\] where the Galois involution $\tau$ of $E'/E$ acts on 
$\GL_N(\C)$ by \[\tau g\tau^{-1}=J_N {}^tg^{-1} J_N^{-1}\] where $J_N=\diag(1,-1,\dots,(-1)^{N-1})w_N$. In this case we set $m=m(n)=N$. 

We denote by $\Phi(G^\circ)$ the set of $L$-parameters for $G^\circ$. These are the $\widehat{\mathbb{G^\circ}}(\C)$-conjugacy classes $[\phi]$ of (admissible) $L$-homomorphisms $\phi:W'_E\rightarrow {}^LG^\circ$. Similarly, we denote by $\Phi(\GL_m(E'))$ the set of $L$-parameters of $\GL_m(E')$. That is, the $\GL_m(\C)$-conjugacy classes $[\phi]$ of $L$-homomorphisms $\phi:W'_{E'}\rightarrow \GL_m(\C)$. 

Note that in all cases an $L$-homomorphism $\phi:W'_E\rightarrow {}^LG^\circ$ for $G^\circ$ maps $W'_{E'}$ into $\GL_m(\C)$ and therefore defines an $L$-homomorphism for $\GL_m(E')$. Denote by $I:\Phi(G^\circ)\rightarrow \Phi(\GL_m(E'))$ the map defined by
\[
I([\phi])=[\phi|_{W'_{E'}}].
\]

The local Langlands correspondence is a conjectural classification of the irreducible representations of a local reductive group in terms of, so called, enhanced  $L$-parameters.  
The work of Arthur \cite{MR3135650} followed by that of Mok \cite{MR3338302} establish the local Langlands correspondence for tempered representations of the quasi-split classical groups $G^\circ$ that we consider. In fact, we will need a much weaker version of their result as formulated in \cite[Conjecture A]{MR3675168}.
In conjunction with \cite{MR3271238} or \cite{MR3834776} it provides a finite to one map from the set of isomorphism classes of irreducible representations of $G^\circ$ to the set $\Phi(G^\circ)$ of L-parameters of $G^\circ$ with certain natural properties. The fibers of this map form the $L$-packets for $G^\circ$.
Together with the local Langlands correspondence for general linear groups (\cite{MR1738446},\cite{MR1876802}) this defines a transfer of irreducible representations from $G^\circ$ to $\GL_m(E')$ along the map $I$. We denote by $T(\pi)$ the transfer of an irreducible representation $\pi$ of $G^\circ$. It is the representation of $\GL_m(E')$ with parameter $I([\phi])$ where $[\phi]$ is the parameter of $\pi$. 

It is known by the work of Arthur, Mok and others (see \cite{MR3801418}), that whenever one fixes a non-degenerate character $\psi$ of the unipotent radical of the Borel subgroup of $G^\circ$, each tempered L-packet contains a unique $\psi$-generic representation. 

\subsection{Transfer of induced generic representations}\label{ss trans ind}
Throughout this section fix a decomposition $n=n_1+\cdots+n_k+r$ and representations $\pi_0$ of $G_r^\circ$ and $\pi_i$ of $\GL_{n_i}(E')$, $i\in [1,k]$. In the split even orthogonal case, if $r=0$ we also fix a sign in order to choose one of the two standard parabolic subgroups of type $(n_1,\dots,n_k)$ (see \S\ref{ss parab}). Consider the representations
\[
\pi=\pi_1\times \cdots \times \pi_k\rtimes \pi_0
\] 
of $G^\circ$ and 
\[
\Pi=\pi_1\times \cdots \times \pi_k\times T(\pi_0) \times (\pi_k^\vee)^\tau \times \cdots \times (\pi_1^\vee)^\tau.
\]
of $\GL_{m}(E')$ where $m=m(n)$ and the representation $T(\pi_0)$ of $\GL_{m(r)}(E')$ is  the transfer of $\pi_0$ defined by the local Langlands correspondence.

We will say that an irreducible representation of $G^\circ$ is generic if it has a Whittaker model with respect to a non-degenerate character of the  standard maximal unipotent subgroup of $G^\circ$. 
By the standard module theorem \cite[Theorem 1.1]{MR1847492} every irreducible generic representation of $G^\circ$ is of the form $\pi$ above where $\pi_0$ is an irreducible generic tempered representation of $G_r^\circ$ and $\pi_1,\dots,\pi_k$ are irreducible essentially square-integrable representations such that $e(\pi_1)\ge \cdots\ge e(\pi_k)>0$. 
The standard parabolic subgroup of $G^\circ$ and the data $(\pi_1,\dots,\pi_k;\pi_0)$ are uniquely determined by $\pi$.
When this is the case, we will show that $\Pi$ is also irreducible and generic and in fact $\Pi=T(\pi)$.

\begin{lemma}\label{lemma explicit transfer}
Assume that $\pi_0$ is tempered, $\pi_1,\dots,\pi_k$ are essentially square-integrable, $e(\pi_1)\ge \cdots\ge e(\pi_k)\ge 0$ and both $\pi$ and $\Pi$ are irreducible.
Then $\Pi=T(\pi)$.
\end{lemma}
\begin{proof}
Let $M^\circ\simeq \GL_{n_1}(E')\times \cdots\times \GL_{n_k}(E')\times G_r^\circ$ be the standard Levi subgroup of the standard parabolic subgroup of $G^\circ$ fixed  in the discussion above from which $\pi$ is induced. Denote by $^LM^\circ$ the relevant standard Levi subgroup of $^LG^\circ$ (cf. \cite[\S3]{MR546608}) corresponding to $M^\circ$. Let $M'$ be the standard Levi subgroup  of $\GL_{m(n)}(\C)$ of type 
$(n_1,\dots,n_k,m(r),n_k,\dots,n_1)$. 
The connected component of
$^LM^\circ$ equals $M'\cap G^\circ$. 

By compatibility of the local Langlands correspondence with parabolic induction for tempered representations (as in \cite[Conjecture A]{MR3675168}) and 
with the Langlands classification (\cite{MR3271238} or \cite{MR3834776}), the L-parameter $[\phi]$ of $\pi$ is represented by some L-homomorphism 
\[\phi:W'_E\rightarrow  {}^LM^\circ\] with image in ${}^LM^\circ$. Then $\phi$ maps $W'_{E'}$ into the connected component of ${}^LM^\circ$ and in particular into $M'=\diag(\GL_{n_1}(\C),\dots,\GL_{n_k}(\C),\GL_{m(r)}(\C),\GL_{n_k}(\C),\dots,\GL_{n_1}(\C))$ and from the compatibility condition defining an $L$-homomorphism one checks that 
\[
I[\phi]=[\phi_{|W'_{E'}}]=[\diag(\phi_1,\dots,\phi_k,(\phi_0)_{|W'_{E'}},\phi_k',\dots,\phi_1')]
\]
where $[\phi_i]$ is the $L$-parameter of $\pi_i$, $i\in[0,k]$ and $[\phi_i']$ the $L$-parameter of $(\pi_i^\vee)^\tau$, $i\in[1,k]$. That is, $[\phi]$ is the $L$-parameter for $\Pi$.

\end{proof}

\begin{lemma}\label{lemma transfer of generic is irreducible}
Assume that $\pi_0$ is generic tempered, $\pi_1,\dots,\pi_k$ are essentially square-integrable, $e(\pi_1)\ge \cdots\ge e(\pi_k)\ge 0$ and $\pi$ is irreducible.
Then $\Pi$ is irreducible and generic. In particular $\Pi=T(\pi)$.
\end{lemma}
\begin{proof}
Note that $T(\pi_0)$ is tempered and in particular unitary.
If $e(\pi_1)=0$ then $\Pi$ is induced from a unitary representation and is therefore irreducible by a theorem of Bernstein \cite{MR748505}. 
Otherwise, let $l$ be maximal such that $e(\pi_l)>0$. Replacing $\pi_0$ by $\pi_{l+1}\times \cdots\times \pi_k \rtimes \pi_0$ and $k$ by $l$ we may assume without loss of generality that $l=k$, i.e., that $e(\pi_k)>0$.

Since $\pi$ is irreducible, we also have that $\pi_i' \rtimes \pi_0$, $\pi_i'\times \pi_j'$ are irreducible for all $i,
,j\in [1,k]$ where $\pi_i'\in \{\pi_i,(\pi_i^\vee)^\tau\}$ for $i\in [1,k]$.
Since $T(\pi_0)$ is tempered it is of the form
\[
T(\pi_0)\simeq \delta_1 \times \cdots \times \delta_t
\]
for a unique multi-set $\{\delta_1,\dots,\delta_t\}$ of square-integrable representations.  Since the standard module $\pi_i \rtimes \pi_0$ is irreducible it follows from \cite[Lemma 1.1 and Theorem 1.1]{MR1847492} that 
\[
L(1,\pi_i^\vee,\pi_0)=L(1,\pi_i^\vee,T(\pi_0))=\prod_{j=1}^t L(1,\pi_i^\vee,\delta_j)\ne \infty
\]
where $L(s,\sigma_1,\sigma_2)$ denotes the local Rankin-Selberg $L$-function of irreducible representations $\sigma_i$ of $\GL_{m_i}(E')$, $i=1,2$. The first equality can be viewed as a definition of the left hand side and the second follows from \cite[Proposition 8.4]{MR701565}.
It follows from \cite[Lemma 2.3]{MR3785415}, in its terminology, that $\delta_j$ does not precede $\pi_i$ and since $e(\pi_i)>0=e(\delta_j)$ it is also the case that $\pi_i$ does not precede $\delta_j$. It therefore follows from \cite[Theorem 9.7]{MR584084} that $\pi_i \times \delta_j$ is irreducible for every $i\in [1,k]$ and $j\in [1,t]$. It follows from \cite[Theorem 8.1]{MR3202556} that $(T(\pi_0)^\vee)^\tau\simeq T(\pi_0)$ and consequently $\{\delta_1,\dots,\delta_t\}$ is stable under $\delta\mapsto (\delta^\vee)^\tau$. Hence $\pi_i\times (\delta_j^\vee)^\tau$ and therefore also its conjugate dual $(\pi_i^\vee)^\tau\times \delta_j$ are irreducible. It now follows from \cite[Theorem 9.7]{MR584084} that $\Pi$ is irreducible and generic. The equality $\Pi=T(\pi)$ follows from Lemma \ref{lemma explicit transfer}. 

\end{proof}

\subsection{A consequence of a conjecture of Prasad}

We state a consequence of a conjecture of Prasad for distinguished representations in the context of a Galois pair. 
The conjecture is formulated in terms of the opposition group $\Y^\op=\Y^{\op,E/F}$ defined in \cite[\S 5]{DP}  and a character $\omega_{\Y}=\omega_{\Y,E/F}$ of $\Y(F)$ of order at most two constructed in \cite[Proposition 6.4]{DP} for any connected reductive group $\Y$ defined and quasi-split over $F$ with respect to the quadratic extension $E/F$. 

We begin by explicating $\Y^{\op}$ and $\omega_{\Y}$ for groups $\Y$ relevant to this work.

\subsubsection{A table for $\Y^{\op}$ and $\omega_{\Y,E/F}$.}

We denote by $\U_{m,K/F}$ the quasi-split unitary group defined over $F$ associated with a quadratic extension $K/F$ and defined by $w_m$. Let $\wsn: \U_{m,K/F}(F)\rightarrow K^*/F^*$ be the composition of $\det:\U_{m,K/F}(F)\rightarrow (K/F)_1$ with the inverse of the isomorphism 
$z F^*\mapsto zz^{-\tau}$ from $ K^*/F^*$ to $\U_{1,K/F}(F)$. 
(In \cite{MR0104764}, Wall interpreted this map as a spinor norm analogue.) For the purpose of the table below, $K/F$ will be a quadratic extension such that $K\cap E=F$ so that $KE/F$ is a Klein extension and contains a third quadratic extension of $F$, that we denote by $K'$, different than $K$ and $E$.

We denote by $\SO_m$ any special orthogonal group defined and quasi-split over $F$ of the form $\{g\in \SL_m: {}^t g yg=y\}$ for some $y={}^t y\in \GL_m(F)$. ( With this convention $\SO_2$ stands for a rank one torus that is either split over $F$ (isomorphic to $\GL_1$) or splits over a quadratic extension $L$ of $F$ (isomorphic to $\U_{1,L/F}$). We denote by $\sn:\SO_m(F)\rightarrow F^*/F^{*2}$ the spinor norm map. 

We also denote by $\eta_{E/F}$ the quadratic character of $F^*$ with kernel $N(E/F)$.
We have\\

\begin{center} {\centering\begin{tabular}{|L|L|L|}
 \hline
  \Y & \omega_{\Y} & \Y^\op \\ 
  \hline  
 \GL_m & \eta_{E/F}^{m-1}\circ \det & \U_{m,E/F}  \\ 
 \hline 
 \U_{m,E/F} & \triv & \GL_n \\ 
 \hline 
  \Sp_{2m} & \triv & \Sp_{2m} \\ 
  \hline
 \SO_m & \eta_{E/F}^m\circ \sn & \SO_m \\ 
 \hline 
 U_{m,K/F} & \eta_{EK/K}^{m-1}\circ \wsn & U_{m,K'/F} \\
 \hline 
\end{tabular}}\end{center}

\vspace{0.5cm}

In the sequel, this table is sometimes used without further mention.
For the sake of completeness we provide the details of the computation in Appendix \ref{app}.

\subsubsection{Prasad's conjecture and transfer of distinction for quasi-split classical groups}

Let $\Y$ be an $F$-quasi-split group. Since $(\Y^\op)_E\simeq \Y_E$ the $L$-group ${}^L \Y(E)$ of $\Y_E$ is a subgroup of the $L$-group ${}^L \Y^\op(F)$ of $\Y^\op$ with the same connected component and for any $L$-homomorphism $\phi:W_F' \rightarrow {}^L \Y^\op(F)$ we have $\phi(W_E')\subseteq {}^L \Y(E)$. This defines a base change map 
\[
\BC_F^E:\Phi(\Y^\op(F))\mapsto  \Phi(\Y(E)), \ \ \ \BC_F^E([\phi])=[\phi_{|W'_E}]
\]
from $L$-parameters of $\Y^\op(F)$ to $L$-parameters of $\Y(E)$.


We say that a representation $\pi$ of $\Y(E)$ is \emph{$E/F$-generic} if it is $(U,\psi)$-generic for a $\sigma$-stable maximal unipotent subgroup $U$ of $\Y(E)$ and a non-degenerate character $\psi$ of $U$ satisfying $\psi^\sigma=\psi^{-1}$ (which is the same as $\psi_{|U^\sigma}\equiv \triv$).
The following is a part of \cite[Conjecture 13.3]{DP}. 

\begin{conjecture}[Dipendra Prasad]\label{conjecture Prasad}
Assume that the Langlands correspondence is known for $\Y(E)$ and $\Y^{\op}(F)$. 
\begin{enumerate}
\item\label{part irreps} Let $\pi$ be a $(\Y(F),\omega_{\Y,E/F})$-distinguished irreducible representation of $\Y(E)$. Then the parameter of the L-packet of $\pi$ belongs to 
$\BC_F^E(\Phi(\Y^\op(F)))$. 
\item Conversely, if the $L$-packet of an $E/F$-generic representation $\pi$ of $\Y(E)$ corresponds to a parameter in $\BC_F^E(\Phi(\Y^\op(F)))$ then $\pi$ is $(\Y(F),\omega_{\Y,E/F})$-distinguished.
\end{enumerate}
\end{conjecture}

\begin{remark}\label{rmk conj GL} 
This conjecture is now a theorem when $\Y=\GL_n$. It follows from the results in \cite{MR2075482, MR2063106, MR2146859, MR2595008, MR2755483, MR3421655}.
\end{remark}
Fix a quadratic extension $K$ of $F$ contained in $E'$ (so that $E'/K$ is also quadratic) and let $\U_n(E'/K)=\U_{n,E'/K}(K)$. We will need the following characterization  of the image of 
\[
\BC_{K}^{E'}:\Phi(\U_n(E'/K))\rightarrow \Phi(\GL_n(E')).
\]

We say that an L-homomorphism 
\[\phi_{E'}:W'_{E'}\rightarrow \GL_n(\C)\] (or the $L$-parameter $[\phi]$) is $E'/K$-conjugate of sign $\epsilon\in \{\pm 1\}$ (we also say $E'/K$-conjugate orthogonal if the sign is $1$ and $E'/K$-conjugate symplectic if the sign is $-1$) if there exists a non-degenerate 
bilinear form 
\[
B:\C^n\times \C^n\rightarrow \C
\] 
such that 
\begin{equation}\label{eq Bcond}
 B(\phi_{E'}(w )x, \phi_{E'}(sws^{-1}) y)=B(x, y) \ \ \ \text{and} \ \ \ B(x, \phi_{E'}(s^2)y)=\epsilon B(y,x) 
\end{equation}
for $w\in W_{E'}',\ \ \ x,\,y\in\C^n$ where $s$ is a choice of an element of $W_K'\setminus W_{E'}'$.
 
\begin{theorem}\emph{(}\cite[Theorem 8.1]{MR3202556}\emph{)}\label{thm ggp}
An L-homomorphism $\phi_{E'}$ as above is $E'/K$-conjugate of sign $(-1)^{n-1}$ if and only if $[\phi_{E'}]\in \BC_{K}^{E'}(\Phi(\U_n(E'/K)))$. 
\end{theorem}

We will also need the following slightly more precise result.

\begin{theorem}\emph{(}\cite{MR2075482, MR2063106, MR2146859, MR2595008, MR2755483}\emph{)}\label{thm conjugate orth.-simp. for GL}
Let $\pi$ be a generic irreducible representation of $\GL_n(E')$ with L-parameter $[\phi_\pi]$. Then $\phi_\pi$ is $E'/K$-conjugate symplectic if and only if $\pi$ is $(\GL_n(K),\eta_{E'/K}\circ \det)$-distinguished and $E'/K$-conjugate orthogonal if and only if $\pi$ is 
$\GL_n(K)$-distinguished.
\end{theorem}

We are now in a position to show that the first part of Conjecture \ref{conjecture Prasad} for the generic spectrum implies a relation between distinction in the classical group and distinction of the transfer.

\begin{proposition}\label{prop distT}
Let $\Y=\mathbb{H}^\circ$ be a quasi-split classical $F$-group of $F$-rank $n$, defined as in \S\ref{ss lgp} and assume that Conjecture \ref{conjecture Prasad} \eqref{part irreps} holds for $\Y$. If $\pi$ is an irreducible, generic and $(\Y(F),\omega_{\Y,E/F})$-distinguished representation of $\Y(E)$ then the representation $T(\pi)$ of $\GL_m(E')$ (for the appropriate $m=m(n)$) is $(\GL_m(F'),\chi)$-distinguished where $\chi$ is the trivial character in the even orthogonal case and equals $\omega_{\GL_m,E'/F'}$ otherwise.
\end{proposition}
\begin{proof}
Let $[\phi]\in \Phi(\Y(E))$ be the parameter of $\pi$ so that $I([\phi])\in \Phi(\GL_m(E'))$ is the parameter of $T(\pi)$. It follows from Lemma \ref{lemma transfer of generic is irreducible} that $T(\pi)$ is generic and therefore, by Theorem \ref{thm conjugate orth.-simp. for GL}, it suffices to show that $[\phi']$ is $E'/F'$-conjugate of sign $\epsilon$ where $\epsilon$ equals one in the even orthogonal case and $(-1)^{m-1}$ otherwise. 

In the unitary case $m=n$, $\Y(E)=\U_n(E'/E)$, $\Y^\op(F)=\U_n(F''/F')$ where $F''=(E')^{\sigma\tau}$ and 
\[
I=\BC_{E}^{E'}:\Phi(\U_n(E'/E)) \rightarrow \Phi(\GL_n(E'))\] is the corresponding base change map. This implies the commutativity of the diagram
\[
\xymatrix{ 
 \Phi(\U_n(E'/E)) \ar[r]^-I & \Phi(\GL_n(E'))    \\
 \Phi(\U_n(F''/F)) \ar[r]^-{\BC_F^{F'}}\ar[u]^{\BC_F^E} & \Phi(\U_n(E'/F')). \ar[u]^{\BC_{F'}^{E'}} }
 \] 
(Both compositions from $\Phi(\U_n(F''/F))$ to $\Phi(\GL_n(E'))$ are defined by restriction to $W'_{E'}$.) By assumption, $[\phi]\in \Phi(\U_n(E'/E))$ lies in the image of $\BC_F^E$ and therefore $I([\phi])$ lies in the image of $\BC_{F'}^{E'}$ and by Theorem \ref{thm ggp} it is $E'/F'$-conjugate of sign $(-1)^{n-1}$. 

Consider now the symplectic or orthogonal case. Again, by assumption $[\phi]$ lies in the image of $\BC_F^E:\Phi(\Y^\op(F))\rightarrow \Phi(\Y(E))$. Let $\phi': W_F' \rightarrow {}^L \Y^{\op}(F)$ be such that 
\[
[\phi]=I(\BC_F^E([\phi]'))=[\phi'_{|W'_E}].
\] 
Note that there exists a non-degenerate bilinear form $B:\C^N \times \C^N\rightarrow \C^N$ such that $B(y,x)=\epsilon B(x,y)$, $x,\,y\in \C^N$ and the image of $\phi'$ preserves $B$. Let $s\in W_F\setminus W_E$ and define the bilinear form $B' :\C^N \times \C^N\rightarrow \C^N$ by
\[
B'(x,y)=B(\phi'(s)x,y).
\]
It is now a straightforward verification that $B'$ and $\phi$ satisfy the conditions \eqref{eq Bcond}
and since $I$ is defined via the inclusion ${}^L \Y(E)\subseteq \GL_N(\C)$ it follows that $I([\phi])$ is $E/F$-conjugate of sign $\epsilon$.
\end{proof}
\begin{remark}
In the odd orthogonal case we could alternatively argue as in the unitary case using the commutative diagram  
\[
\xymatrix{ 
 \Phi(\SO_{2n+1}(E)) \ar[r]^-I & \Phi(\GL_{2n}(E))    \\
 \Phi(\SO_{2n+1}(F)) \ar[r]^{I_F}\ar[u]^{\BC_F^E} & \Phi(\U_{2n}(E/F)) \ar[u]^{\BC_{F}^{E}}  }
 \] 
where $I_F$ is induced by the inclusion $\GL_{2n}(\C)^\sigma \subseteq {}^L U_{2n}(E/F)$ as $\GL_{2n}(\C)^\sigma$ is conjugate to $\Sp_{2n}(\C)$ in $\GL_{2n}(\C)$.
\end{remark}

\begin{corollary}
Assume that $-1\in N(E/F)$. In the notation of Proposition \ref{prop distT} assume that Conjecture \ref{conjecture Prasad} \eqref{part irreps} holds for $\Y$. If $\pi$ is a generic irreducible representation of $\Y(E)$ that is $\Y(F)$-distinguished then $T(\pi)$ is $\GL_{m(n)}(F')$-distinguished. 
\end{corollary}
\begin{proof}
Taking the formula for $\omega_{\Y,E/F}$ into account, the corollary is exactly the statement of Proposition \ref{prop distT} except in the even unitary and odd orthogonal cases. Consider these two cases now. Let $\omega$ be the quadratic character of $\Y(E)$ that extends the character $\omega_{\Y,E/F}$ of $\Y(F)$ given by Lemma \ref{lemma Prasad character extends} of Appendix \ref{app} (see Remark \ref{rem extension is a prasad character}). The representation $\pi\otimes \omega$ is $(\Y(F),\omega_{\Y,E/F})$-distinguished and since in both cases $m(n)$ is even it follows from Proposition \ref{prop distT} that $T(\pi\otimes \omega)$ is $(\GL_{m(n)}(F'),\eta_{E'/F'}\circ \det)$-distinguished. By Lemma \ref{lemma transfer and character twist} of Appendix \ref{app} we have $T(\pi\otimes \omega)=T(\pi)\otimes (\eta_{E'/F'}\circ \det)$ and it therefore follows that $T(\pi)$ is $\GL_{m(n)}(F')$-distinguished as required.
\end{proof}

\subsection{Transfer of distinction for principal series}

Let $\mathbb{H}^\circ=\mathbb{H}_n^\circ$ be a quasi-split classical $F$-group of $F$-rank $n$, defined as in \S\ref{ss lgp}, and set $H^\circ=H^\circ_n=\mathbb{H}^\circ(F)$ and $G^\circ=G^\circ_n=\mathbb{H}^\circ(E)$.

Our goal in this section is to show that for certain irreducible representations $\pi$ of $G^\circ$ induced from a cuspidal representation of a Levi subgroup of the form $M_\alpha^\circ$ with $\alpha\in A_0$ (see \S \ref{ss parab}), i.e., contained in the Siegel Levi, the implication of Proposition \ref{prop distT} can be proved unconditionally.

We need some further preparation. The following will be used in the non-split quasi-split even orthogonal case. In this case $\j\in Y_2(F/F,1)$ is such that $\SO(\j)$ is anisotropic and splits over a unique quadratic extension $K=F[\sqrt{-\det \j}]$ of $F$. 
Let $\alpha$ be a character of $\SO(\j,E)$. 

If $K=E$ then $\SO(\j,E)\simeq E^*$. Let $\lambda_\alpha$ be the character of $E^*$ corresponding to $\alpha$ under this identification. It is a straightforward observation that 
\[
T(\alpha)=\lambda_\alpha \times \lambda_\alpha^{-1}.
\]
Otherwise, let $L=KE$ and  $\SO_2(L/E)=\SO(\j,E)$. In particular, $\SO_2(L/E)\simeq (L/E)_1\simeq L^*/E^*$ where the right isomorphism is via Hilbert 90.  
We denote by $\lambda_\alpha$ the corresponding character of $L^*$ trivial on $E^*$. Let $\phi: W_L' \rightarrow \C^*$ be the $L$-parameter of $\lambda_\alpha$. The two dimensional representation $\Ind_{W_L'}^{W_E'}(\phi)$ is therefore a $L$-homomorphism corresponding to a representation $\pi(\alpha)$ of $\GL_2(E)$. In fact, the representation $\pi(\alpha)$ has a model as the space of $\alpha$-coinvariants in the Weil representation (see \cite[D\'{e}finition 2.1.1]{matringethese}).    
According to our convention, ${}^L\SO_2(L/E)=\O_2(\C)$ and for the $L$-homomorphism $\phi_\alpha:W_E' \rightarrow \O_2(\C)$ corresponding to $\alpha$ it is easy to see that composing with the inclusion $\O_2(\C)\subseteq \GL_2(\C)$ we get that $I(\phi_\alpha)\simeq \Ind_{W_L'}^{W_E'}(\phi)$. That is, $T(\alpha)=\pi(\alpha)$.


\begin{proposition}\label{prop T character of SO(2)}
Let $\alpha$ be a character of $G_0^\circ$ that is $(G_0^\circ)_z$-distinguished for some $z\in X_0$. Then the representation $T(\alpha)$ of $\GL_{m(0)}(E')$ is $\GL_{m(0)}(F')$-distinguished. 
\end{proposition}
\begin{proof}
The lemma is trivially true if $m(0)=0$ (both $\alpha$ and $T(\alpha)$ equal the trivial character of the trivial group). Otherwise, either $m(0)=1$ in the odd unitary case or $m(0)=2$ in the quasi-split non split even orthogonal case. 

In the first case $G_0^\circ=(E'/E)_1$ and $(G_0^\circ)_z=(F'/F)_1$ is independent of $z\in X_0$. Thus $\alpha$ is a character of $(E'/E)_1$ that is trivial on $(F'/F)_1$ and $T(\alpha)(a)=\alpha(a^{1-\tau})$, $a\in (E')^*$. Since $a^{1-\tau}\in (F'/F)_1$ for $a\in (F')^*$ we conclude that $T(\alpha)$ is indeed $(G_0^\circ)_z$-distinguished.

In the second case $\mathbb{H}_0^\circ=\SO(\j)$. If $-\det j\in E^{*2}$ then as discussed above $T(\alpha)=\lambda_\alpha \times \lambda_\alpha^{-1}$ and  $G_0^\circ\simeq E^*$. Under this isomorphism the group $(G_0^\circ)_z$ identifies with $F^*$ if $z\in X_0^\circ$ and with $(E/F)_1$ otherwise. In the first case this means that $\lambda_\alpha$ (and hence also $\lambda_\alpha^{-1}$) is trivial on $F^*$ and in the second case that $\lambda_\alpha^\sigma=\lambda_\alpha$. Consequently, $T(\alpha)$ is $\GL_2(F)$-distinguished by a closed orbit argument in the first case and by an open orbit argument in the second (see \S \ref{ss gl dist}).
If $-\det j\not\in E^{*2}$ set $K=F[\sqrt{-\det \j}]$ and $L=KE$ and let $K'$ be the third quadratic extension of $F$ contained in $L$ that is different from $E$ and $K$.
As observed above $T(\alpha)=\pi(\alpha)$ and $G_0^\circ\simeq (L/E)_1$. In fact, the full orthogonal group $G_0$ is isomorphic to the semi-direct product $(L/E)_1 \rtimes \Z_2$ where the non-trivial element of $\Z_2$ acts on $(L/E)_1$ by the Galois action of $L/E$. Under this isomorphism $(G_0^\circ)_z$ identifies with  $(K/F)_1$ if $z\in X_0^\circ$ and with $(K'/F)_1$ otherwise. In other words, with the above notation, $\lambda_\alpha$ is trivial on either $K^*$ or $(K')^*$. The fact that for such $\alpha$ the representation $\pi(\alpha)$ is $\GL_2(F)$-distinguished is proved in \cite[Th\'{e}or\`{e}me 2.3.3]{matringethese}.

\end{proof}

Let $n=n_1+\cdots+n_k$ and $M^\circ=M^\circ_{(n_1,\dots,n_k)}$ be the corresponding standard Levi subgroup of $G^\circ$. For 
$x=x_w(\{y_i\}_{i\in I(w)},\,z)\in G^\circ \cdot I_N$ defined by \eqref{eq def x} let $\eta\in G^\circ$ be such that $\eta\cdot I_N=x$. Note that $M^\circ_x=M^\circ\cap \eta H^\circ \eta^{-1}$ and therefore the character $\omega_{\mathbb{H}^\circ,E/F}\circ \Ad(\eta^{-1})$ of $\eta H^\circ \eta^{-1}$ restricts to a character of $M^\circ_x$. 

\begin{lemma}\label{lem main charcomp}
Consider either the even unitary or odd orthogonal case. With the above notation let $g_i\in \GL_{n_i}(E')$, $i=1,\dots,k$ be such that $m=\inj(g_1,\dots,g_k)\in M_x$.
We have 
\[
\omega_{\mathbb{H}^\circ,E/F}(\eta^{-1}m\eta)= \prod_{i\not\in\set,\,\rho(i)=i} \eta_{E'/F'}(\det g_i).
\]
\end{lemma}
\begin{proof}
Note that in both cases $G_0^\circ$ is the trivial group. 
In the even unitary case we have 
\[
\omega_{U_{2n,F'/F},E/F}=\eta_{E'/F'}\circ \alpha^{-1} \circ \det
\] 
where $\alpha:(F')^*/F^*\rightarrow (F'/F)_1$ is the isomorphism $aF^*\mapsto a^{1-\tau}$. It follows from Proposition \ref{prop adm orbs} \eqref{part stab} that 
\[
\det(\eta^{-1}m \eta)=\det m=(aa^\sigma)^{1-\tau}b^{1-\tau}c^{1-\tau}
\] 
where 
$a=\prod_{i<\rho(i)}\det g_i\in (E')^*$, $b=\prod_{i\in \set,\ \rho(i)=i} \det g_i\in (E'/(E')^{\sigma\tau})_1$ and $c=\prod_{i\not\in \set,\ \rho(i)=i} \det g_i\in (F')^*$.
We have 
\[
\eta_{E'/F'}(\alpha^{-1}((aa^\sigma)^{1-\tau}))=\eta_{E'/F'}(N_{E'/F'}(a))=1
\] 
and
\[
\eta_{E'/F'}(\alpha^{-1}(c^{1-\tau}))=\eta_{E'/F'}(c).
\]
It remains to see that $\eta_{E'/F'}(\alpha^{-1}(b^{1-\tau}))=1$. Let $d\in (E')^*$ be such that $b=d^{1-\sigma\tau}$ (by Hilbert 90 ). Then $b^{1-\tau}=(d d^\sigma)^{1-\tau}$ and therefore 
\[
\eta_{E'/F'}(\alpha^{-1}(b^{1-\tau}))=\eta_{E'/F'}(N_{E'/F'}(d))=1.
\]
The even unitary case follows.
In the odd orthogonal case 
\[
\omega_{\SO_{2n+1},E/F}=\eta_{E/F}\circ \sn.
\]
It follows from Proposition \ref{prop adm orbs} \eqref{part stab} that $m=\inj(g_1,\dots,g_k) \mapsto (g_i)_{i\le \rho(i)}$ defines an isomorphism
\[
M_x\simeq [\mathop{\times}\limits_{i<\rho(i)} \GL_{n_i}(E)] \times [ \mathop{\times}\limits_{i\not\in \set ,\ \rho(i)=i} \GL_{n_i}(F)] \times  [ \mathop{\times}\limits_{i\in \set ,\ \rho(i)=i} \U(y_i,E/F)].
\]
Let $m_i=\inj(h_1,\dots,h_k)$ where $h_j=I_{n_j}$ unless $j\in\{i,\rho(i)\}$ in which case $h_j=g_j$. Then clearly $m=\prod_{i\le \rho(i)} m_i$ with all $m_i$'s commuting. Set also $m^\circ=\prod_{i\in \set,\ \rho(i)=i} m_i$. We compute separately $\eta_{E/F}\circ \sn (\eta^{-1}m_i \eta)$ for each $i$ such that either $i\not \in \set$ or $\rho(i)\ne i$ as well as $\eta_{E/F}\circ \sn (\eta^{-1}m^\circ \eta)$.

The computation for either $i\not \in \set$ or $\rho(i)\ne i$ is implicitly based on finding $\eta$ explicitly. More precisely, let
\[
x_i=\begin{cases} I_{2n_i} & i\not \in \set,\ \rho(i)=i \\ \inj\sm{}{I_{n_i}}{I_{n_i}}{} & i\not \in \set,\ \rho(i)\ne i \\ \sm{}{I_{2n_i}}{I_{2n_i}}{} & i\in \set,\ \rho(i)\ne i \end{cases}
\ \ \ \text{and}
\ \ \ 
\ell_i=\begin{cases} \inj(g_i) & i\not \in \set,\ \rho(i)=i \\ \inj(g_i,g_i^\sigma) & i\not \in \set,\ \rho(i)\ne i \\ \inj(g_i,\bar g_i^*) & i\in \set,\ \rho(i)\ne i. \end{cases} 
\]
Then $\eta_{E/F}\circ \sn (\eta^{-1}m_i \eta)=\eta_{E/F}\circ \sn(\eta_i^{-1} \ell_i\eta_i)$ where
$\eta_i\in \SO_{2k_i}(E)$ is such that $\eta_i\cdot I_{2k_i}=x_i$ and $k_i$ equals $n_i$ if $\rho(i)=i$ and $2n_i$ otherwise. Below we provide an explicit such $\eta_i$ whenever either $i\not \in \set$ or $\rho(i)\ne i$. This shows that $x\in \SO_{2n+1}(E)\cdot  x'$ where $x'=x_{I(w)}(\{y_i\}_{i\in I(w)},\,z)$ (here we think of the set of indices $I(w)$ as an element of $\weyl_k$), that is, $x'\in \SO_{2n+1}(E)\cdot I_{2n+1}$.
Similarly, set
\[
x^\circ=\begin{pmatrix} & & I_t \\ & (-1)^t & \\ I_t & & \end{pmatrix}\inj(w_t \bar y)\in X_{2t+1}^\circ
\]
where $i_1<\cdots<i_j$ are such that $I(w)=\{i_1,\dots,i_j\}$, $t=n_{i_1}+\cdots,+n_{i_j}$ and $y=\diag(y_{i_1},\dots,y_{i_j})$. Applying Lemma \ref{lem in eorb} to both $x'$ and $x^\circ $ we conclude that $x^\circ =\eta^\circ \cdot I_{2t+1}$ for some $\eta^\circ\in \SO_{2t+1}(E)$ and we then have 
\[
\eta_{E/F}\circ \sn (\eta^{-1}m^\circ \,\eta)=\eta_{E/F}\circ \sn ((\eta^\circ)^{-1}\inj(g_{i_1},\dots,g_{i_j}) \eta^\circ).
\]
We also note that $\inj(g_{i_1},\dots,g_{i_j}) \in U(y,E/F)$.

Furthermore, our explicit computations of the spinor norm are based on the fact that any $h\in \SO_m(F)$ can be written as an even 
product $h=s_{v_1}\dots s_{v_{2k}}$ of orthogonal reflections with respect to anisotropic vectors $v_i\in F^m$ for the 
quadratic form $q(v)= {}^t vw_m v$, $v\in F^m$ and that in such a situation 
\[
\sn(h)=q(v_1)\dots q(v_{2k})F^{*2}\in  F^*/F^{*2}.
\]
Based on this we have
\begin{equation}\label{eq sn siegel}
\sn(\inj(h))=\det h \ F^{*2},\ \ \ h\in \GL_m(F). 
\end{equation}
Indeed, any homomorphism from $\GL_m(F)$ to $F^*/F^{*2}$ factors through determinant and therefore $\sn(\inj(h))=\chi(\det h)$ for some homomorphism $\chi:F^* \rightarrow F^*/F^{*2}$. In order to show that $\chi$ is the natural projection it suffices to show that $\sn(\inj(\diag(I_{m-1},t)))=tF^{*2}$, $t\in F^*$ and this reduces to an $\SO_2(F)$ computation of $\sn(\diag(t,t^{-1}))$. Indeed, 
\begin{equation}\label{eq so2}
\diag(t,t^{-1})=s_{v(t)}s_{v(1)} 
\end{equation}
where ${}^t v(t)=(t,1)$ so that $q(v(t))=2t$.

If $i\not\in \set$ and $\rho(i)=i$ then we may choose $\eta_i=I_{2n_i}$ and 
\[
\eta_{E/F}( \sn (\eta^{-1}m_i \eta))=\eta_{E/F}( \sn (\inj(g_i)))=\eta_{E/F}(\det g_i)
\]
by \eqref{eq sn siegel}.

If $i\not\in \set$ and $i<\rho(i)$ then 
\[
\eta_{E/F}( \sn (\eta^{-1}m_i \eta))=\eta_{E/F} (\sn (\inj(\zeta_{n_i}^{-1} \diag(g_i,g_i^\sigma)\zeta_{n_i}))
\]
where $\zeta_m =\sm{I_m}{\imath I_m}{I_m}{-\imath I_m}$ and the right hand side is computed in $\SO_{2n_i}(F)$. (Recall from the proof of Lemma \ref{lem in eorb} that $\zeta_m\cdot I_{2m}=\sm{}{I_m}{I_m}{}$.)
Note that $\zeta_{n_i}^{-1} \diag(g_i,g_i^\sigma)\zeta_{n_i}\in \GL_{2n_i}(F)$ and by \eqref{eq sn siegel} we see that
\[
\eta_{E/F}( \sn (\eta^{-1}m_i \eta))=\eta_{E/F}(N_{E/F}(\det g_i))=1.
\]
For the next step we introduce the matrix
\[
\xi_m=\begin{pmatrix} \frac12 I_m & & \imath I_m & \\ & -\frac{1}{2\imath} I_m & & I_m \\ \frac12 I_m & & -\imath I_m & \\ & \frac{1}{2\imath} I_m & & I_m          \end{pmatrix}\in \SO_{4m}(E)
\]
that satisfies 
\[
\xi_m\cdot I_{4m}=\begin{pmatrix} & I_{2m} \\ I_{2m} & \end{pmatrix}.
\]
If $i\in \set$ and $i<\rho(i)$ then 
\[
\eta_{E/F}( \sn (\eta^{-1}m_i \eta))=\eta_{E/F} (\sn(\xi_{n_i}^{-1} \inj(g_i,\bar g_i^*)\xi_{n_i})).  
\]
By the argument used in the previous steps, in order to show that this equals to one for every $g_i\in \GL_{n_i}(E)$ it suffices to consider $g_i=\diag(I_{n_i-1},t)$ for all $t\in E^*$. This reduces us to the computation in $\SO_4(F)$ of $\eta_{E/F}(\sn(\xi_1^{-1}\inj(t,t^{-\sigma}) \xi_1))$. Note that for $v={}^t(0,1,1,0)$ the orthogonal reflection $s_v=\diag(1,-w_2,1)$ and clearly $\sn(\xi_1^{-1}\inj(t,t^{-\sigma}) \xi_1)=\sn(s_v\xi_1^{-1}\inj(t,t^{-\sigma}) \xi_1s_v)$. Furhtermore, $s_v\xi_1^{-1}\inj(t,t^{-\sigma}) \xi_1s_v=\inj(a(t))$ where
\[
a(t)=\begin{pmatrix} \frac12(t+t^\sigma) & \imath(t^\sigma-t) \\ \frac{1}{4\imath}(t^\sigma-t)  &\frac12(t+t^\sigma)  \end{pmatrix}.
\]
Since $\det a(t)=N_{E/F}(t)$ it now follows from \eqref{eq sn siegel} that $\eta_{E/F}(\sn(\inj(a(t))))=1$ as required.

To conclude the lemma it remains to show that for $y\in Y_m(E/F,1)$, such that  
\[
x(y)=\begin{pmatrix} & & I_m \\ & (-1)^m & \\ I_m & & \end{pmatrix}\inj(w_m \bar y)\in \SO_{2m+1}(E)\cdot I_{2m+1}
\]
and $g\in \U(y,E/F)$ we have 
\begin{equation}\label{eq ucase}
\eta_{E/F}(\sn(\eta(y)^{-1}\inj(g)\eta(y))=1
\end{equation}
where $\eta(y)\in \SO_{2m+1}(E)$ is such that $\eta(y)\cdot I_{2m+1}=x(y)$.
Indeed, we only need to show this for $y$ of the form $\diag(y_{i_1},\dots,y_{i_j})$ and $g$ of the form $\diag(g_{i_1},\dots,g_{i_j})$ as above. However, it is more convenient to show this more general statement. 

Assume first that $-1\in N(E/F)$. Then there is a quadratic character $\omega$ of $E^*$ that extends $\eta_{E/F}$ (see Remark  \ref{rem extension is a prasad character}). 
Now denoting by $\sn_E$  the spinor norm on $\SO_{2m+1}(E)$ the character $\omega\circ \sn_E$ of $\SO_{2m+1}(E)$ is well defined and extends $\eta_{E/F}\circ \sn$. Note further that $\omega$ is trivial on $(E/F)_1$. Indeed, $\omega$ is trivial on both $N(E/F)$ and $E^{*2}$ and by Hilbert 90 the group $(E/F)_1$ is contained in the product $N(E/F)E^{*2}$ ($z^{\sigma-1}=N_{E/F}(z)z^{-2}$, $z\in E^*$). 
Therefore,
\[
\eta_{E/F}(\sn(\eta(y)^{-1}\inj(g)\eta(y))=\omega(\sn_E(\eta(y)^{-1}\inj(g)\eta(y))=\omega(\sn_E(\inj(g))).
\]
By \eqref{eq sn siegel} this equals $\omega(\det g)$ that equals $1$ since $\det g\in (E/F)_1$.

Assume now that $-1\not \in N(E/F)$. Note that if $\eta(y)\cdot I_{2m+1}=x(y)$ and $h\in \GL_m(E)$ then $(\inj(h)\eta(y))\cdot I_{2m+1}=x({}^t h^{-\sigma}yh^{-1})=x(y \star h^{-1})$ and $U(y\star h^{-1},E/F)=h U(y,E/F)h^{-1}$. With our assumption it therefore suffices to show \eqref{eq ucase} for $y=\diag(a,I_{m-1})$ with $a=\pm 1$ (see \S \ref{sec prel}). Fix the sign $a$ and let $x=x(y)$ and $\eta=\eta(y)$. Let $e_1,\dots,e_{2m+1}$ be the standard basis of $F^{2m+1}$ and $v_i=\eta^{-1}e_i$, $i=1,\dots,2m+1$. Note that $x^2=I_{2m+1}$ and $xe_1=e_{2m+1}$ except when $a=-1$ and $m=1$ where $xe_1=-e_{2m+1}$. Consequently, 
\[
v_1^\sigma=\eta^{-\sigma} e_1=\eta^{-1} xe_1=\begin{cases} -v_{2m+1} & a=-1 \text{ and }  m=1 \\ v_{2m+1} & \text{otherwise.}\end{cases}
\] 

It follows from Lemma \ref{lem in eorb} (or by direct computation) that for $m=2$ and $a=1$ we have $x(y)=w_5\not\in \SO_5(F)\cdot I_5$. We may therefore exclude this case.  In all other cases, we claim that $[\U(y,E/F),\U(y,E/F)]=\SU(y,E/F)$. Indeed, for  $m=1$ this is clear and for all other cases (excluding $m=2$ and $a=1$) $y$ admits an isotropic vector. It therefore follows from \cite[Th\'{e}or\`{e}me 5]{MR0024439} that every proper normal subgroup of $\SU(y,E/F)$ is central. It is easy to see that $[\U(y,E/F),\U(y,E/F)]$ is not central. Consequently, there is a homomorphism $\chi:(E/F)_1 \rightarrow F^*/F^{*2}$ such that  $\sn\circ \Ad(\eta^{-1})\circ \inj=\chi\circ \det$. It thereforefore suffices to show \eqref{eq ucase} for $g=\diag(u,I_{m-1})\in U(y,E/F)$ with $u\in (E/F)_1$. Note that for such $g$ as in \eqref{eq so2} we have
\[
\inj(g)=s_{ue_1+e_{2m+1}}s_{e_1+e_{2m+1}}.
\]
Applying Hilbert 90, let $z\in E^*$ be such that $z^{\sigma-1}=u$. Then $ue_1+e_{2m+1}$ and $\sigma(z)e_1+ze_{2m+1}$ are co-linear in $E^{2m+1}$. Furthermore we have
\[
(v_1+v_{2m+1})^\sigma=\begin{cases} -(v_1+v_{2m+1}) & a=-1 \text{ and }  m=1 \\ v_1+v_{2m+1} & \text{otherwise.}\end{cases} 
\]
and similarly
\[
(\sigma(z)v_1+zv_{2m+1})^\sigma=\begin{cases} -(\sigma(z)v_1+zv_{2m+1}) & a=-1 \text{ and }  m=1 \\ \sigma(z)v_1+zv_{2m+1} & \text{otherwise.}\end{cases} 
\]
It follows that both $w=\alpha(v_1+v_{2m+1})$ and $w'=\alpha(\sigma(z)v_1+zv_{2m+1})$ are in $F^{2m+1}$ where
\[
\alpha=\begin{cases} \imath & a=-1 \text{ and }  m=1 \\ 1 & \text{otherwise.}\end{cases} 
\]
We deduce that
\[
\eta^{-1}\inj(g)\eta=\eta^{-1}s_{\sigma(z)e_1+ze_{2m+1}}s_{e_1+e_{2m+1}}\eta=s_{\sigma(z)v_1+zv_{2m+1}}s_{v_1+v_{2m+1}}=s_{w'}s_w
\]
and since $q=q\circ \Ad(\eta^{-1})$ we have
\[
\sn(\eta^{-1}\inj(g)\eta)=q(w)q(w')F^{*2}=q(\alpha(\sigma(z)e_1+ze_{2m+1}))q(\alpha(e_1+e_{2m+1}))F^{*2}=4\alpha^4 N_{E/F}(z)F^{*2}.
\]
Since $4\alpha^4 \in F^{*2}$, we obtain
\[
\eta_{E/F}(\sn(\eta^{-1}\inj(g)\eta))=\eta_{E/F}(N_{E/F}(z))=1
\]
as required.
\end{proof}

For a certain class of generic irreducible representations $\pi$ of $G^\circ$ we are now in a position to prove unconditionally the relation between distinction of $\pi$ and of its transfer $T(\pi)$ that Proposition \ref{prop distT} shows follows more generally from Conjecture \ref{conjecture Prasad}. In particular, $\pi$ could be any irreducible principal series. 

\begin{theorem}\label{thm dist and trans}
In the notation of \S \ref{ss trans ind} suppose that $\pi_i$ is cuspidal for $i=0,\dots,k$ (in particular, $\pi_0$ is the trivial character of the trivial group except in the odd unitary or quasi-split but non-split even orthogonal cases where $\pi_0$ is a character of $G_0^\circ$) and $\pi$ is irreducible. (In particular, $\pi$ could be any irreducible generic principal series representation of $G^\circ$.)
\begin{enumerate}
\item If $\pi$ is $H^\circ$-distinguished then its transfer $T(\pi)$ is $\GL_{m(n)}(F')$-distinguished.
\item If  $\pi$ is $(H^\circ,\omega_{\mathbb{H}^\circ,E/F})$-distinguished then its transfer $T(\pi)$ is $(\GL_{m(n)}(F'),\chi)$-distinguished where $\chi$ is the trivial character in the even orthogonal case and equals $\omega_{\GL_{m(n)},E'/F'}$ otherwise.
\end{enumerate}
\end{theorem}
\begin{proof}
Since $\pi$ is irreducible, applying an appropriate Weyl element we assume without loss of generality that $e(\pi_1)\ge \cdots \ge e(\pi_k)\ge 0$. 
It follows from Lemma \ref{lemma transfer of generic is irreducible} that
\[
T(\pi)=\pi_1\times \cdots \times\pi_k\times T(\pi_0) \times (\pi_k^\vee)^\tau \times \cdots \times (\pi_1^\vee)^\tau.
\]

If $\pi$ is $H^\circ$-distinguished then it follows from Theorem \ref{thm dist ind class} and (in the odd unitary or quasi-split but non split even orthogonal case) Proposition \ref{prop T character of SO(2)} that $T(\pi)$ satisfies the conditions for Lemma \ref{lem gldist} and we deduce that $T(\pi)$ is $\GL_{m(n)}(F')$-distinguished.
Note that we only applied the necessary condition for distinction of Theorem \ref{thm dist ind class}. 

The second part of the theorem is only different from the first in the odd orthogonal and in the even unitary cases. Assume we are in one of those cases and note that $G_0^\circ$ is now the trivial group.
Suppose that $\pi$ is $(H^\circ,\omega_{\mathbb{H}^\circ,E/F})$-distinguished. It follows from \cite[Corollary 5.2 and Corollary 6.9]{MR3541705}, the orbits and stabilizers analysis done in Section \ref{subsec orb and stab}, Lemma \ref{lem main charcomp} and the fact that $\upsilon^*=(\upsilon^\tau)^\vee$ for an irreducible representation $\upsilon$ of $\GL_m(E')$ (see \cite{MR0404534}), that there exist an involution $\rho\in S_k$ and a subset $\set$ of $[1,k]$ such that $\rho(\set)=\set$ and $n_{\rho(i)}=n_i$, $i\in [1,k]$ and $y_i\in Y_{n_i}(E'/(E')^{\sigma\tau},1)$ for all $i\in \set$ such that $\rho(i)=i$ such that 
\[
\begin{array}{ll}
\pi_{\rho(i)}\simeq\bar \pi_i^\vee & i\not\in \set,\ \rho(i)\ne i \\
\pi_i \text{ is }(\GL_{n_i}(F'),\eta_{E'/F'}\circ\det)\text{-distinguished} & i\not\in\set,\ \rho(i)=i \\
\pi_{\rho(i)}\simeq \pi_i^{\tau\sigma} & i\in  \set, \ \rho(i)\ne i \\
\pi_i \text{ is }\U(y_i,E'/(E')^{\sigma\tau})-\text{distinguished} & i\in \set,\ \rho(i)=i. 
\end{array}
\]
It now follows from Lemma \ref{lem gldist} that $T(\pi)$ is $(\GL_{m(n)}(F'),\eta_{E'/F'}\circ \det)$-distinguished. Since in both cases $m(n)$ is even the theorem follows.
\end{proof}
\begin{remark}
We expect Theorem \ref{thm dist and trans} to hold for any generic irreducible $\pi$ which is cuspidally induced. Our methods can reduce this question to the case where $\pi$ is cuspidal. However,  the cuspidal case requires other methods and we do not address it here. Applying harmonic analysis on Galois symmetric spaces, this problem is addressed in work in preparation of Beuzart-Plessis and Wan.
\end{remark}
It is now tempting to suggest that a global analogue of this problem also holds. 
Let $K/k$ be a quadratic extension of number fields and $\mathbb{H}^\circ$ be a $k$-quasisplit classical group amongst the ones considered in this paper.
Recall from \cite[Theorem 1.1]{MR2767514} that for every globally generic cuspidal automorphic representation $\pi$ of $\mathbb{H}^\circ(\A_K)$ the functorial lift $T(\pi)$ exists and is an automorphic representations of $\GL_{m(n)}(\A_K)$. Furthermore, the definition of Prasad's character $\omega_{\mathbb{H}^\circ,K/k}$ makes sense globally (see \ref{subsec Prasad global char} in Appendix \ref{app}).

\begin{conjecture}
Let $\pi$ be an irreducible globally generic cuspidal automorphic representation of $\mathbb{H}^\circ(\A_K)$ such that 
$T(\pi)$ is a cuspidal automorphic representation of $\GL_{m(n)}(\A_K)$. 
\begin{enumerate}
\item If $\pi$ is $\mathbb{H}^{\circ}(\A_k)$-distinguished, then $T(\pi)$ is $\GL_{m(n)}(\A_k)$-distinguished. 
\item If $\pi$ is $(\mathbb{H}^{\circ}(\A_k),\omega_{\mathbb{H}^\circ,K/k})$-distinguished, then $T(\pi)$ is $(\GL_{m(n)}(\A_k),\chi_{\GL_{m(n),K/k}})$-distinguished where $\chi_{\GL_{m(n),K/k}}$ is the trivial character in the even orthogonal case and equals $\omega_{\GL_{m(n)},K/k}$ otherwise.
\end{enumerate}
\end{conjecture}

\appendix
\section{Explicit computations of Prasad's quadratic character and opposition group}\label{app}

\subsection{Prasad's quadratic character}\label{subsec prasad char}

Let $\Y$ be a connected reductive group defined and quasi-split over $F$. We recall the construction of a quadratic character $\omega_{\Y,E/F}$ of $Y=\Y(F)$ in \cite[Proposition 6.4]{DP} associated to $\Y$ and the quadratic extension $E/F$.

The character $\omega_{\Y,E/F}$ factors through the natural projection $Y\mapsto Y^{\ad}=\Y^{\ad}(F)$ to the adjoint group and it suffices to define it for adjoint groups.

Denote by $\Y^{\sc}$ the simply connected cover of $\Y^{\ad}$, by $\mathbf{Z}$ its center and by $p:\Y^{\sc}\rightarrow \Y^{\ad}$ the natural projection. 
Then we have an exact sequence of algebraic $F$-groups
\begin{equation}\label{eq defining P char} 
1\rightarrow \mathbf{Z} \rightarrow \Y^{\sc} \rightarrow \Y^{\ad} \rightarrow 1 
\end{equation} 
that gives rise to an exact sequence of local groups
\[
1\rightarrow \mathbf{Z}(F) \rightarrow \Y^{\sc}(F) \mathop{\longrightarrow}\limits^p \Y^{\ad}(F) \mathop{\longrightarrow}\limits^{\delta_{p,F}} H^1(F,\mathbf{Z}).
\] 
(The Galois cohomology group $H^1(F,\mathbf{Z})$ and the connecting homomorphism $\delta_{p,F}$ are defined for example in \cite{MR1324577}.)

Let $\mathbf{B}^{\ad}=\mathbf{T}^{\ad}\ltimes \mathbf{U}$ be a Borel subgroup of $\Y^{\ad}$ with a maximal torus $\mathbf{T}^{\ad}$ and unipotent radical $\mathbf{U}$ definded over $F$. Then $p^{-1}(\mathbf{B}^{\ad})$ is a Borel subgroup of $\Y^{\sc}$ and $\mathbf{Z}\subseteq \mathbf{T}^{\sc}=p^{-1}(\mathbf{T}^{\ad})$. Thus $p$ identifies $\mathbf{U}$ with its pre-image in $\Y^{\sc}$ that we still denote by $\mathbf{U}$. It is known that the half sum (with multiplicities) $\rho$ of the positive roots of $\mathbf{T}^{\sc}$ (with respect to $\mathbf{U}$) is a weight (i.e. an algebraic character) of $\mathbf{T}^{\sc}$ defined over $F$ and that $2\rho$ is a weight of $\mathbf{T}^{\ad}$, that is, contains $\mathbf{Z}$ in its kernel. By restriction $\rho$ induces a character $\rho_{|\mathbf{Z}}\rightarrow \{\pm 1\}$. 
By the functorial properties of Galois cohomology $\rho$ gives rise to a homomorphism
\[
[\rho]_F: H^1(F,\mathbf{Z})\rightarrow H^1(F,\{\pm 1\})\simeq F^*/F^{*2}.
\]
Let $\eta_{E/F}$ be the quadratic character of $F^*$ with kernel $(E/F)_1$, the norm one subgroup. Then
\[
\omega_{\Y^{\ad},E/F}=\eta_{E/F}\circ [\rho]_F\circ \delta_{p,F}:Y^{\ad}\rightarrow \{\pm 1\}
\]
is the character defined by Prasad for $Y^{\ad}$.
Composing with the natural map from $Y$ to $Y^{\ad}$ we obtain Prasad's character 
$\omega_{\Y,E/F}: Y \rightarrow \{\pm{1}\}$.

\begin{remark}\label{rmk ss}
For semi-simple $\Y$ the character $\omega_{\Y,E/F}$ can be defined more directly since the projection $p: \Y^{\sc}\rightarrow \Y^{\ad}$ factors through $\Y$. Let $p':\Y^{\sc} \rightarrow \Y$ be the corresponding projection (so that $p$ is the composition of $p'$ with the narural projection $\Y \rightarrow \Y^{\ad}$). Then the short exact sequence
\[
1\rightarrow \mathbf{Z'} \rightarrow \Y^{\sc} \mathop{\longrightarrow}\limits^{p'} \Y \rightarrow 1
\]
gives rise to the connecting homomorphism $\delta_{p',F}: \Y(F)\rightarrow H^1(F,\mathbf{Z'})$. The kernel $\mathbf{Z'}$ of $p'$ is a subgroup of $\mathbf{Z}$ so that $\rho|_{\mathbf{Z'}}:\mathbf{Z'}\rightarrow \{\pm 1\}$ defines the map $[\rho]_F: H^1(F,\mathbf{Z'}) \rightarrow H^1(F,\{\pm 1\})$ and we have
\[
\omega_{\Y,E/F}=\eta_{E/F}\circ [\rho]_F\circ \delta_{p',F}.
\]
\end{remark}

Taking both $F$-points and $E$-points in Equation (\ref{eq defining P char}) and composing with $[\rho]_E$ and $[\rho]_F$ respectively we obtain the commutative diagram 
\[
\xymatrix{ 
1 \ar[r] & \mathbf{Z}(E)  \ar[r] & \Y^{\sc}(E)  \ar[r]^{} & \Y^{\ad}(E) \ar[r]^{\d_{p,E}} & H^1(E,\mathbf{Z})\ar[r]^{[\rho]_E} & H^1(E,\{\pm1\}) \ar[r]^{\sim} &  E^*/E^{*2}  \\
1 \ar[r] & \mathbf{Z}(F) \ar[r] \ar[u]^{i_{F,E}} & \Y^{\sc}(F) \ar[r] \ar[u]^{i_{F,E}} & \Y^{\ad}(F) \ar[r]^{\d_{p,F}} \ar[u]^{i_{F,E}}  & H^1(F,\mathbf{Z})\ar[r]^{[\rho]_F } \ar[u]^{[i_{F,E}]}  & H^1(F,\{\pm1\})  \ar[u]^{[i_{F,E}]} \ar[r]^{\sim} &  F^*/F^{*2} \ar[u]^i}
\] 
where the vertical arrows $i_{F,E}$ are the natural embeddings of $F$-points into $E$-points of an algebraic group defined over $F$ and $[i_{F,E}]$ are the maps induced by restriction of cocycles from the absolute Galois group of $F$ to that of $E$.  
We arrive at the following statement.
\begin{lemma}\label{lemma Prasad character extends}We have the following equality of characters of $\Y^{\ad}(F)$
\[
[\rho]_E\circ \d_{p,E}\circ i_{F,E}=[i_{F,E}] \circ [\rho]_F\circ \d_{p,F}.
\] 
In particular, if $\eta_{E/F}$ extends to a quadratic character $\eta$ of $E^*$ then 
\[
\omega_{E,\eta}^{\ad}:=\eta\circ [\rho]_E \circ \d_{p,E}
\] 
is a character of $\Y^{\ad}(E)$ that extends $\omega_{\Y^{\ad},E/F}$ and hence, composition with the natural projection $\Y(E)\rightarrow \Y^{\ad}(E)$ gives rise to a character $\omega_{E,\eta}$ of $\Y(E)$ that extends $\omega_{\Y,E/F}$.
\end{lemma}
\begin{proof}
The lemma is immediate from the observation that the map $[i_{F,E}]:H^1(F,\{\pm{1}\}) \rightarrow H^1(E,\{\pm{1}\})$ identifies with the natural map $i:F^*/F^{*2}\rightarrow E^*/E^{*2}$ induced by the inclusion $F^*\subseteq E^*$.
\end{proof}

\begin{remark}\label{rem extension is a prasad character}
There exists a quadratic character $\eta$ of $E^*$ that extends $\eta_{E/F}$ if and only if $-1\in N(E/F)$.
Indeed, such $\eta$ exists if and only if $\eta_{E/F}$ considered as a character of $F^* /F^{*2}$ is trivial on the kernel of the natural map $aF^{*2}\mapsto aE^{*2}: F^* /F^{*2}\rightarrow E^*/E^{*2}$ and if $a\in E\setminus F$ is such that $a^2\in F$ then this kernel is $F^{*2}\sqcup a^2F^{*2}$ while $-a^2\in N(E/F)$. 
When this is the case, we can choose a quadratic extension $L/E$ such that $\eta_{L/E}$ extends $\eta_{E/F}$. 
Note that in this case we have 
\[
\omega_{E,\eta_{L/E}}=\omega_{\Y_E,L/E}
\]
where $\Y_E$ is the base change of $\Y$ to $E$ and the left hand side is defined in the lemma.
\end{remark}

We compute $\omega_{\Y,E/F}$ when $\Y$ is either $\GL_n$ or the quasi-split classical group $\mathbb{H}^\circ$ as in \S\ref{ss lgp}. The method that we use was explained to us by Beuzart-Plessis and we thank him for his help. We also mention that he obtained a different proof of Prasad's character computation in the case of unitary groups, similar to that of the general linear and odd orthogonal case.

\begin{itemize}
\item Assume that $\Y=\GL_n$ so that $\Y^{ad}=\PGL_n$ and $\Y^{\sc}=\SL_n$. Recall that $\operatorname{End}(\mult)=\Z$ where we identify $n\in \Z$ with the $n$th power endomorphism of the multiplicative group $\mult$. Let $\mu_n$ be the algebraic subgroup of $\mult$ defined by the exact sequence
\[
1 \longrightarrow \mu_n \longrightarrow \mult \mathop{\longrightarrow}\limits^n \mult \longrightarrow 1.
\]
Consider the following two commutative diagrams 
\[
\xymatrix{ 
1 \ar[r] & \mu_n \ar[d]_{n\choose 2} \ar[r] & \mult \ar[d]_{n\choose 2} \ar[r]^n & \mult \ar[d]_{n-1} \ar[r] & 1 \\
1 \ar[r] & \mu_2 \ar[r] & \mult \ar[r]^2 & \mult \ar[r] & 1 }
\]
and 
\[\xymatrix{ 
1 \ar[r]& \mu_n \ar[r] & \SL_n  \ar[r] & \PGL_n \ar[r] & 1 \\
1 \ar[r] & \mu_n \ar@{=}[u] \ar[d]_{-1} \ar[r] & \mult \times \SL_n \ar[u]^{p_2} \ar[d]_{p_1} \ar[r] & \GL_n \ar[u]^s \ar[d]_{\det}\ar[r] & 1 \\
1 \ar[r] & \mu_n \ar[r] & \mult \ar[r]^n & \mult \ar[r] & 1 }\]
where $s:\GL_n\rightarrow \PGL_n$ is the natural projection and $p_i$ the projection to the $i$th coordinate, $i=1,2$. 
Note that $\mathbf{Z}=\mu_n$ is the center of $\SL_n$ and $\rho|_{\mathbf{Z}}={n\choose 2}$. 
By passing to $F$-points and applying the long exact sequence of \cite[Proposition 43]{MR1324577} we, in particular, get from the first diagram the commutative diagram
\[\xymatrix{ 
 F^* \ar[d]_{n-1} \ar[r]^-{\d_n}& H^1(F,\mu_n) \ar[d]_{[{n\choose 2}]} & \\
 F^* \ar[r]^-{\d_2} & H^1(F,\mu_2) \ar[r]^{\sim} &   F^*/(F^*)^2}\]
and from the second diagram the  commutative diagram
\[\xymatrix{ 
 \PGL_n(F) \ar[r]^-{\d} & H^1(F,\mu_n) \\
 \GL_n(F) \ar[u]^s \ar[d]_{\det}\ar[r]^-{\d'} &  H^1(F,\mu_n)\ar@{=}[u] \ar[d]_{[-1]}  \\
F^* \ar[r]^-{\delta_n} &  H^1(F,\mu_n) }\]
where $\delta_n$, $\delta_2$, $\delta$ and $\delta'$ are the corresponding connecting homomorphisms.
Note that 
\[
\omega_{\GL_n,E/F}=\eta_{E/F}\circ [{n \choose 2}]_F \circ \delta\circ s.
\]
Putting the two last diagrams together we obtain 
\[
 [-{n \choose 2}]_F\circ \d\circ s=\delta_2 \circ (n-1) \circ \det.
 \] 
Since $\delta_2$ induces the natural projection $F^* \rightarrow  F^*/F^{*2}$ and $[-1]_F$ the identity map on $F^*/F^{*2}$ we deduce that
\[\omega_{\GL_n,E/F}= \eta_{E/F}^{n-1}\circ \det.\]

\item In the symplectic case we have that $\omega_{\Sp_{2n},E/F}$ is trivial. Indeed, if $\Y=\Sp_{2n}$ then $\Y=\Y^{\ad}=\Y^{\sc}$ and $\mathbf{Z}=\{1\}$. In fact, $\Sp_{2n}(F)$ has no non-trivial cahracters.

\item Consider the quasi-split orthogonal case. Thus, $\Y=\SO(\j[n])$ where in the odd case we may take $\j[n]=w_{2n+1}$ and in the even case, either $\j[n]=w_{2n}$ (the split case) or $\j$ is anisotropic of size two (the quasi-split non-split case). If $n=0$ then $\Y$ is anisotropic, hence $\rho$ is trivial and therefore so is $\omega_{\Y,E/F}$. Assume now that $n\ge 1$. Then $\Y$ is semi-simple and $\Y^{\sc}=\Spin(\j[n])$ is the associated spin group. The short exact sequence
\[
1\rightarrow \mu_2 \rightarrow \Spin(\j[n]) \rightarrow \SO(\j[n])\rightarrow 1
\] 
gives rise to the connecting homomorphism 
\[
\d_{\SO}:\SO(\j[n],F)\rightarrow  H^1(F,\mu_2)\simeq F^*/(F^*)^2.
\] 
By Remark \ref{rmk ss} we have $\omega_{\SO(\j[n]),E/F}=\eta_{E/F}\circ[\rho]_F\circ \delta_{\SO}$.
One verifies that in the even case $\rho$ is a weight of the diagonal maximal split torus while in the odd case it is not. In particular, in the even case $[\rho]_F$ is the trivial map so that $\omega_{\SO(\j[n]),E/F}$ is trivial and in the odd case $[\rho]_F$ is the identity map so that $\omega_{\SO(\j[n]),E/F}=\eta_{E/F} \circ \delta_{\SO}$.

The connecting homomorphism $\delta_{\SO}$ is known to be the spinor norm $\sn$. For convenience we provide the argument that is analogous to the $\GL_n$ computation. Denote by 
$\widetilde{\sn}: \GSpin(\j[n])\rightarrow \mult$ the spinor norm on $\GSpin(\j[n])$, so that if \[p:\GSpin(\j[n])\rightarrow \SO(\j[n])\] is the natural projection with kernel $\mult$, the composed maps 
\[\GSpin(\j[n],F)\overset{p}{\rightarrow} \SO(\j[n])\overset{\sn}{\rightarrow} F^*/F^{*2} \] and 
\[\GSpin(\j[n],F)\overset{\widetilde{\sn}}{\rightarrow} F^* \rightarrow F^*/F^{*2} \] are equal.
Consider the commutative diagram 

\[\xymatrix{ 
1 \ar[r]& \mu_2 \ar[r] & \Spin(\j[n])  \ar[r]  & \SO(\j[n])\ar[r] & 1 \\
1 \ar[r] & \mu_2 \ar@{=}[u] \ar@{=}[d]  \ar[r] & \mult \times \Spin(\j[n]) \ar[u]^{p_2}  \ar[d]_{p_1} \ar[r] & \GSpin(\j[n])\ar[u]^p \ar[d]_{\widetilde{\sn}}\ar[r] & 1 \\
1 \ar[r] & \mu_2 \ar[r] & \mult \ar[r]^2 & \mult \ar[r] & 1. }\]
Taking $F$-points we get the following commutative diagram 
\[\xymatrix{ 
 \SO(\j[n],F) \ar[r]^-{\d_{\SO}} & H^1(F,\mu_2)   \simeq  F^*/(F^*)^2 \\
 \GSpin(\j[n],F) \ar[u]^p \ar[d]_{\widetilde{\sn}}\ar[r]^-{\d_{\GSpin}} &  H^1(F,\mu_2)\ar@{=}[u] \ar@{=}[d] \simeq F^*/(F^*)^2 \ar@{=}[u] \ar@{=}[d] \\
F^* \ar[r]^-{\delta_2} &  H^1(F,\mu_2)  \simeq  F^*/F^{*2}. }\] 
Since $p$ is surjective on $F$-points (see e.g. \cite[Chapter 9, Theorem 3.3]{MR770063}) and $\delta_2$ induces the natural projection from $F^*$ to $F^*/F^{*2}$ we deduce that $\delta_{\SO}=\sn$. 
All together we conclude that for all $n\ge 0$ we have
\[
\omega_{\SO(\j[n]),E/F}=\eta_{E/F}^{n_0}\circ \sn
\] 
where $n_0\in \{0,1,2\}$ is the size of $\j$. 
\item Let $\Y=\U_{n,K/F}=\U_{K/F}(w_n)$ be the $F$-quasi-split unitary group with respect to a quadratic extension $K/F$ and $\tau$ the Galois involution associated to $K/F$. Thus, $\Y^{\sc}=\SU_{n,K/F}$ and $\Y^{\ad}=\PU_{n,K/F}$. Let $m:\U_{n,K/F}\rightarrow \PU_{n,K/F}$ be the natural projection. Its restriction to $\SU_{n,K/F}$ gives rise to the short exact sequence
 \[
 1\rightarrow \mu_n' \rightarrow \SU_{n,K/F} \rightarrow \PU_{n,K/F}\rightarrow 1
 \] 
where the center $\mu_n'$ of $\SU_{n,K/F}$ is its intersection with the center $\Res_{K/F}(\mu_n)$ of $\Res_{K/F}(\SL_n)$. By taking $F$-points we get the connecting homomorphism 
$\delta: \PU_{n,K/F}(F)\rightarrow H^1(F,\mu_n')$ and 
\[
\omega_{\U_{n,K/F},E/F}=\eta_{E/F}\circ [\rho]_F\circ \delta \circ m.
\]
Note that if $n$ is odd then the only homomorphism $\mu_n'\rightarrow \mu_2$ is the trivial one and it follows that $[\rho]_F$ and therefore also $\omega_{\U_{n,K/F},E/F}$ is trivial. 

We assume from now on that $n$ is even. Then one verifies that 
\[
\rho(z)=z^{-\frac n2}, \ \ \ z\in \mu_n'(F).
\]
We now interpret the connecting homomorphism $\delta$. 
Denote by 
\[
\wsn:\U_{n,K/F}(F)=\U_n(K/F)\rightarrow K^*/F^*
\] 
the composition of the determinant 
\[
\det:\U_{n,K/F}(F)\rightarrow \U_{1,K/F}(F)
\] 
with the inverse of the isomorphism 
$z F^*\mapsto zz^{-\tau}$ from $ K^*/F^*$ to $\U_{1,K/F}(F)$. 
(In \cite{MR0104764}, Wall interpreted this map as a spinor norm analogue.)
We show that 
\begin{equation}\label{eq wsn} 
[\rho]_F\circ \d \circ m=N_{K/F}\circ \wsn. 
\end{equation} 

The character $N_{K/F}\circ \wsn$ naturally appears in another manner, which we now describe. Let $\d_K\in K\setminus F$ be such that $\d_K^2\in F$. We denote by $\O_{2n}^q$ the orthogonal group defined over $F$ with respect to the quadratic form $q$ on $F^{2n}$ associated to the bilinear form 
\[
b:(\begin{pmatrix}X_1 \\ X_2 \end{pmatrix}, \begin{pmatrix}Y_1 \\ Y_2 \end{pmatrix})\mapsto \Tr_{K/F}(({}^tX_1-\d_K {}^tX_2)w_n
(Y_1+\d_K Y_2)),\ \ \ X_i,\,Y_i\in F^n,\ i=1,2.
\] Since we assume that $n$ is even, one checks that the quadratic space $(F^{2n},q)$ is a sum of hyperbolic planes hence $\O_{2n}^q$ is $F$-split (for $n$ odd it is quasi-split and splits over $K$). There is now an obvious natural imbedding 
\[
i:\U_{n,K/F}\rightarrow \SO_{2n}^q,
\] 
defined over $F$, and 
by  \cite[Chapter 10, Theorem 1.5]{MR770063} we have
\[N_{K/F}\circ \wsn=\sn \circ i.\] We will make use of this equality to prove equation (\ref{eq wsn}). 
Consider the commutative diagram 

\[\xymatrix{ 
 1  \ar[r] & \mu_n'  \ar[r]^-{i_1} \ar@{=}[d] & \U_{1,K/F}\times \SU_{n,K/F} \ar[r]^-{s_1}\ar[d]_{p_2} & \U_{n,K/F} \ar[d]_m \ar[r] & 1\\
 1  \ar[r] & \mu_n' \ar[r]^{} & \SU_{n,K/F} \ar[r]^{}&   \PU_{n,K/F} \ar[r] & 1}\] where 
 $i_1(\mu)=(\mu^{-1},\mu I_n)$ and $s_1(z,u)=zu$, that implies the following one 
 
 \[
 \xymatrix{ 
  \U_{n,K/F}(F) \ar[d]^m \ar[r]^{\d_{s_1,F}} & H^{1}(F,\mu_n')\ar@{=}[d] \\
  \PU_{n,K/F}(F) \ar[r]^{\d} & H^{1}(F,\mu_n')}
  \] 
  so that 
 \[\d\circ m=\d_{s_1,F}.\]

Let  
\[
\overline{\U_{n,K/F}}^2=\{(z,u)\in \U_{1,K/F}\times \U_{n,K/F}, \ \det(u)=z^2\}
\] and  $s_{2}:\overline{\U_{n,K/F}}^2\rightarrow \U_{n,K/F}$ the projection to the second coordinate
\[s_{2}(z,u)=u.\] The map $i_2:\mu\mapsto (\mu,I_n)$ identifies $\mu_2'=\mu_2$ with the kernel of $s_2$. Let $p_{n,2}:\U_{1,K/F}\times \SU_{n,K/F}\rightarrow \overline{\U_{n,K/F}}^2$ be given by 
\[p_{n,2}(z,u)=(z^{n/2}, zu)=(\rho(z)^{-1},zu).\]
\end{itemize} 

This gives the commutative diagram 
 \[\xymatrix{  1\ar[r] & \mu_n' \ar[r]^-{i_1} \ar[d]_{\rho}&   \U_{1,K/F}\times \SU_{n,K/F} \ar[r]^-{s_1} \ar[d]_{p_{n,2}}& \U_{n,K/F} \ar@{=}[d]\ar[r]& 1\\ 
1 \ar[r] &  \mu_2' \ar[r]^-{i_2} & \overline{\U_{n,K/F}}^2 \ar[r]^-{s_2} & \U_{n,K/F} \ar[r] & 1.}\] 
By taking $F$-points we obtain the commutative diagram 
 \[\xymatrix{ 
  \U_{n,K/F}(F) \ar@{=}[d] \ar[r]^{\d_{s_1,F}} & H^{1}(F,\mu_n')\ar[d]_{[\rho]_F} \\
  \U_{n,K/F}(F) \ar[r]^{\d_{s_2,F}} & H^{1}(F,\mu_2')}\] so that 
  \[\d_{s_2,F}=[\rho]_F\circ \d_{s_1,F}.\]

To prove (\ref{eq wsn}) it remains to show that 
\begin{equation}\label{eq dsni}
\d_{s_2,F}=\sn \circ i.
\end{equation} 
To this end we recall the exact sequence 
\[
1\rightarrow \mu_2 \overset{u}{\rightarrow} \Spin_{2n}^q \overset{v}{\rightarrow} \SO_{2n}^q \rightarrow 1
\] 
and show that there exists a map
\[
j:\overline{\U_{n,K/F}}^2\rightarrow \Spin_{2n}^q
\] 
defined over $F$ such that the following diagram is commutative
 \[
 \xymatrix{ 1 \ar[r] & \mu_2' \ar@{=}[d] \ar[r]^{i_2} &  \overline{\U_{n,F'/F}}^2 \ar[d]^j \ar[r]^{s_2} & \U_{n,F'/F} 
 \ar[d]^i \ar[r] & 1 \\  
 1 \ar[r] & \mu_2  \ar[r]^u & \Spin_{2n} \ar[r]^v & \SO_{2n}  \ar[r] & 1.}
 \] 
 Taking $F$-points and recalling from the orthogonal case that the spinor norm $\sn$ is the connecting homomorphism associated with the bottom row we arrive at \eqref{eq dsni} and complete our proof.
 
 To show that such $j$ exists we first make the following simple observation. There exists at most one map $j:\overline{\U_{n,K/F}}^2\rightarrow \Spin_{2n}^q$ such that $v\circ j=i\circ s_2$. Indeed, if $j'$ is another such map then $u\mapsto j'(u)j(u)^{-1}: \overline{\U_{n,K/F}}^2 \rightarrow \mu_2$ must be trivial since $\overline{\U_{n,K/F}}^2$ is connected. It is therefore enough to find such a $j$ defined over $K$. Indeed, if the above diagram commutes for $j$ it also does for $j^\tau$. 
Over $K$ we may identify
\[
\U_{n,K/F}\simeq \GL_n,
\] 
and
\[\overline{\U_{n,K/F}}^2 \simeq \{(z,g)\in \mult\times \GL_n, \det(g)=z^2\}=:\overline{\GL_n}^2\]
 and with these identifications $s_2(z,g)=g$ and $i:\GL_n\rightarrow \SO_{2n}^q$ is an imbedding defined over $K$. 
It now suffices to show that there exists $j:\overline{\GL_n}^2 \rightarrow \Spin_{2n}^q$ defined over $K$ such that $v\circ j=s_2\circ i$ and $j(-1,I_n)=u(-1)$.
In fact, we claim that if there exists $j$ such that $v\circ j=s_2\circ i$ then $j(-1,I_n)=u(-1)$ is automatic. Indeed, otherwise $j$ defines a map $\bar j:\GL_n\rightarrow \Spin_{2n}^q$ such that $v\circ \bar j=i$. Note that $\GL_n$, $\Spin_{2n}^q$ and $\SO_{2n}^q$ are all of rank $n$. Since $i$ is injective it maps any maximal torus of $\GL_n$ to a maximal torus $\T$ of $\SO_{2n}^q$. Furthermore, $\bar j$ must also be injective and therefore $\bar j(\T)$ is a maximal torus of $\Spin_{2n}^q$. Clearly $\bar j(\T)$ is a subgroup of $v^{-1}(\T)$.
It is a well known fact on simply connected covers that $v^{-1}(\T)$ is a maximal torus of $\Spin_{2n}^q$ and therefore $v^{-1}(\T)=\bar j(\T)$. Since $u(\mu_2)\subseteq v^{-1}(\T)$ is in the kernel of $v$ this contradicts the injectivity of $i$.  

 Note that we may further identify $\overline{\GL_n}^2$ as the semidirect product $\mult \ltimes \SL_n$ where the action of $\mult$ on $\SL_n$ is given by
 \[
 zgz^{-1}=\begin{pmatrix} z^2 & \\ & I_{n-1} \end{pmatrix} g\begin{pmatrix} z^{-2} & \\ & I_{n-1} \end{pmatrix}, \ \ \ z\in\mult,\ g\in \SL_n.
 \] 
 Indeed $(z,g)\mapsto (z,\diag(z^2,I_{n-1}) g): \mult \ltimes \SL_n \rightarrow \overline{\GL_n}^2$ is an isomorphism. 
 It now suffices to show that there is a map $j: \mult \ltimes \SL_n \rightarrow \Spin_{2n}^q$ defined over $K$ such that 
 \[
 j(z,g)=i(\diag(z^2,I_{n-1})g).
 \]
Let $j_1=i|_{\SL_n}:\SL_n\rightarrow \SO_{2n}^q$ and $j_2:\mult \rightarrow \SO_{2n}^q$ be given by $j_2(z)=i(\diag(z^2,I_{n-1}))$. 
We claim that it suffices to lift $j_1$ and $j_2$ to $\Spin_{2n}^q$. Indeed, if $\bar j_1:\SL_n\rightarrow \Spin_{2n}^q$ and $\bar j_2:\mult\rightarrow \Spin_{2n}^q$ 
are $K$-morphisms such that $v\circ \bar j_t=j_t$, $t=1,2$ then we claim that 
\[
\j(z,g)=\bar j_2(z)\bar j_1(g)
\]
is automatically a $K$-morphism. Indeed, we need to verify that 
\[
\bar j_2(z)\bar j_1(g)\bar j_2(z)^{-1}=\bar j_2 \begin{pmatrix} z^2 & \\ & I_{n-1} \end{pmatrix} g\begin{pmatrix} z^{-2} & \\ & I_{n-1} \end{pmatrix}.
\]
Fix $z$ and let $f,\,f':\SL_n \rightarrow \Spin_{2n}^q$ be defined by the left and right hand sides of the above equation respectively. Then $v\circ f=v\circ f'$ and therefore $g\mapsto f'(g)f(g)^{-1}:\SL_n\rightarrow \mu_2$ must be trivial as $\SL_n$ is connected. Thus, $f'=f$ as required.

To complete the proof we observe that $j_1$ and $j_2$ both lift to $\Spin_{2n}^q$. The map $j_1$ has a lifting since $\SL_n$ is simply connected. Indeed, let $\Y$ be the image of $j_1$ in $\SO_{2n}^q$ and $\Y'=v^{-1}(\Y)^\circ$ (the connected component of the pre-image of $\Y$ in $\Spin_{2n}^q$). Then $j_1$ is an isomorphism of $\SL_n$ with $\Y$ and $v:\Y'\rightarrow \Y$ is an isogeny (with kernel contained in $\mu_2$). Since $\SL_n$ and therefore also $\Y$ is simply connected, $v|_{\Y'}$ is in fact an isomorphism and therefore $\bar j_1=(v|_{\Y'})^{-1} \circ j_1:\SL_n \rightarrow \Spin_{2n}^q$ is such that $j_1=v\circ \bar j_1$.
Let $\mathbf{Z}\simeq \mult$ be the image of $j_2$ in $\SO_{2n}^q$ and $\mathbf{Z'}=v^{-1}(\mathbf{Z})^\circ\simeq \mult$ be the connected component of its pre-image in $\Spin_{2n}^q$. If $\mathbf{Z'}$ contains the kernel $u(\mu_2)$ of $v$ then $v_{|\mathbf{Z'}}=w^2$ is the square of an isomorphism $w:\mathbf{Z'}\rightarrow \mathbf{Z}$ (any endomorphism of $\mult$ with kernel $\mu_2$ is either the square map $2$ or its composition with the inverse map). We can then define the desired lifting $\bar j_2(z)=w^{-1}(i(\diag(z,I_{n-1})))$ such that $v\circ \bar j_2=j_2$. Otherwise, $v|_{\mathbf{Z'}}:\mathbf{Z'}\rightarrow\mathbf{Z}$ is an isomorphism and we can define $\bar j_2= (v|_{\mathbf{Z'}})^{-1} \circ j_2$. This completes our computation. (In fact, the second case cannot occur, it would result in a map $\bar j:\GL_n\rightarrow \Spin_{2n}^q$ such that $v\circ \bar j=i$ and as observed before this contradicts the fact that $i$ is injective.)

All together we proved the formula \[\omega_{\U_{n,K/F}, E/F}= \eta_{E/F}^{n-1} \circ N_{K/F} \circ \wsn.\]
\begin{remark}
In particular, $\omega_{\U_{n,E/F}, E/F}$ is always the trivial character since $\eta_{E/F}\circ N_{E/F}$ is the trivial character of $E^*$. For $K\ne E$ we have 
\[
\omega_{\U_{n,K/F}, E/F}=\eta_{EK/K}^{n-1} \circ \wsn.
\]
Indeed, in order to see that $\eta_{EK/K}=\eta_{E/F}\circ N_{K/F}$ we observe that on the one hand the right hand side is clearly trivial on $N(EK/K)$ and on the other hand it is not the trivial character since by local class field theory $N(K/F)$ and $N(E/F)$ are two different subgroups of $F^*$ of index two. Note further that $\eta_{EK/K}\circ \wsn$ is well defined since $F^*\subseteq N(EK/K)$ by Lemma \ref{lem klein ext}.
\end{remark}

We describe another interpretation of $\omega_{\Y,E/F}$ provided in \cite{DP} (see also \cite[\S 5.2]{MR3858402}). If $\Y$ is quasi-split then according to \cite{LL} the Langlands morphism 
\[
\alpha_{\Y}:H^1(W_F,Z(\widehat{\Y}(\C))) \rightarrow \Hom(Y,\C^\times)
\] 
is bijective. The class
\[
\chi_{\Y,E/F}:=\alpha_{Y}^{-1}(\omega_{\Y,E/F}):W_F\rightarrow Z(\widehat{\Y}(\C))
\] 
is represented by the unique cocycle $\chi_{\Y,E/F}$ that maps $W_E$ to the identity and the non-trivial class in $W_F/W_E$ to the image of $-I_2$ by any principal $\SL_2$-morphism from $\SL_2(\C)$ to $\widehat{\Y}(\C)$. In fact, $\chi_{\Y,E/F}$ is a quadratic character of $W_F$ and in the sequel we view $\chi_{\Y,E/F}$ as a character of 
$W_F'$ trivial on $\SL_2(\C)$ so that it makes sense to twist an $L$-homomorphism of $\Y$ by $\chi_{\Y,E/F}$. The following lemma follows easily. 

\begin{lemma}\label{lemma transfer and character twist}
Suppose that $-1\in N(E/F)$ and let $L/E$ be as in Remark \ref{rem extension is a prasad character}, so that $L'=LE'$ is a quadratic extension of $E'$. Let $\Y$ be either $\SO_{2n+1}$ or $U_{2n,E/F}$ and let $N$ be either $2n+1$ or $2n$ respectively. For $[\phi]\in \Phi(\Y(E))$ we have 
\[
I([\chi_{\Y_E,L/E} \,\phi])=[\chi_{\GL_N,L'/E'}\,\phi'] \ \ \ \text{where} \ \ \ [\phi']=I([\phi]).
\]
\end{lemma}
\begin{proof}
In both cases, the principal $\SL_2$-morphism from $\SL_2(\C)$ to $\widehat{\Y}(\C)$ is the irreducible $2n$-dimensional representation
(it has image in $\Sp_{2n}(\C)$ according to \cite[Lemma 3.2.15]{MR2522486}).

If $\Y=\SO_{2n+1}$ then $E'=E$ and $L'=L$, and $\chi_{\Y_E,L/E}$ identifies with the character of $\Gal(L/E)$ sending the non-trivial Galois element $\sigma_{L/E}$ to 
$-I_{2n}$ in the center of $\Sp_{2n}(\C)$, whereas $\chi_{\GL_N, L/E}$ identifies with the character of $\Gal(L/E)$ sending $\sigma_{L/E}$ to 
$-I_{2n}$ in the center of $\GL_{2n}(\C)$. As $I$ is induced by composing L-homomorphisms with the natural inclusion of dual groups the result follows. 

If $\Y=\U_{2n,E/F}$ then $I[\phi]=[\phi_{|W_{E'}}]$. Now $\chi_{\Y_E,L/E}$ identifies with the character of $\Gal(L/E)$ sending 
$\sigma_{L/E}$ to 
$-I_{2n}$ in the center of $\GL_{2n}(\C)$, whereas $\chi_{\GL_{2n},L'/E'}$ identifies with the character of $\Gal(L'/E')$ sending 
$\sigma_{L'/E'}$ to 
$-I_{2n}$ in the center of $\GL_{2n}(\C)$, and the result follows. 
\end{proof}

\subsection{Prasad's global character}\label{subsec Prasad global char} Let $K/k$ be a quadratic extension of number fields, and let $\eta_{K/k}$ be the quadratic automorphic character of $\A_k^\times$ attached to $K/k$ by global class field theory. It decomposes as $\eta_{K/k}=\prod_v \eta_{K_v/k_v}$ where the product is over all places $v$ of $k$ and $\eta_{K_v/k_v}$ is the trivial character if $v$ splits in $K$.

Denote by $\mathbf{Y}$ a $k$-quasi-split group. Thanks to Section \ref{subsec prasad char} for each finite place $v$ of $k$ the natural map  
$\Y(k_v) \mathop{\longrightarrow}\limits^{p_v} \Y^{\ad}(k_v)$ gives rise to the local homomorphism 
\[N_v:=[\rho]_{k_v}\circ \delta_{p,k_v}\circ p_v: \Y(k_v)\rightarrow k_v^*/k_v^{*2},\] and with the obvious notation
the exact sequence (\ref{eq defining P char}) also gives rise to the homomorphism 
\[N_k:=[\rho]_{k}\circ \delta_{p,k}\circ p_k: \Y(k)\rightarrow k^*/k^{*2}\] that satisfies 
$j_{k_v} \circ N_k= N_v \circ i_{k_v}$ where $i_{k_v}:\Y(k)\rightarrow \Y(k_v)$ is the natural inclusion and 
$j_{k_v}: k^*/k^{*2}\rightarrow k_v^*/{k_v^*}^2$ the natural homomorphism. This also holds by the same arguments for archimedean places. 
The local homomorphisms $N_v$ allows one to define the global homomorphism 
\[N=\prod_v N_v:\Y(\A_k)\rightarrow \A_k^\times/{\A_k^\times}^2.
\] 
Note that $N_{|\Y(k)}=N_k$. We set 
\[\omega_{\Y, K/k}:=\eta_{K/k}\circ N.\] It is a smooth character of $\Y(\A_k)$, which is in fact automorphic since 
 $\eta_{K/k}$ is and $N_{|\Y(k)}$ equals $N_k$ composed with the natural diagonal map from 
 $k^*/k^{*2}$ to $ \A_k^\times/\A_k^{\times 2}$. 
 Prasad's global character is related to Prasad's local characters by the following formula:
\[\omega_{\Y, K/k}=\prod_v \omega_{\Y, K_v/k_v}\]
where $\omega_{\Y, K_v/k_v}$ is the trivial character of $\Y(k_v)$ if $v$ splits in $K$.
(The local characters at archimedean places are similarly defined in \cite{DP}.)

\subsection{The opposition group}

Following \cite[\S 5]{DP} we define the opposition group $\Y^\op$ for any $F$-quasi-split reductive group $\mathbf{Y}$. 
We recall that $\mathbf{Y}$ is a form of a unique $F$-split group $\mathbf{Y_s}$, and that any quasi-split form of $\mathbf{Y_s}$
corresponds to a homomorphism in $\Hom(\Gal(\bar F/F),\Out(\mathbf{Y_s}))$. 
Denote by $\mu_\Y:\Gal(\bar F/F) \rightarrow \Out(\mathbf{Y_s})$ the homomorphism corresponding to $\Y$.

The group $\Out(\Y_s)$ of outer automorphisms of $\Y_s$ identifies with the group of pinned automorphisms of $\mathbf{Y_s}$ associated with a quadruple $(\mathbf{Y_s},\mathbf{B},\mathbf{T},(\mathbf{X_{\alpha}})_{\alpha\in \Delta})$
where $\mathbf{T}$ is a maximal torus inside a Borel subgroup $\mathbf{B}$ of $\mathbf{Y_s}$, $\Delta$ is a basis of simple roots of $\mathbf{Y_s}$ with respect to $(\mathbf{B},\mathbf{T})$ and $(\mathbf{X_{\alpha}})_{\alpha\in \Delta}$ is a pinning, all defined over $F$ (see for example \cite{MR546608}). 
Denote by $w_0$ the longest element in the Weyl group of $\mathbf{Y_s}$ associated to $\mathbf{T}$. Then $\Aut(\mathbf{H_s},\mathbf{B},\mathbf{T},(\mathbf{X_{\alpha}})_{\alpha\in \Delta})$ contains a unique element sending $\mathbf{X_{\alpha}}$ to $\mathbf{X}_{-w_0(\alpha)}$, $\alpha\in\Delta$, known as the \textit{Chevalley involution} of $\mathbf{Y_s}$, 
which corresponds to an involution $i_c$ in the center $Z(\Out(\mathbf{Y_s}))$ of $\Out(\mathbf{Y_s})$ independent of the choice of data $(\mathbf{B},\mathbf{T},(\mathbf{X_{\alpha}})_{\alpha\in \Delta})$. There is a unique element 
\[
\mu_{E/F}\in \Hom(\Gal(E/F),Z(\Out(\mathbf{Y_s})))\hookrightarrow \Hom(\Gal(\bar F/F),\Out(\mathbf{Y_s}))
\] 
such that $\mu_{E/F}(\sigma)=i_c$. Then 
\[
\mu_{\Y^\op}:=\mu_{E/F}\mu_{\Y}\in  \Hom(\Gal(\bar F/F),\Out(\mathbf{Y_s}))
\] 
corresponds to a unique $F$-quasi-split form of $\mathbf{Y_s}$ that we denote by $\Y^\op$. Taking $F$-points we call $Y^\op=\Y^\op(F)$ the opposition group of $Y$. Note that $\Y$ and $\Y^\op$ are isomorphic over $E$. It further follows from the definition that $(\Y^\op)^\op=\Y$.

We now compute the opposition groups in the cases of interest to us.

\begin{itemize}
\item If $\Y=\GL_n$ then $\Y=\mathbf{Y_s}$ and $\mu_{\Y}$ is the trivial homomorphism. Furthermore, the Chevalley involution of $\GL_n$ corresponds to $i_c\in Z(\Out(\GL_n))$ the class of the automorphism $g\mapsto w_n {}^t\! g^{-1}w_n^{-1}$. It is now easily verified that  
$\Y^\op=\U_{n,E/F}$. 
\item If $\Y=\Sp_{2n}$ or $\Y=\SO(\j[n])$ (either the quasi-split orthogonal or symplectic case) then the Chevalley involution of $\Y_s$ is trivial and therefore $\Y=\Y^\op$. 

\item If $\Y=\U_{n,F'/F}$ then $\mathbf{Y_s}=\GL_n$ and $\mu_{\Y}$ factors through $\Gal(F'/F)$ and sends $\tau$ to $i_c$. Hence $\mu_{\Y^\op}=\mu_{E/F}\mu_{\Y}$ factors through $\Gal(E'/F)$ and furthermore maps $\sigma\tau$ to $i_c^2$ hence lies in the kernel of $\mu_{\Y^\op}$. Thus, $\mu_{\Y^\op}$ factors through $\Gal(F''/F)$ where $F''=(E')^{\sigma\tau}$, and maps the non-trivial involution of $F''/F$ (the restriction of either $\sigma$ or $\tau$) to $i_c$. That is, 
\[
\Y^\op=\U_{n,F''/F}.
\]

\end{itemize}

\def\cprime{$'$} \def\Dbar{\leavevmode\lower.6ex\hbox to 0pt{\hskip-.23ex
  \accent"16\hss}D} \def\cftil#1{\ifmmode\setbox7\hbox{$\accent"5E#1$}\else
  \setbox7\hbox{\accent"5E#1}\penalty 10000\relax\fi\raise 1\ht7
  \hbox{\lower1.15ex\hbox to 1\wd7{\hss\accent"7E\hss}}\penalty 10000
  \hskip-1\wd7\penalty 10000\box7}
  \def\polhk#1{\setbox0=\hbox{#1}{\ooalign{\hidewidth
  \lower1.5ex\hbox{`}\hidewidth\crcr\unhbox0}}} \def\dbar{\leavevmode\hbox to
  0pt{\hskip.2ex \accent"16\hss}d}
  \def\cfac#1{\ifmmode\setbox7\hbox{$\accent"5E#1$}\else
  \setbox7\hbox{\accent"5E#1}\penalty 10000\relax\fi\raise 1\ht7
  \hbox{\lower1.15ex\hbox to 1\wd7{\hss\accent"13\hss}}\penalty 10000
  \hskip-1\wd7\penalty 10000\box7}
  \def\ocirc#1{\ifmmode\setbox0=\hbox{$#1$}\dimen0=\ht0 \advance\dimen0
  by1pt\rlap{\hbox to\wd0{\hss\raise\dimen0
  \hbox{\hskip.2em$\scriptscriptstyle\circ$}\hss}}#1\else {\accent"17 #1}\fi}
  \def\bud{$''$} \def\cfudot#1{\ifmmode\setbox7\hbox{$\accent"5E#1$}\else
  \setbox7\hbox{\accent"5E#1}\penalty 10000\relax\fi\raise 1\ht7
  \hbox{\raise.1ex\hbox to 1\wd7{\hss.\hss}}\penalty 10000 \hskip-1\wd7\penalty
  10000\box7} \def\lfhook#1{\setbox0=\hbox{#1}{\ooalign{\hidewidth
  \lower1.5ex\hbox{'}\hidewidth\crcr\unhbox0}}}
\providecommand{\bysame}{\leavevmode\hbox to3em{\hrulefill}\thinspace}
\providecommand{\MR}{\relax\ifhmode\unskip\space\fi MR }
\providecommand{\MRhref}[2]{%
  \href{http://www.ams.org/mathscinet-getitem?mr=#1}{#2}
}
\providecommand{\href}[2]{#2}


\end{document}